\documentclass[10pt]{article}

\usepackage[T1]{fontenc}
\usepackage{lmodern}
\usepackage[utf8]{inputenc}
\usepackage{microtype}
\usepackage{framed}
\usepackage{listings}
\usepackage{vmargin}
\usepackage{setspace}
\usepackage{mathrsfs, mathenv}
\usepackage{amsmath, amsthm, amssymb, amsfonts, amscd}
\usepackage{graphicx}
\usepackage{epstopdf}
\usepackage[svgnames]{xcolor}
\usepackage{hyperref}
\hypersetup{citecolor=blue, colorlinks=true, linkcolor=black}
\usepackage[capitalise]{cleveref}
\usepackage{subcaption}
\usepackage{bbm}
\usepackage{cite}
\usepackage{verbatim}
\usepackage{pgfplots}
\usepackage{tikz}
\usetikzlibrary{arrows,decorations.pathmorphing,backgrounds,positioning,fit,matrix}
\usepackage{etoolbox}
\usepackage{color}
\usepackage{lipsum}
\usepackage{ifthen}
\usepackage[ruled, vlined]{algorithm2e}
\usepackage[title]{appendix}
\usepackage{autonum}

\theoremstyle{plain}
\newtheorem{theorem}{Theorem}[section]
\newtheorem{corollary}[theorem]{Corollary}
\newtheorem{lemma}[theorem]{Lemma}
\newtheorem{proposition}[theorem]{Proposition}
\numberwithin{equation}{section}

\theoremstyle{definition}

\theoremstyle{remark}
\newtheorem{remark}[theorem]{Remark}
\newtheorem{assumption}[theorem]{Assumption}
\newtheorem{example}[theorem]{Example}

\ifpdf
  \DeclareGraphicsExtensions{.eps,.pdf,.png,.jpg}
\else
  \DeclareGraphicsExtensions{.eps}
\fi

\usepackage{mathtools}

\usepackage{booktabs}

\usepackage{pgfplots}
\usepackage{tikz}
\usetikzlibrary{patterns,arrows,decorations.pathmorphing,backgrounds,positioning,fit,matrix}
\usepackage[labelfont=bf]{caption}
\usepackage{here}
\usepackage[font=normal]{subcaption}

\usepackage{enumitem}
\setlist[itemize]{leftmargin=.5in}
\setlist[enumerate]{leftmargin=.5in,topsep=3pt,itemsep=3pt,label=(\roman*)}

\newcommand*\samethanks[1][\value{footnote}]{\footnotemark[#1]}

\newcommand{\TheTitle}{Drift Estimation of Multiscale Diffusions  \\ Based on Filtered Data} 
\newcommand{\TheAuthors}{A. Abdulle, G. Garegnani, G. Pavliotis, A. M. Stuart, A. Zanoni}
\title{\TheTitle}
\author{Assyr Abdulle \thanks{Institute of Mathematics, École Polytechnique Fédérale de Lausanne}
		\and Giacomo Garegnani  \samethanks
		\and Grigorios A. Pavliotis \thanks{Department of Mathematics, Imperial College London}
		\and Andrew M. Stuart \thanks{Department of Computing and Mathematical Sciences, Caltech}
		\and Andrea Zanoni \samethanks[1]
}
\date{}

\usepackage{amsopn}

\newcommand{\abs}[1]{\left\lvert#1\right\rvert}
\newcommand{\norm}[1]{\left\|#1\right\|}
\renewcommand{\phi}{\varphi}
\renewcommand{\theta}{\vartheta}

\newcommand{\R}{\mathbb{R}}

\newcommand{\epl}{\varepsilon}

\newcommand{\defeq}{\coloneqq}
\newcommand{\eqdef}{\eqqcolon}

\newcommand{\E}{\operatorname{\mathbb{E}}}

\newcommand{\trace}{\operatorname{tr}}

\renewcommand{\d}{\mathrm{d}}
\newcommand{\dd}{\,\mathrm{d}}
\definecolor{shade}{RGB}{100, 100, 100}
\definecolor{bordeaux}{RGB}{128, 0, 50}

\usepackage[usestackEOL]{stackengine}

\definecolor{leg1}{RGB}{0,114,189}
\definecolor{leg2}{RGB}{217,83,25}
\definecolor{leg3}{RGB}{237,177,32}
\definecolor{leg4}{RGB}{126,47,142}
\definecolor{leg5}{RGB}{119,172,48}

\definecolor{leg21}{RGB}{62,38,169}
\definecolor{leg22}{RGB}{46,135,247}
\definecolor{leg23}{RGB}{55,200,151}
\definecolor{leg24}{RGB}{254,195,56}

\ifpdf
\hypersetup{
	pdftitle={\TheTitle},
	pdfauthor={\TheAuthors}
}
\fi

\begin{document}
\maketitle	

\textbf{Abstract.} We study the problem of drift estimation for two-scale continuous time series. We set ourselves in the framework of overdamped Langevin equations, for which a single-scale surrogate homogenized equation exists. In this setting, estimating the drift coefficient of the homogenized equation requires pre-processing of the data, often in the form of subsampling; this is because the two-scale equation and the homogenized single-scale equation are incompatible at small scales, generating mutually singular measures on the path space. We avoid subsampling and work instead with filtered data, found by application of an appropriate kernel function, and compute maximum likelihood estimators based on the filtered process. We show that the estimators we propose are asymptotically unbiased and demonstrate numerically the advantages of our method with respect to subsampling. Finally, we show how our filtered data methodology can be combined with Bayesian techniques and provide a full uncertainty quantification of the inference procedure.
 


\section{Introduction}

Efficient parameter estimation for stochastic models is essential in a wide range of applications in natural and social sciences. In several areas, the data originate from phenomena which vary continuously in time and which are endowed with a multiscale structure. This is the case, for example, in molecular dynamics, oceanography and atmosphere science or in econometrics. Frequently, it is desirable in these areas to infer from data a simpler model which captures effectively large-scale structures, or slow variations, disregarding small-scale fluctuations or treating them as a source of noise. The mismatch between the data and their desired slow-scale representation is a typical instance of a problem of model misspecification, which, if ignored or handled incorrectly, can lead to erroneous inference. Indeed, the data, coming from the full dynamics, are compatible with the coarse-grained model only at the time scales at which the effective dynamics is valid.

In this paper we consider a simple multiscale setting arising from models of molecular dynamics, with a complete separation between the fast and the slow scale. In particular, we consider diffusion processes for motion in a confining potential which has slow variations with rapid order-one oscillations superimposed. Given data in the form of a sample path from this simple class of model problems, we are interested in determining the drift coefficient of an equation of the overdamped Langevin type in which the fast-scale potential is eliminated. The theory of homogenization guarantees that such a single-scale equation can be uniquely determined, and our goal is therefore to obtain effective coarse-grained dynamics from data consistently with respect to the homogenization result. 

Several methods to take into account model misspecification in multiscale frameworks as above exist. For diffusion processes, the proposed approaches rely in different measures to subsampling, which has proved itself to some extent effective in many applications, but which requires nevertheless precise knowledge of how separated the two characteristic time scales are. Robustness of this methodology is dubious, too, as inference results tend to be extremely sensitive to the subsampling rate. 

In the rest of the introduction, we first give a brief overview of the existing literature on the topic of deterministic and stochastic multiscale inference problems, then introduce our novel methodology and its favourable properties and conclude with an outline of this paper.

\subsection{Literature Review}

For simple models in molecular dynamics, the effect of model misspecification was studied in a series of papers \cite{ABT10,ABJ13,PaS07,PPS09,PPS12,GaS17,GaS18} under the assumption of scale separation. In particular, for Brownian particles moving in two-scale potentials it was shown that, when fitting data from the full dynamics to the homogenized equation, the maximum likelihood estimator (MLE) is asymptotically biased \cite[Theorem 3.4]{PaS07}. To be more precise, in the large sample size limit, the data remains consistent with the multi-scale problem at small scale. Ostensibly this would seem related only to the estimation of the diffusion coefficient. However, because of detail balance, it also has the effect that the MLE, for the drift in a parameter fit of a single-scale model, incorrectly identifies the coefficient of the homogenized equation. The bias of the MLE can be eliminated by subsampling at an appropriate rate, which lies between the two characteristic time scales of the problem \cite[Theorems 3.5 and 3.6]{PaS07}. 

Similar techniques can be employed in econometrics, in particular for the estimation of the integrated stochastic volatility in the presence of market microstructure noise. In this case, too, the data have to be subsampled at an appropriate rate \cite{AMZ05,OSP10}. The correct subsampling rate can, in some instances, be rather extreme with respect to the frequency of the data itself, resulting in ignoring as much as $99\%$ of the time-series. As the intuition suggests, this increases significantly the variance of the estimator, which is usually taken care of with additional bias corrections and variance reduction procedures. The need of such methodology is accentuated by data being obtained at high-frequency \cite{AiJ14,ZMA05}.  

The problem of extracting large-scale variations from multiscale data is studied in atmosphere and ocean science. In this field, too, subsampling the data is necessary to obtain an accurate coarse-grained model \cite{CoP09,YMV19}.

The necessity to subsample the data can be alleviated by using appropriate martingale estimators, as was done in \cite{KPK13,KKP15}. This class of estimators can be applied to the case where the noise is multiplicative and also given by a deterministic chaotic system, as opposed to white noise. Estimators of this family have been applied to time series from paleoclimatic data and marine biology and augmented with appropriate model selection methodologies \cite{KPP15}. 

In case the data consist of discrete observations and not of continuous time series, it is possible to employ estimators based on a spectral decomposition of the generator of the stochastic process. Methodologies of this kind have been applied successfully to inference problems for single-scale problems \cite{CrV06,KeS99}, as well as more recently for multiscale diffusions \cite{CrV11}.

Inference of diffusion processes can be naturally performed under a Bayesian perspective. If one focuses on the drift coefficient, the form of the likelihood function guarantees, under a Gaussian prior hypothesis, that the posterior distribution is itself a Gaussian. The versatility of the Bayesian approach in the infinite-dimensional case \cite{Stu10, DaS16} gives the possibility to extend the study of inferring the drift of a diffusion process to the non-parametric case \cite{PSV09, PSZ13}. 

The issue of model misspecification in inverse problems with a multiscale structure has been treated in the context of partial differential equations, too. In particular, it has been shown that it is possible to infer a coarse-grained equation from data coming from the full model and to retrieve, in the large data limit, the correct result \cite{NPS12}. A series of papers \cite{AbD20, AbD19, AGZ20} focuses on retrieving the full model when the multiscale coefficient is endowed with a specific parametrized structure. Since these problems are ill-posed, the latter is achieved via Tikhonov regularization \cite{AbD19,NPS12}, adopting a Bayesian approach \cite{AbD20, NPS12} or exploiting techniques of Kalman filtering \cite{AGZ20}. In \cite{AbD20,AGZ20}, the authors highlight the need to account explicitly for the modelling error due to homogenization and apply statistical techniques taken from \cite{CDS18,CES14}.

\subsection{Our Contributions}
In this paper, we bypass subsampling by designing a methodology based on filtered data. In particular, we smooth the time-series data from the multiscale model by application of an appropriate linear time-invariant filter, from the exponential family, and show that doing so allows us to accurately retrieve the drift coefficient of the homogenized model. The methodology we present is straightforward to implement, robust in practice and backed by theory. In particular, we show theoretically and demonstrate via numerical experiments that:
\begin{enumerate}
	\item The smoothing width of the filter can be alternatively tuned to be proportional to the speed of the slow process or to smaller scales and provide in both cases unbiased results for maximum likelihood parameter estimation. Sharp estimates on the minimal width with respect to the multiscale parameter are provided. The unbiasedness results are given in \cref{thm:mainTheorem,thm:mainTheorem_zeta} for filtered data in the homogenized and in the multiscale regimes, respectively.
	\item We additionally propose in the multiscale regime an estimator of the effective diffusion coefficient based on filtered data, as shown by \cref{thm:diffusion_unbiasedness}.
	\item Estimations based on our technique are robust in practice with respect to the parameter of the filter. This is not the case for subsampling, which is strongly influenced by the subsampling frequency. The robustness of our technique is demonstrated via numerical experiments in \cref{sec:Num_Param,sec:Num_Multi}.
	\item The entire stream of data is employed, which, in practice, enhances the quality of the filter-based MLE in terms of bias. Moreover, avoiding subsampling and thus discretising the data allows us to employ continuous-time theoretical tools.
	\item It is possible to employ the filtered data approach within a continuous-time Bayesian framework by a careful modification of the likelihood function. Under mild hypotheses on the filter parameters, we are able to show that the posterior distributions obtained with our methodology are asymptotically consistent with respect to the drift parameter of the homogenized equation. Our main theoretical result is given in \cref{thm:Bayesian}, and a numerical experiment for the combination of the filtered data approach and of Bayesian techniques is presented in \cref{sec:Num_Bayes}.
\end{enumerate} 

\subsection{Outline}
The rest of the paper is organised as follows. In Section \ref{sec:Setting} we introduce the problem and lay the basis of our analysis setting the main assumptions and notation. In Section \ref{sec:Filter} we present our filtered data methodology, with a particular focus on ergodic properties, on multiscale convergence and, naturally, on the properties of our estimators. In Section \ref{sec:Bayesian} we introduce the Bayesian framework and show how it can be enhanced employing filtered data. Finally, in Section \ref{sec:NumExp} we demonstrate the effectiveness of our methodology via a series of numerical experiments.

\section{Problem Setting}\label{sec:Setting}

In this section, we introduce the class of diffusion processes which we treat in this paper and the classical methodology employed for the estimation of the drift. Let $\epl > 0$ and let us consider the one-dimensional multiscale stochastic differential equation (SDE)
\begin{equation}\label{eq:SDE_MS}
	\d X_t^\epl = -\alpha \cdot V'(X_t^\epl) \dd t - \frac1\epl p'\left(\frac{X_t^\epl}\epl\right) \dd t + \sqrt{2\sigma} \dd W_t,
\end{equation}
where, given a positive integer $N$, we have that $\alpha \in \R^N$ and $\sigma > 0$ are the drift and diffusion coefficients respectively and $W_t$ is a standard one-dimensional Brownian motion. The functions $V\colon \R \to \R^N$ and $p\colon \R \to \R$ define the slow-scale and the fast-scale confining potentials respectively. In particular, we assume 
\begin{equation}\label{eq:Potential}
	V(x) = \begin{pmatrix} V_1(x) & V_2(x) & \cdots & V_N(x) \end{pmatrix}^\top,
\end{equation}
for smooth functions $V_i\colon \R \to \R$, $i = 1, \ldots, N$. Moreover, we assume $p$ to be smooth and periodic of period $L$. The theory of homogenization \cite[Chapter 3]{BLP78} guarantees the existence of an SDE of the form
\begin{equation}\label{eq:SDE_HOM}
	\d X_t = - A \cdot V'(X_t) \dd t + \sqrt{2\Sigma} \dd W_t,
\end{equation}
such that $X_t^\epl \to X_t$ for $\epl\to 0$ in law as random variables in $\mathcal C^0([0, T]; \R)$. In particular, we have $A = K\alpha$ and $\Sigma = K \sigma$, where the coefficient $0<K<1$ is given by the formula
\begin{equation}\label{eq:K_HOM}
	K = \int_0^L (1 + \Phi'(y))^2 \, \mu(\d y),
\end{equation}
with 
\begin{equation}
	\mu(\d y) = \frac1Z e^{-p(y)/\sigma} \dd y, \quad\text{where}\quad Z = \int_0^L e^{-p(y)/\sigma} \dd y,
\end{equation}
and where the function $\Phi$ is the unique solution with zero-mean with respect to the measure $\mu$ of the two-point boundary value problem
\begin{equation}\label{eq:CellProblem}
	-p'(y)\Phi'(y) + \sigma \Phi''(y) = p'(y), \quad 0 \leq y \leq L,
\end{equation}
endowed with periodic boundary conditions. Let us remark that in this one-dimensional setting it is possible to determine $\Phi$ explicitly, and the homogenization coefficient $K$ is given by
\begin{equation}
	K = \frac{L^2}{Z\widehat Z},
\end{equation}
where
\begin{equation}
	Z = \int_0^L e^{-p(y)/\sigma} \dd y, \quad \widehat Z = \int_0^L e^{p(y)/\sigma} \dd y.
\end{equation}

We now briefly present the classical methodology for estimating the drift coefficient. Let $T > 0$ and let $X \defeq (X_t, 0\leq t \leq T)$ be a realization of the solution of \eqref{eq:SDE_HOM} up to final time $T$. Girsanov's change of measure formula applied to \eqref{eq:SDE_HOM} allows to write the likelihood of $X$ given a drift coefficient $A$ as
\begin{equation}\label{eq:Likelihood}
p(X \mid A) = \exp\left(-\frac{I(X\mid A)}{2\Sigma} \right), 
\end{equation}
where 
\begin{equation}
I(X \mid A) = \int_0^T A \cdot V'(X_t) \dd X_t + \frac12 \int_0^T \left( A \cdot V'(X_t) \right)^2 \dd t.
\end{equation}
Minimizing the functional $I(X \mid A)$ with respect to $A$ therefore gives the maximum likelihood estimator (MLE) of $A$, which can be formally computed in closed form as
\begin{equation}\label{eq:MLE_Simple}
	\widehat A(X, T) \defeq \arg \min_{A \in \R^N} I(X \mid A) = - M^{-1}(X)h(X),
\end{equation}
where $M(X)\in\R^{N\times N}$ and $h(X)\in\R^N$ are defined as
\begin{equation}\label{eq:MandH}
M(X) = \frac1T \int_0^T V'(X_t) \otimes V'(X_t) \dd t, \quad h(X) = \frac1T \int_0^T V'(X_t) \dd X_t,
\end{equation}
and where $\otimes$ denotes the outer product in $\R^N$. Let us now state the assumptions which will be employed throughout the rest of our work. In particular, we consider the same dissipative setting as \cite[Assumption 3.1]{PaS07}.
\begin{assumption}\label{as:regularity} The potentials $p$ and $V$ satisfy
	\begin{enumerate}
		\item $p \in \mathcal C^\infty(\R)$ and is $L$-periodic for some $L > 0$;
		\item\label{as:regularity_diss} $V_i \in \mathcal C^\infty(\R)$ for all $i=1, \ldots, N$ is polynomially bounded from above and bounded from below,  and there exist $a,b > 0$ such that
		\begin{equation}
		-\alpha \cdot V'(x) x \leq a - bx^2;
		\end{equation} 
		\item\label{as:regularity_Lip} $V'$ is Lipschitz continuous, i.e. there exists a constant $C > 0$ such that
		\begin{equation}
			\norm{V'(x) - V'(y)}_2 \leq C\abs{x - y},
		\end{equation} 
		and the components $V'_i$ are polynomially bounded for all $i = 1, \ldots, N$;
		\item\label{as:regularity_SPD} for all $T > 0$, the symmetric matrix $M(X)$ is positive definite and there exists $\bar \lambda > 0$ such that $\lambda_{\min}(M(X)) \geq \bar \lambda$
	\end{enumerate}
\end{assumption}

\begin{remark}\label{rem:regularity_diss} In the following, in particular in the proof of Lemma \ref{lem:ergodicity}, we will employ Assumption \ref{as:regularity}\ref{as:regularity_diss} for the whole drift of the SDE \eqref{eq:SDE_MS}, i.e., the function 
	\begin{equation}
		V^\epl(x) \defeq \alpha \cdot V(x) + p\left(\frac{x}{\epl}\right).
	\end{equation}
	Since $p \in C^\infty(\R)$ and is periodic, all derivatives of $p$ are in $L^\infty(\R)$. Therefore, the assumption above is sufficient for $V^\epl$ to satisfy Assumption \ref{as:regularity}\ref{as:regularity_diss} with different values for $a$ and $b$. In particular, assume Assumption \ref{as:regularity}\ref{as:regularity_diss} holds for $V$. Then, we have for all $\gamma > 0$ by Young's inequality
	\begin{equation}
	\begin{aligned}
		-(V^\epl)'(x)x &\leq a - bx^2 - \frac1\epl p'\left(\frac{x}{\epl}\right) x \\
		&\leq \left(a + \frac1{2\epl^2\gamma}\norm{p'}_{L^\infty(\R)}^2\right) - \left(b - \frac{\gamma}2\right) x^2.
	\end{aligned}
	\end{equation}
	Hence, Assumption \ref{as:regularity}\ref{as:regularity_diss} holds for $V^\epl$ with a coefficient $b$ which is arbitrarily close to the coefficient for $V$, alone.
	
\end{remark}

Under these assumptions, the MLE given in \eqref{eq:MLE_Simple} is indeed the unique minimizer of the likelihood function, as shown in \cite[Theorem 2.4]{PSV09}.

Let us consider the modified estimator of the drift coefficient obtained replacing $X$ with $X^\epl \defeq (X_t^\epl, 0 \leq t \leq T)$ solution of \eqref{eq:SDE_MS}, i.e.,
\begin{equation}\label{eq:MLE}
	\widehat A(X^\epl, T) \defeq \arg \min_{A \in \R^N} I(X^\epl \mid A) = - M^{-1}(X^\epl)h(X^\epl),
\end{equation}
where $I(X^\epl \mid A)$, the matrix $M(X^\epl)$ and the vector $h(X^\epl)$ are obtained replacing each occurrence of $X$ with $X^\epl$. In the following, we assume that Assumption \ref{as:regularity}\ref{as:regularity_SPD} holds as well for the matrix $M(X^\epl)$, and simply denote by $M \defeq M(X^\epl)$ and $h \defeq h(X^\epl)$ in case of no ambiguity. Given the convergence of $X^\epl \to X$ in the space of continuous stochastic processes, one would expect that the MLE \eqref{eq:MLE} would be asymptotically unbiased for the drift coefficient $A$ of the homogenized equation \eqref{eq:SDE_HOM}. Instead, it is possible to prove that in the asymptotic limit for $T \to \infty$ and $\epl \to 0$, the MLE tends to the drift coefficient $\alpha$ of the unhomogenized equation \eqref{eq:SDE_MS}. We report here this result, whose proof can be found for the case $N = 1$ in \cite[Theorem 3.4]{PaS07}. We remark that the proof for $N > 1$ follows directly from the one-dimensional case.

\begin{theorem}\label{thm:Bias} Let Assumption \ref{as:regularity} hold and let $X^\epl_0$ be distributed according to the invariant measure of the process $X^\epl$ solution of \eqref{eq:SDE_MS}. Then
	\begin{equation}
		\lim_{\epl \to 0}\lim_{T \to \infty} \widehat A(X^\epl, T) = \alpha, \quad \text{a.s.},
	\end{equation}
	where $\alpha$ is the drift coefficient of equation \eqref{eq:SDE_MS}.
\end{theorem}

As anticipated in the introduction, the main existing tool for obtaining unbiased estimators in the literature is subsampling the data. In particular, let the dimension of the parameter $N = 1$, let $\delta > 0$ and let $T = n\delta$ with $n$ a positive integer. Then, a subsampled estimator for $A$ is given by
\begin{equation}
	\widehat A_\delta(X^\epl, T) = - \frac{\sum_{j=0}^{n-1} V'(X^\epl_{j\delta})\left(X^\epl_{(j+1)\delta} - X^\epl_{j\delta}\right)}{\delta \sum_{j=0}^{n-1} V'(X^\epl_{j\delta})^2},
\end{equation}
which is a discretized version of $\widehat A(X^\epl, T)$. It is possible to show \cite[Theorem 3.5]{PaS07} that choosing $\delta = \epl^\zeta$ with $\zeta \in (0, 1)$, then $\widehat A_\delta(X^\epl, T)$ is an asymptotically unbiased estimator of $A$ in the limit for $\epl \to 0$, in probability. Despite being widely employed in practice, estimators based on subsampling present some drawbacks, such as having a high variance, as mentioned in the introduction. In the following, we will introduce and analyse a novel approach for the drift estimation.

Estimating the effective diffusion coefficient $\Sigma$ of the homogenized SDE \eqref{eq:SDE_HOM} is as well a relevant problem. Indeed, knowing $\Sigma$ besides the drift coefficient $A$ gives a complete estimation of the effective model \eqref{eq:SDE_HOM}, which is effective for the multiscale data generated by \eqref{eq:SDE_MS} in the sense of homogenization theory. The standard approach for estimating the diffusion coefficient is to compute the quadratic variation of the path. In \cite[Theorem 3.4]{PaS07}, the authors show that this approach fails in case the data is not pre-processed, meaning that the quadratic variation of $X^\epl$ equals the diffusion coefficient $\sigma$ of \eqref{eq:SDE_MS}, even in the limit for $\epl \to 0$. They propose therefore the estimator $\widehat \Sigma_\delta$ based on subsampling that tends to the effective diffusion coefficient $\Sigma$ \cite[Theorem 3.6]{PaS07}. Despite the focus of this work being mainly the effective drift coefficient, we propose in the following an unbiased estimator for the effective diffusion coefficient which fits our novel approach.

\begin{remark}\label{rem:SemiParam} We note that our framework may be viewed in the semi-parametric setting as the one of \cite{KPK13}. In particular the functions $V_i$, $i=1, \ldots, N$ can be seen as the known basis functions of an expansion (e.g. a Taylor expansion) for the unknown confining potential $V_\alpha \colon \R \to \R$ given by
	\begin{equation}
		V_\alpha(x) = \sum_{i=1}^N \alpha_i V_i(x).
	\end{equation}
A numerical example highlighting the potential of our method in such a setting is given in Section \ref{sec:Num_Multi}.
\end{remark}

\begin{remark} Let us remark that for enhancing the clarity of the exposition, in this article we chose to focus on the case of a multi-dimensional parameter in the setting of one-dimensional diffusion processes. In fact, all the theory we present in the following could be generalized to the case of the $d$-dimensional version of the SDE \eqref{eq:SDE_MS}, which can be written as
	\begin{equation}
		\d X_t^\epl = - \sum_{i=1}^N \alpha_i \nabla V_i(X_t^\epl) \dd t - \frac1\epl \nabla p\left(\frac{X_t^\epl}{\epl}\right) \dd t + \sqrt{2\sigma} \dd W_t,
	\end{equation}
	where $W_t$ is a standard $d$-dimensional Brownian motion. Slight modifications of the proof demonstrate that analogous results to ours may be obtained in the $d$-dimensional case. 
\end{remark}

\section{The Filtered Data Approach}\label{sec:Filter}

In this section, we introduce and analyse a novel approach based on filtered data to address the issue that the MLE estimator, when confronted with multiscale data, is biased. Let $\beta, \delta > 0$ and let us consider a family of exponential kernel functions $k \colon \R^+ \to \R$ defined as
\begin{equation}\label{eq:filter}
k(r) = C_\beta \delta^{-1/\beta} e^{-r^\beta/\delta},
\end{equation}
where $C_{\beta}$ is the normalizing constant given by
\begin{equation}
C_\beta =  \beta \, \Gamma(1/\beta)^{-1},
\end{equation}
so that
\begin{equation}
	\int_0^\infty k(r) \dd r = 1,
\end{equation}

and where $\Gamma(\cdot)$ is the gamma function. We consider the process $Z^\epl \defeq (Z^\epl_t, 0 \leq t \leq T)$ defined by the weighted average
\begin{equation}\label{eq:ZDef}
	Z^{\epl}_t \defeq \int_0^t k(t - s)X^\epl_s \dd s.
\end{equation}
The process $Z^\epl$ can be interpreted as a smoothed version of the original trajectory $X^\epl$. In fact, in the field of signal processing the kernel \eqref{eq:filter} belongs to the class of low-pass linear time-invariant filters, which cut the high frequencies in a signal to highlight its slowest components. In the following, rigorous analysis is conducted only when $\beta = 1$. Nonetheless, numerical experiments show that for higher values of $\beta$ the performances of estimators computed employing the filter are more robust and qualitatively better. 

\begin{remark} Given a trajectory $X^\epl$, it is relatively inexpensive to compute $Z^\epl$ from a computational standpoint. In particular, the process $Z^\epl$ is the truncated convolution of the kernel with the process $X^\epl$. Hence, computational tools based on the Fast Fourier Transform (FFT) exist and allow to compute $Z^\epl$ fast component-wise. Moreover, the process $Z^\epl$ can be computed, in case $\beta = 1$, in a recursive manner and therefore ``online''.
\end{remark}

\begin{figure}[t]
	\centering
	\begin{tabular}{cc}
		\includegraphics[]{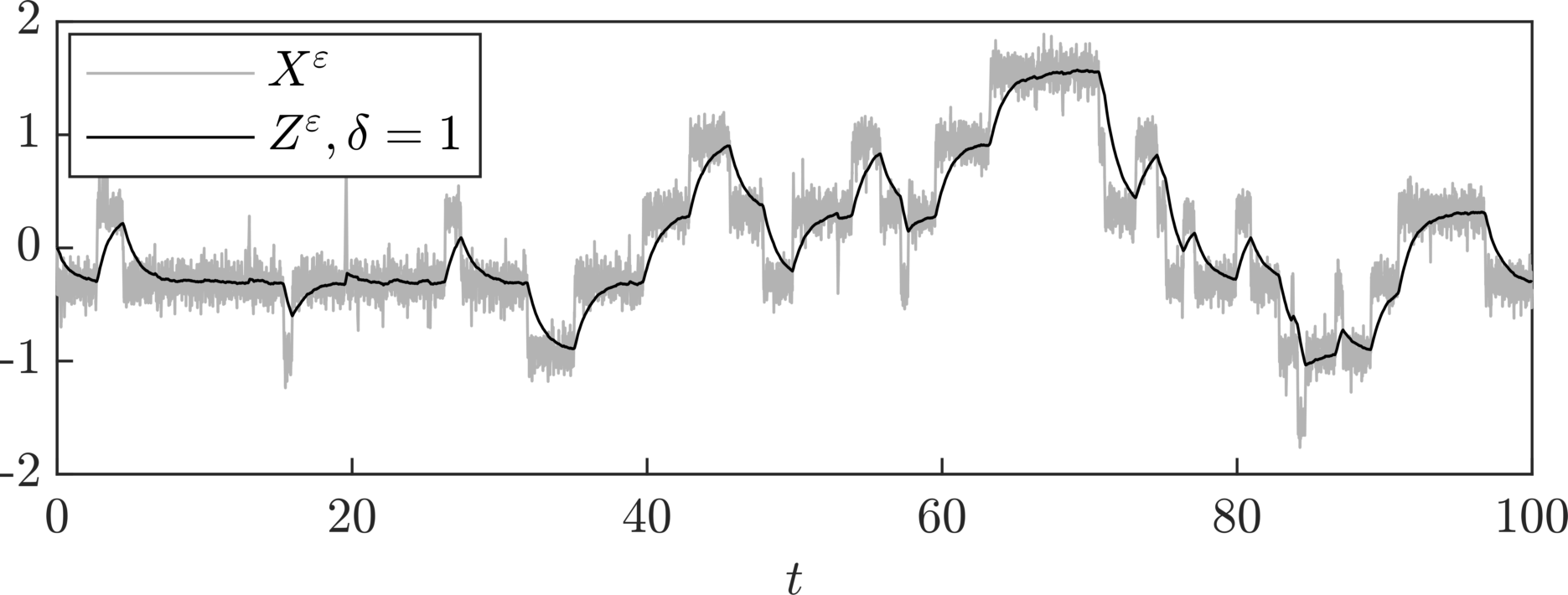} & \includegraphics[]{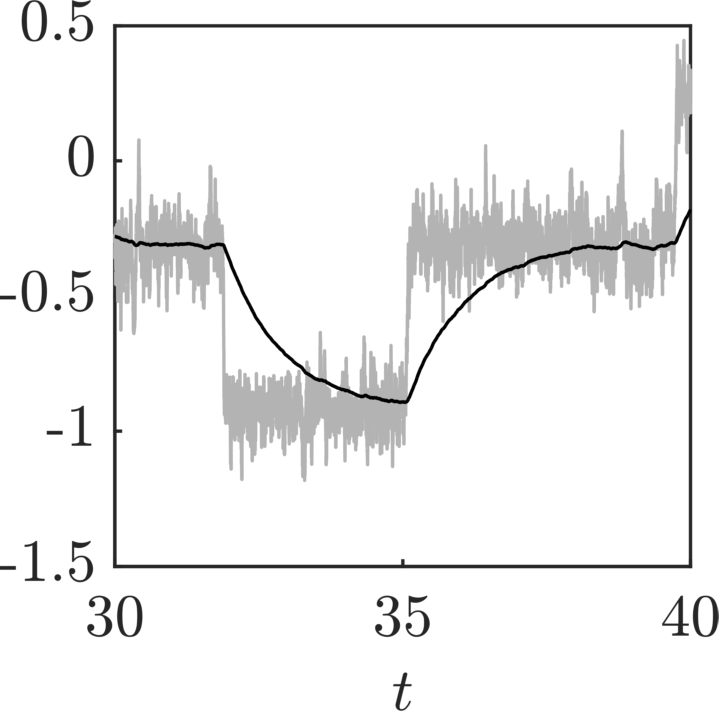}\\
		\includegraphics[]{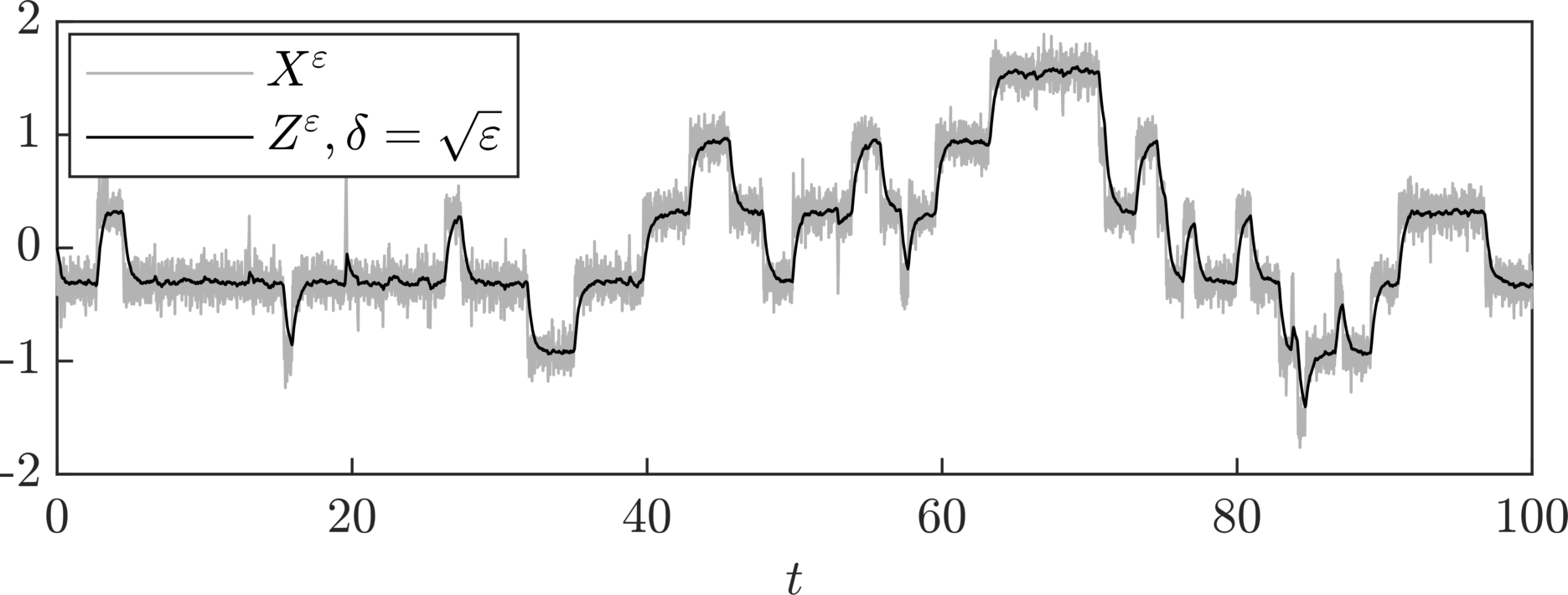} & \includegraphics[]{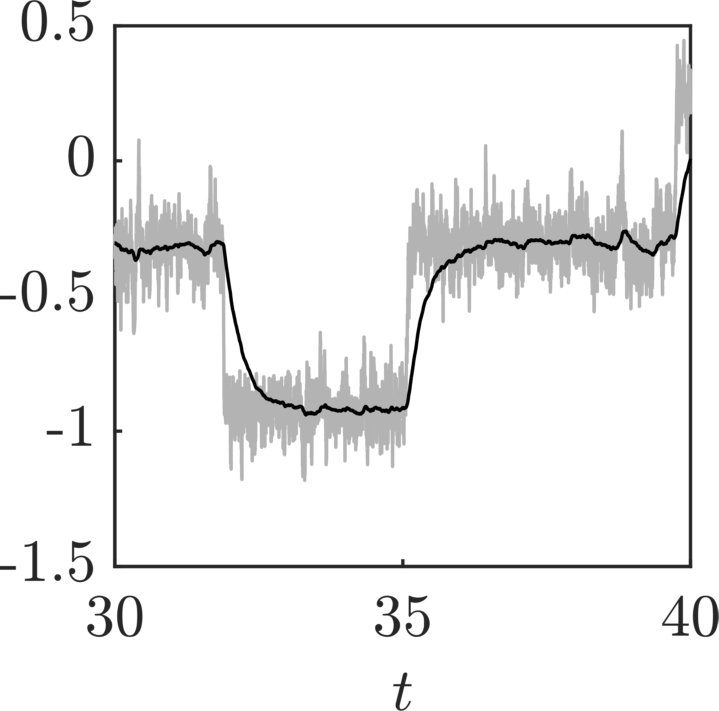} \\
		\includegraphics[]{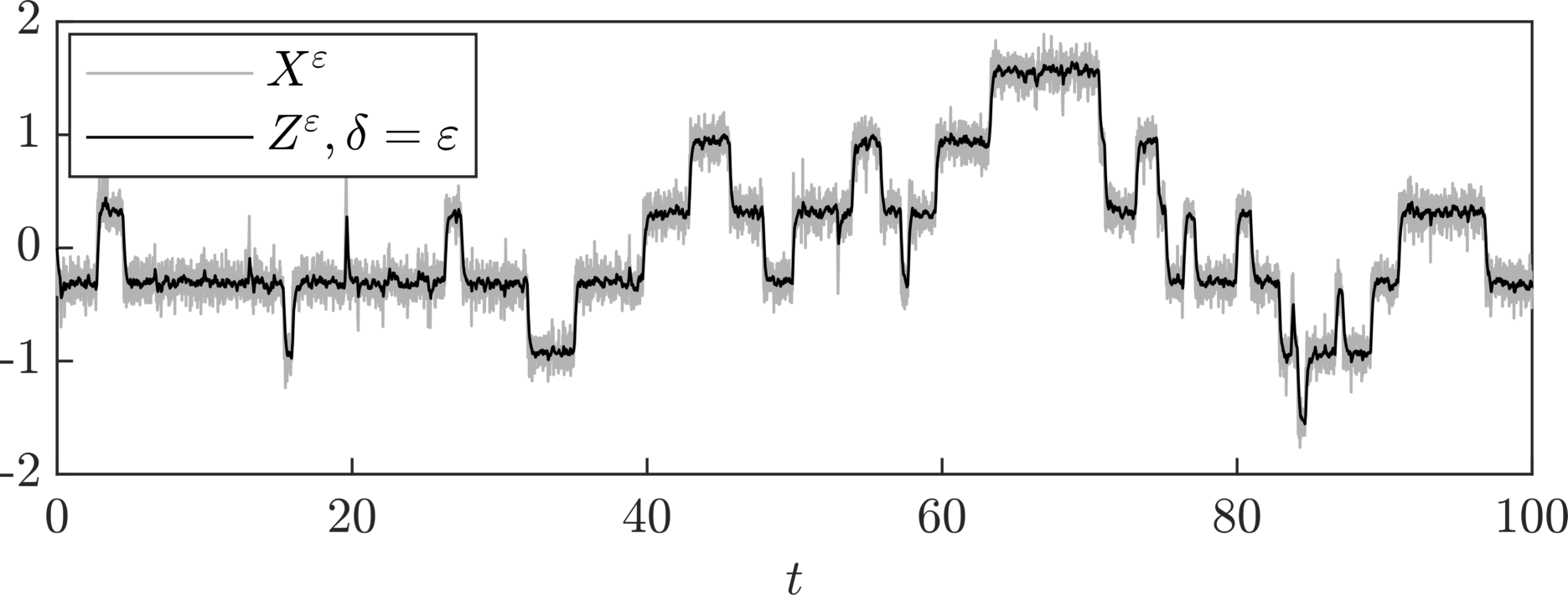}& \includegraphics[]{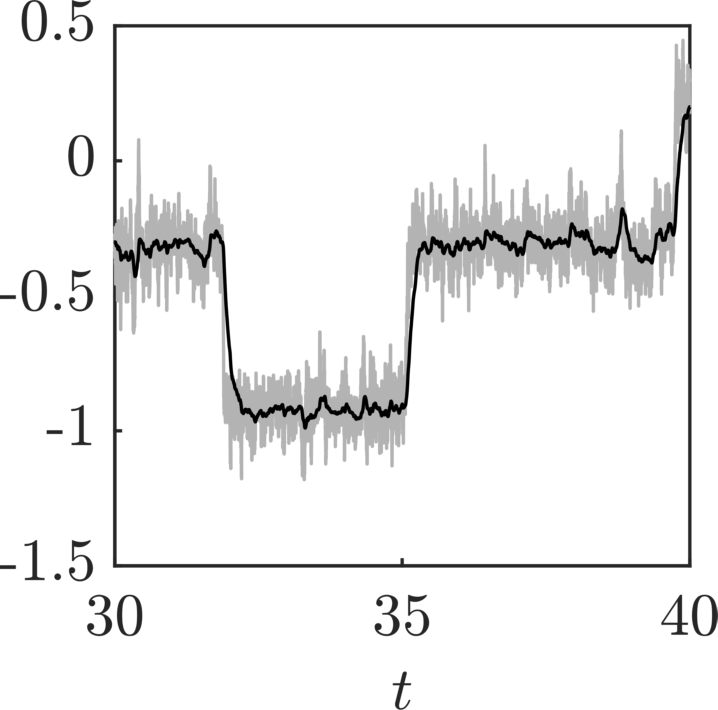}
	\end{tabular}
	\caption{Filtering a trajectory $X^\epl$ obtained with $V(x) = x^2 / 2$, $p(y) = \cos(y)$, $\alpha = 1$, $\sigma = 0.5$ and $\epl = 0.1$. The filtering width is $\delta = \{1, \sqrt{\epl}, \epl\}$ from top to bottom, respectively, and $\beta = 1$.}
	\label{fig:filter}
\end{figure}

Given a trajectory $X^\epl$ and the filtered data $Z^\epl$, the estimator of the drift coefficient we propose is given by
\begin{equation}\label{eq:AHatMixed}
	\widehat A_k(X^\epl, T) = - \widetilde M^{-1}(X^\epl) \widetilde h(X^\epl),
\end{equation}
where we employ the subscript $k$ for reference to the filter's kernel in \eqref{eq:filter}, and where
\begin{equation}\label{eq:MandHTilde}
	\widetilde M(X^\epl) = \frac{1}{T} \int_0^T V'(Z^\epl_t) \otimes V'(X^\epl_t) \dd t, \qquad \text{and} \qquad \widetilde h(X^\epl) = \frac{1}{T} \int_0^T V'(Z^\epl_t) \dd X^\epl_t.
\end{equation}
For economy of notation we drop explicit reference to the dependence of $\widetilde M$ and $\widetilde h$ on $X^\epl$. Let us remark that the formula above is obtained from \eqref{eq:MLE} by replacing only one instance of $X_t^\epl$ with $Z_t^\epl$ in both $M$ and $h$. In particular, it is fundamental for proving unbiasedness to keep in the definition of $h$ the differential of the original process $\d X^\epl_t$ (see Remark \ref{rem:hXZ}). Let us furthermore remark that $\widehat A_k(X^\epl, T)$ need not be the minimizer of some likelihood function based on filtered data. In fact, if one were to replace $Z_t^\epl$ directly in \eqref{eq:Likelihood}, the symmetric part of the matrix $\widetilde M$ would appear and $\widehat A_k(X^\epl, T)$ would not be the minimizer. Therefore, the estimator $\widehat A_k(X^\epl, T)$ has to be thought of as a perturbation of $\widehat A(X^\epl, T)$, directly at the level of estimators and after the maximization procedure. The only theoretical guarantee which is still needed for the well-posedness of $\widehat A_k(X^\epl, T)$ is for $\widetilde M$ to be invertible, which we assume to be true and which we observed to hold in practice.

We now consider the diffusion coefficient, and propose the estimator for $\Sigma$ in \eqref{eq:SDE_HOM} given by
\begin{equation} \label{eq:SigmaHat}
	\widehat \Sigma_k(X^\epl,T) \defeq \frac{1}{\delta T} \int_0^T \left( X_t^\epl - Z_t^\epl \right)^2 \dd t,
\end{equation} 
where again we employ the subscript $k$ for reference to the kernel \eqref{eq:filter} of the filter. As we will show in the following, and in particular in \cref{thm:diffusion_unbiasedness}, the estimator $\widehat \Sigma$ is unbiased for the effective diffusion coefficient $\Sigma$ in case $\beta = 1$ and when we filter data at the multiscale regime, i.e., when $\delta$ is a vanishing function of $\epl$.

Let us from now on consider $\beta = 1$. For this value of $\beta$, the parameter $\delta$ appearing in \eqref{eq:filter} regulates the width of the filtering window. In practice, larger values of $\delta$ will lead to trajectories which are smoother and for which fast-scale oscillations are practically canceled. Let us remark that the filtering width resembles the subsampling step employed for the estimator $\widehat A_\delta(X^\epl, T)$ introduced and analyzed in \cite{PaS07}. For subsampling, the choice guaranteeing asymptotically unbiased results is $\delta = \epl^\zeta$ with $\zeta \in (0, 1)$, and a similar analysis is due for our technique. For visualization purposes, we depict in Figure \ref{fig:filter} the filtered trajectory $Z^\epl$ for three different values of $\delta$, namely $\delta = \{1, \sqrt{\epl}, \epl\}$. With $\delta = 1$, all oscillations at the fast scale are canceled and the filtered trajectory $Z^\epl$ presents only slow-scale variations. Reducing the value of $\delta$, fast-scale oscillations are progressively taken into account. 

In the following, we first focus on the ergodic properties of the process $Z^\epl$ when it is coupled with the process $X^\epl$. This analysis is practically independent of the choice of $\delta$, and is therefore presented on its own. Then, we focus on two different cases which depend on the choice of the width $\delta$ of the filter. First, in Section \ref{sec:Slow}, we consider $\delta$ to be independent of $\epl$, and therefore we filter at the speed of the homogenized process. In this case, we are able to prove that our estimator of the drift coefficient of the homogenized equation is asymptotically unbiased almost surely. This result will be presented in Theorem \ref{thm:mainTheorem}. We then move on in Section \ref{sec:Fast} to the case $\delta \propto \epl^\zeta$, which corresponds to filtering the data at the speed of the multiscale process. In this case, we show that under some conditions on the exponent $\zeta$, we can still obtain estimators which are asymptotically unbiased in probability. This result is proved in Theorem \ref{thm:mainTheorem_zeta}. For this second case, we widely employ techniques and estimates which come from \cite{PaS07}.

\subsection{Ergodic Properties}\label{sec:ergodic}

Let us consider the filtering kernel \eqref{eq:filter} with $\beta = 1$, i.e.,
\begin{equation}\label{eq:filter_beta1}
	k(r) = \frac{1}{\delta} e^{-r/\delta}.
\end{equation}
In this case, Leibniz integral rule yields the equality
\begin{equation}
	\d Z^\epl_t = k(0) X^\epl_t \dd t + \int_0^t k'(t-s) X^\epl_s \dd s \dd t = \frac{1}{\delta} \left ( X^\epl_t - Z^\epl_t \right ) \dd t,
\end{equation}
which can be interpreted as an ordinary differential equation for $Z_t^\epl$ driven by the stochastic signal $X^\epl$. Considering the processes $X^\epl$ and $Z^\epl$ together, we obtain the system of two one-dimensional SDEs
\begin{equation}
\label{eq:systemSDE}
\begin{aligned}
\d X_t^\epl &= -\alpha \cdot V'(X_t^\epl) \dd t - \frac1\epl p'\left(\frac{X_t^\epl}\epl\right) \dd t+ \sqrt{2\sigma} \dd W_t, \\
\d Z^\epl_t &= \frac{1}{\delta} \left ( X^\epl_t - Z^\epl_t \right ) \dd t.
\end{aligned}
\end{equation}
The first ingredient for verifying the ergodic properties of the two-dimensional process $(X^\epl, Z^\epl)^\top \defeq ((X^\epl_t, Z^\epl_t)^\top, 0 \leq t \leq T)$ is verifying that the measure induced by the stochastic process admits a smooth density with respect to the Lebesgue measure. Since noise is present only on the first component, this is a consequence of the theory of hypo-ellipticity, as summarized in the following Lemma, whose proof is given in Appendix \ref{ap:ProofsErgodic}.

\begin{lemma}\label{lem:density} Let $(X^\epl, Z^\epl)^\top$ be the solution of \eqref{eq:systemSDE} and let $\mu^\epl_t$ be the measure induced by the joint process at time $t$. Then, the measure $\mu^\epl_t$ admits a smooth density $\rho^\epl_t$ with respect to the Lebesgue measure.
\end{lemma}

Once it is established that the law of the process admits a smooth density for all times $t>0$, which satisfies a time-dependent Fokker--Planck equation, we are interested in the limiting properties of this law. In particular, we know that the process $X^\epl$ alone is geometrically ergodic \cite[Theorem 4.4]{MSH02}, and we wish the couple $(X^\epl, Z^\epl)^\top$ to inherit the same property. The following Lemma guarantees that the couple is indeed geometrically ergodic, and its proof is given in Appendix \ref{ap:ProofsErgodic}.

\begin{lemma}\label{lem:ergodicity} Let Assumption \ref{as:regularity} hold and let $b > 0$ be given in Assumption \ref{as:regularity}\ref{as:regularity_diss}. Then, if $\delta > 1/(4b)$, the process $(X^\epl, Z^\epl)^\top$ solution of \eqref{eq:systemSDE} is geometrically ergodic, i.e., there exists $C, \lambda > 0$ such that for all measurable $f\colon \R^2\to \R$ such that for some integer $q > 0$ 
	\begin{equation}
		f(x, z) \leq 1 + \norm{\begin{pmatrix} x & z \end{pmatrix}^\top}_2^q,
	\end{equation}
	it holds
	\begin{equation}
		\abs{\E f(X^\epl_t, Z^\epl_t) - \int_\R\int_\R f(x, z) \rho^\epl(x, z) \dd x \dd z} \leq C\left(1 + \norm{\begin{pmatrix} X^\epl_0 & Z^\epl_0 \end{pmatrix}^\top}_2^q \right)e^{-\lambda t},
	\end{equation}
	for $\rho^\epl$-a.e. couple $(X_0^\epl, Z_0^\epl)^\top$, where $\E$ denotes expectation with respect to the Wiener measure, and $\rho^\epl$ is the solution to the stationary Fokker--Planck equation
	\begin{equation}
	\label{eq:FPsystem}
	\sigma \partial^2_{xx} \rho^\epl(x,z) +  \partial_x \left( \left( \alpha \cdot V'(x) + \frac{1}{\epl} p' \left ( \frac{x}\epl \right ) \right ) \rho^\epl(x,z) \right) + \frac{1}{\delta} \partial_z \left((z - x) \rho^\epl(x,z)\right) = 0.
	\end{equation}
\end{lemma}

\begin{remark} The condition $\delta > 1 / (4b)$ is not very restrictive. Let the parameter dimension $N = 1$ and let $V(x) \propto x^{2r}$ for an integer $r > 1$. Then, Assumption \ref{as:regularity}\ref{as:regularity_diss} holds for an arbitrarily large $b > 0$. Therefore, the parameter of the filter $\delta$ can be chosen along the entire positive real axis. A similar argument can be employed for higher dimensions $N > 1$.
\end{remark}

In a general case, it is not possible to find an explicit solution to \eqref{eq:FPsystem}. Nevertheless, it is possible to show some relevant properties of the solution itself, which are summarized in the following Lemma, whose proof is given in Appendix \ref{ap:ProofsErgodic}.

\begin{lemma}\label{lem:FPMarginal} Under the assumptions of Lemma \ref{lem:ergodicity}, let $\rho^\epl$ be the solution of \eqref{eq:FPsystem} and let us write 
\begin{equation}\label{eq:densityDecomposition}
	\rho^\epl(x, z) = \phi^\epl(x)\psi^\epl(z)R^\epl(x,z),
\end{equation}
where $\phi^\epl$ and $\psi^\epl$ are the marginal densities of $X^\epl$ and $Z^\epl$ respectively, i.e., 
\begin{equation}
	\phi^\epl(x) = \int_{\R} \rho^\epl(x,z) \dd z, \quad  \psi^\epl(z) = \int_{\R} \rho^\epl(x,z) \dd x.
\end{equation}
Then, it holds
\begin{equation}\label{eq:marginalX}
	\phi^\epl(x) = \frac{1}{C_{\phi^\epl}} \exp\left(-\frac{1}{\sigma} \alpha \cdot V(x) - \frac{1}{\sigma} p \left ( \frac{x}\epl \right )\right),
\end{equation}
where
\begin{equation}
	C_{\phi^\epl} = \int_{\R} \exp\left(-\frac{1}{\sigma} \alpha \cdot V(x) - \frac{1}{\sigma} p \left ( \frac{x}\epl \right )\right) \dd x.
\end{equation}
Moreover, it holds
\begin{equation}\label{eq:MagicEquality}
	\sigma \delta \int_{\R} \int_{\R} V'(z) \phi^\epl(x) \psi^\epl(z) \partial_x R^\epl(x,z) \dd x \dd z = \E^{\rho^\epl}[(X^\epl - Z^\epl)^2 V''(Z^\epl)].
\end{equation}
\end{lemma}

\begin{remark} Lemma \ref{lem:FPMarginal}, and in particular the equality \eqref{eq:MagicEquality}, plays a fundamental role in the proof of unbiasedness of the estimator based on filtered data. In particular, this equality allows to bypass the explicit knowledge of the function $R(x, z)$, which governs the correlation between the processes $X^\epl$ and $Z^\epl$ at stationarity, for which a closed-form expression is not available in the general case.
\end{remark}

\begin{remark}\label{rem:hXZ} Let us return to the definition of $\widehat A_k$ and replace the differential $\d X^\epl_t$ with $\d Z^\epl_t$ in $\widetilde h$. In this case, it holds
	\begin{equation}
		\lim_{T \to \infty} \frac1T \int_0^T V'(Z_t^\epl) \dd Z_t^\epl = \lim_{T \to \infty} \frac1{\delta T} \int_0^T V'(Z_t^\epl) (X_t^\epl - Z_t^\epl) \dd t = \frac1\delta \E^{\rho^\epl} \left[ V'(Z^\epl) (X^\epl - Z^\epl) \right] = 0,
	\end{equation}
	where the last equality is obtained as in the proof of Lemma \ref{lem:FPMarginal}, with the choice $f(x,z) = V(z)$ at the last line. Therefore, we stress again that it is indeed necessary to employ the original differential $\d X^\epl_t$ in the vector $\widetilde h$ in the definition \eqref{eq:AHatMixed} of $\widehat A_k^\epl$. 
\end{remark}

\begin{remark} Let us consider the kernel \eqref{eq:filter} with $\beta > 1$. In this case, the steps leading to the system \eqref{eq:systemSDE} do not yield a system of Itô SDEs, but of stochastic delay differential equations. The analysis of the estimator in case $\beta > 1$ is therefore based on different arguments than the one we present in this work.
\end{remark}

\subsection{Filtered Data in the Homogenized Regime}\label{sec:Slow}

In this section, we analyze the behavior of the estimator $\widehat A_k(X^\epl, T)$ based on filtered data given in \eqref{eq:AHatMixed} when the filtering width $\delta$ is independent of $\epl$. The analysis in this case is based on the convergence of the couple $(X^\epl, Z^\epl)^\top$ with respect to the multiscale parameter $\epl \to 0$. In particular, it is known that the invariant measure of $X^\epl$ converges weakly to the invariant measure of $X$, the solution of the homogenized equation \eqref{eq:SDE_HOM}. The following result guarantees the same kind of convergence for the couple $(X^\epl, Z^\epl)^\top$.

\begin{lemma}\label{lem:convMeasure} Under Assumption \ref{as:regularity}, let $\mu^\epl$ be the invariant measure of the couple $(X^\epl, Z^\epl)^\top$. If $\delta$ is independent of $\epl$, then the measure $\mu^\epl$ converges weakly to the measure $\mu^0(\d x, \d z) = \rho^0(x, z) \dd x \dd z$, whose density $\rho^0$ is the unique solution of the Fokker--Planck equation
\begin{equation} \label{eq:FPsystem_homogenized}
	\Sigma \partial^2_{xx} \rho^0(x,z) + \partial_x\left( A \cdot V'(x) \rho^0(x,z) \right) + \frac{1}{\delta}\partial_z\left((z - x) \rho^0(x,z) \right) = 0,
\end{equation}
where $A$ and $\Sigma$ are the coefficients of the homogenized equation \eqref{eq:SDE_HOM}.
\end{lemma}

\begin{proof} Let $(X, Z)^\top \defeq \left((X_t, Z_t)^\top, 0\leq t \leq T\right)$ be the solution of
	\begin{equation}
	\label{eq:systemSDEHom}
	\begin{aligned}
	\d X_t &= -A \cdot V'(X_t) \dd t + \sqrt{2\Sigma} \dd W_t, \\
	\d Z_t &= \frac{1}{\delta} \left ( X_t - Z_t \right ) \dd t,
	\end{aligned}
	\end{equation} 
	with $(X_0, Z_0)^\top \sim \mu^0$. The arguments of Section \ref{sec:ergodic} can be repeated to conclude that the invariant measure of $(X, Z)^\top$ admits a smooth density $\rho^0$ which satisfies \eqref{eq:FPsystem_homogenized}. Moreover, standard homogenization theory (see e.g. \cite[Chapter 3, Theorem 6.4]{BLP78} or \cite[Theorem 18.1]{PaS08}) guarantees that $(X^\epl,Z^\epl)^\top \to (X,Z)^\top$ for $\epl \to 0$ in law as random variables with values in $\mathcal C^0([0, T]; \R^2)$, provided that $(X_0^\epl, Z_0^\epl)^\top \sim \mu^\epl$. We remark that traditionally it is assumed that the initial conditions satisfy $(X_0^\epl, Z_0^\epl)^\top = (X_0, Z_0)^\top$ for the homogenization result to hold, but notice that the proof of e.g. \cite[Theorem 18.1]{PaS08} can be shown to hold with a minor modification in case both the multiscale and the homogenized processes are at stationarity. Denoting $E = C^0([0,T], \R^2)$, this means that the measure induced by $(X^\epl, Z^\epl)^\top$ on $(E, \mathcal B(E))$ converges weakly to the measure induced by $(X, Z)^\top$ on the same measurable space (see e.g. \cite[Definition 3.24]{PaS08}). Hence, the measure $\mu^\epl$ converges weakly to $\mu^0$ for $\epl \to 0$.
\end{proof}

\begin{example}\label{ex:OrnUhl} A closed form solution of \eqref{eq:FPsystem_homogenized} can be obtained in a simple case. Let the dimension of the parameter $N=1$ and let $V(x) = x^2/2$. Then, the analytical solution is given by
	\begin{equation}
	\rho^0(x,z) = \frac{1}{C_{\rho^0}} \exp\left(-\frac{A}{\Sigma} \frac{x^2}{2} - \frac{1}{\delta \Sigma} \frac{(x - (1+A \delta)z)^2}{2}\right),
	\end{equation}
	where
	\begin{equation}
	C_{\rho^0} = \int_{\R} \int_{\R} \exp\left(-\frac{A}{\Sigma} \frac{x^2}{2} - \frac{1}{\delta \Sigma} \frac{(x - (1+A \delta)z)^2}{2}\right) \dd x \dd z = \frac{2\pi\Sigma\sqrt \delta}{(1+A\delta)\sqrt A}.
	\end{equation}
	This is the density of a multivariate normal distribution $\mathcal N(0, \Gamma)$, where the covariance matrix is given by
	\begin{equation}
	\Gamma = \frac{\Sigma}{A (1 + A\delta)} \begin{pmatrix} 1+A\delta & 1 \\ 1 & 1 \end{pmatrix}.
	\end{equation}
	Let us remark that this distribution can be obtained from direct computations involving Gaussian processes. In particular, we have that $X$ is in this case an Ornstein--Uhlenbeck process and it is therefore known that $X \sim \mathcal{GP}(m_t, \mathcal C(t, s))$, where at stationarity $m_t = 0$ and
	\begin{equation}
	\mathcal C(t, s) = \frac{\Sigma}{A} e^{-A|t-s|}.
	\end{equation}
	The basic properties of Gaussian processes imply that $Z$ is a Gaussian process, and that the couple $(X, Z)^\top$ is a Gaussian process, too, whose mean and covariance are computable explicitly.
\end{example}

We now present an analogous result to Lemma \ref{lem:FPMarginal} for the limit distribution.

\begin{corollary}\label{lem:FPMarginal_Hom} Let $\rho^0$ be the solution of \eqref{eq:FPsystem_homogenized} and let us write 
	\begin{equation}
		\rho^0(x, z) = \phi^0(x)\psi^0(z)R^0(x,z),
	\end{equation}
	where $\phi^0$ and $\psi^0$ are the marginal densities, i.e., 
	\begin{equation}
		\phi^0(x) = \int_{\R} \rho^0(x,z) \dd z, \quad \psi^0(z) = \int_{\R} \rho^0(x,z) \dd x.
	\end{equation}
	Then, if $A$ and $\Sigma$ are the coefficients of the homogenized equation \eqref{eq:SDE_HOM}, it holds
	\begin{equation} \label{eq:phi0}
		\phi^0(x) = \frac{1}{C_{\phi^0}} \exp\left(- \frac1{\Sigma} A\cdot V(x)\right), \quad \text{where } \quad C_{\phi^0} = \int_{\R} \exp\left(- \frac1{\Sigma} A\cdot V(x) \right) \dd x.
	\end{equation}
	Moreover, it holds
	\begin{equation}
		\Sigma \delta \int_{\R} \int_{\R} V'(z) \phi^0(x) \psi^0(z) \partial_x R^0(x,z) \dd x \dd z = \E^{\rho^0}[(X - Z)^2 V''(Z)].
	\end{equation}
\end{corollary}
\begin{proof} The proof is directly obtained from Lemma \ref{lem:FPMarginal} setting $p(y)=0$ and replacing $\alpha, \sigma$ by $A, \Sigma$ respectively. 
\end{proof}

Let us introduce a notation which will be used throughout the rest of the paper. We denote
\begin{equation}\label{eq:DefCalMTilde}
	\widetilde{\mathcal M}_\epl \defeq \E^{\rho^\epl}[V'(Z^\epl)\otimes V'(X^\epl)], \quad \widetilde{\mathcal M}_0 \defeq \E^{\rho^0}[V'(Z)\otimes V'(X)],
\end{equation}
i.e., $\widetilde{\mathcal M}_\epl$ is obtained in the limit for $T \to \infty$ applying the ergodic theorem elementwise to the matrix $\widetilde M$, and $\widetilde{\mathcal M}_0$ is the limit for $\epl \to 0$ of the matrix $\widetilde{\mathcal M}_\epl$ due to Lemma \ref{lem:convMeasure}. For completeness, we introduce here the symmetric matrices $\mathcal M_\epl$ and $\mathcal M_0$ which are defined as
\begin{equation}\label{eq:DefCalM}
	\mathcal M_\epl \defeq \E^{\rho^\epl}[V'(X^\epl)\otimes V'(X^\epl)], \quad \mathcal M_0 \defeq \E^{\rho^0}[V'(X)\otimes V'(X)],
\end{equation}
and which will be employed in the following. We can now introduce the main result, namely the convergence of the estimator based on filtered data of the drift coefficient of the homogenized equation.

\begin{theorem}\label{thm:mainTheorem} Let the assumptions of Lemma \ref{lem:ergodicity} and Lemma \ref{lem:convMeasure} hold, and let $\widehat A_k(X^\epl, T)$ be defined in \eqref{eq:AHatMixed} with $\delta$ independent of $\epl$. If $\widetilde M$ is invertible, then
	\begin{equation}
	\lim_{\epl \to 0} \lim_{T \to \infty} \widehat A_k(X^\epl,T) = A, \quad \text{a.s.},
	\end{equation}
	where $A$ is the drift coefficient of the homogenized equation \eqref{eq:SDE_HOM}.
\end{theorem}

\begin{proof} Replacing the expression of $\d X^\epl_t$ into \eqref{eq:MandHTilde}, we get for $\widetilde h$
\begin{equation}\label{Ahat_decomposition}
\widetilde h = -\widetilde M \alpha - \frac{1}{T} \int_0^T \frac{1}{\epl} p' \left ( \frac{X^\epl_t}{\epl} \right ) V'(Z^\epl_t) \dd t + \frac{\sqrt{2\sigma}}{T} \int_0^T V'(Z^\epl_t) \dd W_t.
\end{equation}
Therefore, we have
\begin{equation}\label{eq:alphaDecomposition}
\begin{aligned}
	\widehat A_k(X^\epl, T) &= \alpha + \frac{1}{T} \widetilde M^{-1} \int_0^T \frac{1}{\epl} p' \left ( \frac{X^\epl_t}{\epl} \right ) V'(Z^\epl_t) \dd t - \frac{\sqrt{2\sigma}}{T}  \widetilde M^{-1} \int_0^T V'(Z^\epl_t) \dd W_t\\
	&\eqdef \alpha + I_1^\epl(T) - I_2^\epl(T).
\end{aligned}
\end{equation}
We study the terms $I_1^\epl(T)$ and $I_2^\epl(T)$ separately. First, the ergodic theorem applied to $I_1^\epl(T)$ yields
\begin{equation}\label{limit_estimator}
\lim_{T \to \infty} I_1^\epl(T) = \widetilde {\mathcal M}_\epl^{-1} \E^{\rho^\epl} \left [ \frac{1}{\epl} p' \left ( \frac{X^\epl}{\epl} \right ) V'(Z^\epl) \right ], \quad \text{a.s.}
\end{equation}
Replacing the decomposition \eqref{eq:densityDecomposition}, the expression \eqref{eq:marginalX} of $\phi^\epl$ and integrating by parts, we have
\begin{equation}
\begin{aligned}
	\E^{\rho^\epl} \left [ \frac{1}{\epl} p' \left ( \frac{X^\epl}{\epl} \right ) V'(Z^\epl) \right ] &= \int_{\R} \int_{\R} V'(z) \, \frac1\epl p' \left ( \frac{x}\epl \right ) \frac{1}{C_{\phi^\epl}} e^{- \frac{1}{\sigma} \alpha \cdot V(x)} e^{ - \frac{1}{\sigma} p \left ( \frac{x}\epl \right)} \psi^\epl(z) R^\epl(x,z) \dd x \dd z \\
	&= -\sigma \int_{\R} \int_{\R} \frac{\d}{\d x}\left( e^{ - \frac{1}{\sigma} p \left ( \frac{x}\epl \right)} \right) \frac{1}{C_{\phi^\epl}} e^{- \frac{1}{\sigma} \alpha \cdot V(x)} V'(z) \psi^\epl(z) R^\epl(x,z) \dd x \dd z \\
	&= \sigma \int_{\R} \int_{\R} \frac{1}{C_{\phi^\epl}} e^{ - \frac{1}{\sigma} p \left ( \frac{x}\epl \right)} \partial_x \left ( e^{- \frac{1}{\sigma} \alpha \cdot V(x)} R^\epl(x,z) \right ) V'(z) \psi^\epl(z) \dd x \dd z,
\end{aligned}
\end{equation}
which implies
\begin{align}
	\E^{\rho^\epl} \left [ \frac{1}{\epl} p' \left ( \frac{X^\epl}{\epl} \right ) V'(Z^\epl) \right ] &= 
	\begin{aligned}[t]
		&- \left(\int_{\R} \int_{\R} V'(z) \otimes V'(x) \rho^\epl(x, z) \dd x \dd z\right) \alpha \\
		&+\sigma \int_{\R} \int_{\R} V'(z) \phi^\epl(x) \psi^\epl(z) \partial_x R^\epl(x,z) \dd x \dd z 
	\end{aligned}
	\\
	&= - \widetilde {\mathcal M}_\epl\alpha + \sigma \int_{\R} \int_{\R} V'(z) \phi^\epl(x) \psi^\epl(z) \partial_x R^\epl(x,z) \dd x \dd z.
\end{align}
Replacing the equality above into \eqref{limit_estimator}, we obtain
\begin{equation}
\lim_{T \to \infty} I_1^\epl(T) = -\alpha + \widetilde {\mathcal M}_\epl^{-1} \sigma \int_{\R} \int_{\R} V'(z) \phi^\epl(x) \psi^\epl(z) \partial_x R^\epl(x,z) \dd x \dd z, \quad \text{a.s.}
\end{equation}
Due to Lemma \ref{lem:FPMarginal}, we therefore have
\begin{equation} \label{eq:for_zeta}
	\lim_{T \to \infty} I_1^\epl(T) = -\alpha + \frac1\delta \widetilde {\mathcal M}_\epl^{-1} \E^{\rho^\epl}[(X^\epl - Z^\epl)^2 V''(Z^\epl)], \quad \text{a.s.}	
\end{equation}
Since $\delta$ is independent of $\epl$, we can pass to the limit as $\epl$ goes to zero and Lemma \ref{lem:convMeasure} yields
\begin{equation} \label{eq:limit_estimator_2}
	\lim_{\epl \to 0} \lim_{T \to \infty} I_1^\epl(T) = -\alpha + \frac1\delta \widetilde {\mathcal M}_0^{-1} \E^{\rho^0}[(X - Z)^2 V''(Z)], \quad \text{a.s.}
\end{equation}
Due to Corollary \ref{lem:FPMarginal_Hom}, we have
\begin{equation}\label{eq:final_result}
	\frac{1}{\delta} \E^{\rho^0}[(X - Z)^2 V''(Z)] = \Sigma \int_{\R} \int_{\R} V'(z) \phi^0(x) \psi^0(z) \partial_x R^0(x,z) \dd x \dd z,
\end{equation}
and moreover, an integration by parts yields
\begin{equation}
\begin{aligned}
	\frac{1}{\delta} \E^{\rho^0}[(X - Z)^2 V''(Z)] &= -\Sigma \int_{\R} \int_{\R} V'(z) (\phi^0)'(x) \psi^0(z) R^0(x,z) \dd x \dd z\\
	&= -\Sigma \int_{\R} \int_{\R} V'(z) \frac{\d}{\d x} \left ( \frac{1}{C_{\phi^0}} e^{-\frac{1}{\Sigma} A\cdot V(x)} \right ) \psi^0(z) R^0(x,z) \dd x \dd z \\
	&= \left(\int_{\R} \int_{\R} V'(z) \otimes V'(x) \rho^0(x,z) \dd x \dd z\right)A\\
	&=  \widetilde {\mathcal M}_0 A.
\end{aligned}
\end{equation}
We can therefore conclude that
\begin{equation}\label{eq:Proof_I1}
\lim_{\epl \to 0} \lim_{T \to \infty} I_1^\epl(T) = -\alpha + A, \quad \text{a.s.}
\end{equation}
We now consider the second term $I_2^\epl(T)$, and rewrite it as
\begin{equation}
	I^\epl_2(T) = \sqrt{2\sigma} I_{2,1}^\epl(T)  I_{2,2}^\epl(T),
\end{equation}
where
\begin{equation}
\begin{aligned}
	I_{2,1}^\epl(T) &\defeq \left(\frac{1}{T} \int_0^T V'(Z^\epl_t) \otimes V'(X^\epl_t) \dd t\right)^{-1}\left(\frac{1}{T} \int_0^T V'(Z^\epl_t) \otimes V'(Z^\epl_t) \dd t\right),\\
	I_{2,2}^\epl(T) &\defeq \left(\frac{1}{T} \int_0^T V'(Z^\epl_t) \otimes V'(Z^\epl_t) \dd t\right)^{-1} \left(\frac1T \int_0^T V'(Z^\epl_t) \dd W_t\right).
\end{aligned}
\end{equation}
The ergodic theorem yields
\begin{equation}
	\lim_{T \to \infty} I_{2,1}^\epl(T) = \widetilde {\mathcal M}_\epl^{-1}\E^{\rho^\epl}\left[V'(Z^\epl) \otimes V'(Z^\epl)\right] \eqdef R^\epl,
\end{equation}
where $R^\epl$ is bounded uniformly in $\epl$ due to the theory of homogenization, Assumption \ref{as:regularity}\ref{as:regularity_Lip}-\ref{as:regularity_SPD} and Lemma \ref{lem:bounded_momentZ}. Moreover, always due to Lemma \ref{lem:bounded_momentZ} and Assumption \ref{as:regularity}\ref{as:regularity_Lip} we have that $V'(Z^\epl)$ is square integrable, and hence the strong law of large numbers for martingales implies
\begin{equation}
	\lim_{T \to \infty} I_{2,2}^\epl(T) = 0, \quad \text{a.s.},
\end{equation}
independently of $\epl$. Therefore
\begin{equation}
	\lim_{\epl \to 0}\lim_{T \to \infty} I_2^\epl(T) = 0, \quad \text{a.s.},
\end{equation}
which, together with \eqref{eq:Proof_I1} and \eqref{eq:alphaDecomposition}, proves the desired result.
\end{proof}

\begin{remark}\label{rem:deltaPower} Let us remark that the assumption that $\delta$ is independent of $\epl$ is necessary to pass from \eqref{eq:for_zeta} to \eqref{eq:limit_estimator_2} but is not needed before \eqref{eq:for_zeta}. Moreover, the term $I_2^\epl(t)$ in the proof vanishes a.s. independently of $\epl$. Therefore, in the analysis of the case $\delta = \mathcal O(\epl^\zeta)$ it will be sufficient for unbiasedness to show that
	\begin{equation}
		\lim_{\epl\to 0}\frac1\delta \widetilde {\mathcal M}_\epl^{-1} \E^{\rho^\epl}[(X^\epl - Z^\epl)^2 V''(Z^\epl)] = A,
	\end{equation}
	which is a non-trivial limit since $\delta \to 0$ for $\epl \to 0$.
\end{remark}

\subsection{Filtered Data in the Multiscale Regime}\label{sec:Fast}

We now consider the case of the filtering width $\delta = \mathcal O(\epl^\zeta)$, where $\zeta > 0$ will be specified in the following. In this case, the filtered process resembles more the original process $X^\epl$, as can be noted in Figure \ref{fig:filter}. Moreover, the techniques employed for proving Theorem \ref{thm:mainTheorem} can only be partly exploited, as highlighted by Remark \ref{rem:deltaPower}. In fact, in order to prove unbiasedness it is necessary to characterize precisely the difference between the processes $Z^\epl$ and $X^\epl$. A first characterization is given by the following Proposition, whose proof is found in Appendix \ref{ap:ProofDistance}.

\begin{proposition} \label{prop:zeta} Let Assumption \ref{as:regularity} hold and $\epl,\delta>0$ be sufficiently small. Then, it holds for every $t > 0$
	\begin{equation}
	X_t^\epl - Z_t^\epl = \delta B^\epl_t + R(\epl,\delta),
	\end{equation}
	where the stochastic process $B_t^\epl$ is defined as
	\begin{equation}\label{eq:defBeta}
	B_t^\epl \defeq \sqrt{2\sigma} \int_0^t k(t-s)(1 + \Phi'(Y_s^\epl)) \dd W_s,
	\end{equation}
	where $\Phi$ is the solution of the cell problem \eqref{eq:CellProblem}, $W_s$ is the Brownian motion appearing in \eqref{eq:SDE_MS} and $Y_t^\epl = X_t^\epl / \epl$. Moreover, $B_t^\epl$ and the remainder $R(\epl,\delta)$ satisfy for every $p \geq 1$ the estimates
	\begin{equation} \label{eq:estimate_beta}
	\left( \E^{\phi^\epl} \abs{B_t^\epl}^p \right)^{1/p} \leq C \delta^{-1/2},
	\end{equation}
	and
	\begin{equation} \label{eq:remainder_zeta}
	\left( \E^{\phi^\epl} \left| R(\epl,\delta) \right|^p \right)^{1/p} \le C \left( \delta + \epl + \max \{ 1,t \} e^{-t/\delta} \right),
	\end{equation}
	where $C$ is independent of $\epl$, $\delta$ and $t$, and $\phi^\epl$ is the density of the invariant measure of $X^\epl$.
\end{proposition}

It is clear from the Proposition above that understanding the properties of the process $B_t^\epl$ is key to understanding the behavior of the difference between $X^\epl$ and $Z^\epl$. In particular, we can write the dynamics of $B_t^\epl$ with an application of the Itô formula and due to the properties of the kernel $k(t)$ as
\begin{equation}
	\d B_t^\epl = - \frac1\delta B_t^\epl \dd t + \frac{\sqrt{2\sigma}}{\delta}(1+\Phi'(Y_t^\epl)) \dd W_t.
\end{equation}
This equation can be coupled with the dynamics of the processes $X_t^\epl$, $Y_t^\epl$ and $Z_t^\epl$, thus describing the evolution of the quadruple $(X^\epl, Y^\epl, Z^\epl, B^\epl)$ together. In particular, it is possible to show that the results of Section \ref{sec:ergodic} hold for the quadruple, and the properties of the invariant measure of the quadruple can be exploited to prove the unbiasedness of the estimator in the case $\delta = \mathcal O(\epl^\zeta)$ in the same way as in the case $\delta$ independent of $\epl$. In this context, a further assumption on the potential $V$ is necessary.
\begin{assumption}\label{as:regularityZeta} The derivatives $V''$ and $V'''$ of the potential $V \colon \R \to \R^N$ are component-wise polynomially bounded, and the second derivative is Lipschitz, i.e., there exists a constant $L > 0$ such that
	\begin{equation}
		\norm{V''(x)-V''(y)} \leq L \abs{x - y},
	\end{equation}
	for all $x, y \in \R$.
\end{assumption}

In light of Remark \ref{rem:deltaPower}, it is fundamental to understand the behavior of the quantity
\begin{equation}
	\frac1\delta (X_t^\epl - Z_t^\epl)^2 V''(Z_t^\epl),
\end{equation}
as well as its limit for $t\to \infty$ and for $\epl \to 0$. Let us remark that due to Proposition \ref{prop:zeta} we have
\begin{equation}
	\frac1\delta (X_t^\epl - Z_t^\epl)^2 V''(Z_t^\epl) \approx \delta (B_t^\epl)^2 V''(Z_t^\epl),
\end{equation}
and therefore studying the right hand side of the approximate equality above is the goal of the upcoming discussion. The following result, whose proof is in Appendix \ref{ap:ProofsDeltaZeta}, gives a first characterization.

\begin{lemma} \label{lem:quadruple} Under Assumptions \ref{as:regularity} and \ref{as:regularityZeta}, let $\eta^\epl$ be the invariant measure of the quadruple $(X^\epl, Y^\epl, Z^\epl, B^\epl)$. Then it holds
	\begin{equation}
	\delta \E^{\eta^\epl} \left[ (B^\epl)^2 V''(Z^\epl)  \right] = \sigma \E^{\eta^\epl} [(1 + \Phi'(Y^\epl))^2V''(Z^\epl) ] + \widetilde R(\epl,\delta),
	\end{equation}
	where the remainder $\widetilde R(\epl,\delta)$ satisfies
	\begin{equation} \label{eq:quadruple_remainder}
	\left| \widetilde R(\epl,\delta) \right| \le C \left( \delta^{1/2} + \epl \right).
	\end{equation}
\end{lemma}

Let us remark that the quantity appearing above hints towards the theory of homogenization. In fact, we recall that the homogenization coefficient $K$ is given by
\begin{equation}
	K = \int_0^L \left(1 + \Phi'(y)\right)^2 \mu(\d y),
\end{equation}
where $\mu$ is the marginal measure of the process $Y^\epl$ when coupled with $X^\epl$. Therefore, the next step is the homogenization limit, i.e., the limit of vanishing $\epl$, which is considered in the following Lemma, and whose proof is given in Appendix \ref{ap:ProofsDeltaZeta}.

\begin{lemma} \label{lem:quadruple_convergence} Let the assumptions of Lemma \ref{lem:quadruple} hold, and let $\delta = \epl^\zeta$ with $\zeta > 0$. Then, it holds
	\begin{equation}
	\lim_{\epl \to 0} \sigma \E^{\eta^\epl} [(1 + \Phi'(Y^\epl))^2V''(Z^\epl) ] = \Sigma \E^{\phi^0} [V''(X)],
	\end{equation}
	where $\Sigma$ is the diffusion coefficient of the homogenized equation \eqref{eq:SDE_HOM}.
\end{lemma}

Provided with the results presented above, we can prove the following Theorem, stating that the estimator $\widehat A_k(X^\epl, T)$ is asymptotically unbiased even in the case of the filtering width $\delta$ vanishing with respect to the multiscale parameter $\epl$.

\begin{theorem}\label{thm:mainTheorem_zeta} Let the assumptions of Lemma \ref{lem:ergodicity} and Lemma \ref{lem:quadruple_convergence} hold. Let $\widehat A_k(X^\epl, T)$ be defined in \eqref{eq:AHatMixed} and $\delta = \epl^\zeta$ with $\zeta \in (0,2)$. If $\widetilde M$ is invertible, then
		\begin{equation}
		\lim_{\epl \to 0} \lim_{T \to \infty} \widehat A_k(X^\epl, T) = A, \quad \text{in probability},
		\end{equation}
		where $A$ is the drift coefficient of the homogenized equation \eqref{eq:SDE_HOM}.
\end{theorem}
\begin{proof}
	Let us introduce the notation
	\begin{equation}
	\mathcal A^\epl(\delta) \defeq \frac1\delta \widetilde{\mathcal M}_\epl^{-1} \E^{\rho^\epl}[(X^\epl - Z^\epl)^2 V''(Z^\epl)],
	\end{equation}
	where $\widetilde{\mathcal M}_\epl$ is defined in \eqref{eq:DefCalMTilde}.	Then following the proof of Theorem \ref{thm:mainTheorem} and in light of Remark \ref{rem:deltaPower}, we only need to show that if $\delta = \epl^\zeta$ with $\zeta \in (0,2)$ we have
	\begin{equation}
	\lim_{\epl \to 0} \mathcal A^\epl(\delta) = A, \quad \text{in probability}.
	\end{equation}
	Using Proposition \ref{prop:zeta} and geometric ergodicity for taking the limit for $t \to \infty$ (Lemma \ref{lem:ergodicity}), we have the following equality
	\begin{equation}
	\begin{aligned}
	\mathcal A^\epl(\delta) &= \widetilde{\mathcal M}_\epl^{-1} \frac1\delta \lim_{t \to \infty} \E [(X_t^\epl - Z_t^\epl)^2 V''(Z_t^\epl)] \\
	&= \widetilde{\mathcal M}_\epl^{-1} \frac1\delta \lim_{t \to \infty} \E \left[ \left( \delta B_t^\epl + R(\epl,\delta) \right)^2 V''(Z_t^\epl) \right] \\
	&\eqdef \widetilde{\mathcal M}_\epl^{-1} \lim_{t \to \infty} \left( J_1^\epl(t) + J_2^\epl(t) + J_3^\epl(t) \right),
	\end{aligned}
	\end{equation}
	where $R(\epl, \delta)$ is given in Proposition \ref{prop:zeta}, $\E$ denotes the expectation with respect to the Wiener measure and
	\begin{equation}
	\begin{aligned}
	J_1^\epl(t) &= \delta \E \left[ (B_t^\epl)^2 V''(Z_t^\epl) \right], \\
	J_2^\epl(t) &= 2 \E \left[ B_t^\epl R(\epl,\delta) V''(Z_t^\epl) \right], \\
	J_3^\epl(t) &= \frac1\delta \E\left[ R(\epl,\delta)^2 V''(Z_t^\epl) \right].
	\end{aligned}
	\end{equation}
	Let us consider the three terms separately. First, by geometric ergodicity and applying Lemma \ref{lem:quadruple} and Lemma \ref{lem:quadruple_convergence} we get
	\begin{equation}
	\begin{aligned}
	\lim_{\epl \to 0} \lim_{t \to \infty} J_1^\epl(t) &= \lim_{\epl \to 0}\delta \E^{\eta^\epl} \left[ (B^\epl)^2 V''(Z^\epl) \right] \\
	&= \lim_{\epl \to 0} \left(\sigma \E^{\eta^\epl} [V''(Z^\epl) (1 + \Phi'(Y^\epl))^2] + \widetilde R(\epl,\delta)\right)\\
	&= \Sigma \E^{\phi^0} [V''(X)].
	\end{aligned}
	\end{equation}
	Let us now consider $J_2^\epl(t)$. Considering Hölder conjugates $p,q,r$ the Hölder inequality yields
	\begin{equation}
		\abs{J_2^\epl(t)} \leq \E[(B_t^\epl)^p]^{1/p}\E[R(\epl,\delta)^q]^{1/q}\E[V''(Z^\epl)^r]^{1/r}.
	\end{equation}
	Now, we can bound the first two terms with \eqref{eq:estimate_beta} and \eqref{eq:remainder_zeta}, respectively. The third term is bounded due to Assumption \ref{as:regularityZeta} and Lemma \ref{lem:bounded_momentZ}. Hence, we have for $t$ sufficiently large
	\begin{equation}
		\abs{J_2^\epl(t)} \le C \left( \delta^{1/2} + \epl \delta^{-1/2} \right).
	\end{equation}
	We consider now $J_3^\epl(t)$. The Hölder inequality yields for conjugates $p$ and $q$
	\begin{equation}
		\abs{J_3^\epl(t)} \le \E[R(\epl, \delta)^{2p}]^{1/p} \E[V''(Z_t^\epl)^q]^{1/q},
	\end{equation}
	which, similarly as above, yields for $t$ sufficiently large
	\begin{equation}
		\abs{J_3^\epl(t)} \le C \left( \delta + \epl^2 \delta^{-1} \right).
	\end{equation}
	Therefore, since $\delta = \mathcal O(\epl^\zeta)$ for $\zeta \in (0, 2)$, the terms $J_2^\epl(t)$ and $J_3^\epl(t)$ vanish in the limit for $t \to \infty$ and $\epl \to 0$. Furthermore, by Lemma \ref{lem:distanceMandTildeM} and by weak convergence of the invariant measure $\mu^\epl$ to $\mu^0$, we have
	\begin{equation}
	\lim_{\epl \to 0} \widetilde{\mathcal M}_\epl = \mathcal M_0,
	\end{equation}
	where $\mathcal M_0$ is defined in \eqref{eq:DefCalM}. Therefore
	\begin{equation}
	\lim_{\epl \to 0} \mathcal A^\epl(\delta) = \Sigma \mathcal M_0^{-1} \E^{\phi^0} [V''(X)],
	\end{equation}
	and, finally, employing \eqref{eq:phi0} and \eqref{eq:DefCalM} and integrating by parts yields
	\begin{equation}
	\lim_{\epl \to 0} \mathcal A^\epl(\delta) = \Sigma \mathcal M_0^{-1} \frac1\Sigma \mathcal M_0 A = A,
	\end{equation}
	which implies the desired result.
\end{proof}

We conclude the analysis concerning the estimator $\widehat A_k$ for the effective drift coefficient with a negative convergence result, i.e., that if $\delta = \epl^\zeta$ with $\zeta > 2$, the estimator based on filtered data converges to the coefficient $\alpha$ of the unhomogenized equation. This result is relevant for two reasons. First, it shows the sharpness of the bound on $\zeta$ in the assumptions of Theorem \ref{thm:mainTheorem_zeta}. Second, it shows an interesting switch between two completely different regimes at $\zeta = 2$, which happens arbitrarily fast in the limit $\epl \to 0$. 
\begin{theorem}\label{thm:mainTheorem_zetaAlpha} Let the assumptions of Lemma \ref{lem:ergodicity} and Assmuption \ref{as:regularityZeta} hold. Let $\widehat A_k(X^\epl, T)$ be defined in \eqref{eq:AHatMixed} and $\delta = \epl^\zeta$ with $\zeta > 2$. If $\widetilde M$ is invertible, then
	\begin{equation}
	\lim_{\epl \to 0} \lim_{T \to \infty} \widehat A_k(X^\epl, T) = \alpha, \quad \text{in probability},
	\end{equation}
	where $\alpha$ is the drift coefficient of the multiscale equation \eqref{eq:SDE_MS}.
\end{theorem}
The proof is given in Appendix \ref{ap:ProofsDeltaZeta}. 

We conclude this section by proving a result of asymptotic unbiasedness for the estimator $\widehat \Sigma_k$ of the effective diffusion coefficient $\Sigma$ defined in \eqref{eq:SigmaHat}. The proof is given in Appendix \ref{ap:diffusion}.
\begin{theorem} \label{thm:diffusion_unbiasedness}
	Let the Assumptions of Theorem \ref{thm:mainTheorem_zetaAlpha} hold. Then, if $\delta = \epl^\zeta$, with $\zeta \in (0,2)$, it holds
	\begin{equation}
		\lim_{\epl \to 0} \lim_{T \to \infty} \widehat \Sigma_k(X^\epl,T) = \Sigma, \quad \text{in probability},
	\end{equation} 
	where $\Sigma$ is the diffusion coefficient of the homogenized equation \eqref{eq:SDE_HOM}.
\end{theorem}

\section{The Bayesian Setting}\label{sec:Bayesian}

In this section we present a Bayesian reinterpretation of the inference procedure, which, given the structure of the problem, allows full uncertainty quantification with little more computational effort than required for the MLE. 

Let us fix a Gaussian prior $\mu_0 = \mathcal N(A_0, C_0)$ on $A$, where $A_0 \in \R^N$ and $C_0 \in \R^{N\times N}$ is symmetric positive definite. Then, given a final time $T > 0$, the posterior distribution $\mu_{T,\epl}$ admits a density $p(A \mid X^\epl)$ with respect to the Lebesgue measure which satisfies
\begin{equation}
p(A \mid X^\epl) = \frac1{Z^\epl} \, p(X^\epl \mid A) \, p_0(A),
\end{equation}
where $Z^\epl$ is the normalization constant, $p_0$ is the density of $\mu_0$, and where the likelihood $p(X^\epl \mid A)$ is given in \eqref{eq:Likelihood}. The log-posterior density is therefore given by
\begin{equation}\label{eq:Posterior}
\log p(A \mid X^\epl) = -\log Z^\epl - \frac{T}{2\Sigma} A \cdot h - \frac{T}{4\Sigma} A \cdot M A - \frac12 (A - A_0) \cdot C_0^{-1}(A-A_0),
\end{equation}
where $M$ and $h$ are defined in \eqref{eq:MLE}. Since the log-posterior density is quadratic in $A$, the posterior is Gaussian, and it is therefore sufficient to determine its mean and covariance to fully characterize it. We denote by $m_{T,\epl}$ and $C_{T,\epl}$ the mean and covariance matrix, respectively. Completing the squares in the log-posterior density, we formally obtain
\begin{equation}\label{eq:ParamBayes}
\begin{aligned}
C_{T,\epl}^{-1} &= C_0^{-1} + \frac{T}{2\Sigma} M, \\
C_{T,\epl}^{-1}m_{T,\epl} &= C_0^{-1}A_0 - \frac{T}{2\Sigma} h. 
\end{aligned}
\end{equation}
Under Assumption \ref{as:regularity}, one can show that the posterior at time $T > 0$ is well defined and given by $\mu_{T,\epl}(\cdotp\mid X^\epl) = \mathcal N(m_{T,\epl}, C_{T,\epl})$. Let us remark that in order to compute the posterior covariance $C_{T,\epl}$ the value of the diffusion coefficient $\Sigma$ of the homogenized equation is needed. Although the exact value is in general unknown, it can be estimated employing the subsampling technique presented in \cite{PaS07} or with the estimator $\widehat \Sigma_k$ given in \eqref{eq:SigmaHat} based on filtered data. In fact, we verified in practice that the estimator of the diffusion coefficient based on subsampling is more robust with respect to the subsampling step than the estimator for the drift coefficient. In the following theorem, we show that the posterior distribution obtained with no pre-processing of the data contracts asymptotically to the drift coefficient of the unhomogenized equation. We characterize the contraction by verifying that the posterior measure concentrates in arbitrarily small balls. Let us finally remark that the measure $\mu_{T, \epl}$ is a random measure, and therefore contraction has to be considered averaged with respect to the Wiener measure. The choice of the contraction measure and some parts of the proof are taken from \cite[Theorem 5.2]{PSZ13}.

\begin{theorem}\label{thm:BayesianBias} Under Assumption \ref{as:regularity}, the posterior measure $\mu_{T,\epl}(\cdotp\mid X^\epl) = \mathcal N(m_{T,\epl}, C_{T,\epl})$ satisfies for all $c > 0$
	\begin{equation}
		\lim_{\epl \to 0}\lim_{T \to \infty} \E \left[\mu_{T,\epl}\left(\{a\colon \norm{a-\alpha}_2\geq c\}\mid X^\epl\right)\right] = 0,
	\end{equation}
	where $\E$ denotes expectation with respect to the Wiener measure and $\alpha$ is the drift coefficient of the unhomogenized equation \eqref{eq:SDE_MS}.
\end{theorem}

\begin{remark} The result above has the same consequences in the Bayesian setting as Theorem \ref{thm:Bias} has for the MLE. In particular, it shows that the posterior distribution obtained when data is not pre-processed concentrates asymptotically on the drift coefficient of the unhomogenized equation \eqref{eq:SDE_MS}. Moreover, a partial result which can be deduced from the proof is that in the limit for $T \to \infty$ and for a positive value $\epl > 0$ the Bayesian and the MLE approaches are equivalent. In particular, we have for all $\epl > 0$
	\begin{equation}
	\begin{aligned}
	&\lim_{T\to\infty} \norm{C_{T,\epl}}_2 = 0,\\
	&\lim_{T \to\infty}\norm{m_{T,\epl} - \widehat A(X^\epl, T)}_2 =0,
	\end{aligned}	
	\end{equation} 
	i.e., the weak limit of the posterior $\mu_{T,\epl}$ for $T\to \infty$ is the Dirac delta concentrated on the limit of $\widehat A(X^\epl, T)$ for $T\to \infty$. 
\end{remark}

\begin{proof}[Proof of Theorem \ref{thm:BayesianBias}] The proof of \cite[Theorem 5.2]{PSZ13} guarantees that if the trace of $C_{T,\epl}$ tends to zero and if the mean $m_{T,\epl}$ tends to $\alpha$, then the desired result holds. Indeed, the triangle inequality yields
	\begin{equation}
	\begin{aligned}
		\E \left[\mu_{T,\epl}\left(\{a\colon \norm{a-\alpha}_2\geq c\}\mid X^\epl\right)\right] &\leq \E \left[\mu_{T,\epl}\left(\left\{a\colon \norm{a-m_{T,\epl}}_2\geq \frac{c}{2}\right\}\mid X^\epl\right)\right]\\
		&\quad + \mathbb P\left(\norm{m_{T,\epl} - \alpha}_2 \geq \frac{c}{2}\right).
	\end{aligned}
	\end{equation}
	If the mean converges in probability, then the second term vanishes. For the first term, Markov's inequality yields
	\begin{equation}
		\mu_{T,\epl}\left(\left\{a\colon \norm{a-m_{T,\epl}}_2\geq \frac{c}{2}\right\}\mid X^\epl\right) \leq \frac{4}{c^2}\int_{\R^N} \norm{a-m_{T,\epl}}_2^2 \, \mu_{T,\epl}(\d a\mid X^\epl),
	\end{equation}
	and a change of variable simply gives
	\begin{equation}
		\int_{\R^N} \norm{a-m_{T,\epl}}_2^2 \, \mu_{T,\epl}(\d a\mid X^\epl) = \trace(C_{T,\epl}).
	\end{equation}
	This proves that we just have to verify that the covariance matrix vanishes and that the mean tends to the coefficient $\alpha$. Let us first consider the covariance matrix. An algebraic identity yields
	\begin{equation}
	C_{T,\epl} = \frac{2\Sigma}{T} \left(M^{-1} - Q^{-1}\right),
	\end{equation}
	where 
	\begin{equation}
	Q = M + \frac{T}{2\Sigma} M C_0 M.
	\end{equation}
	Let us first remark that due to the hypothesis on $M$ (Assumption \ref{as:regularity}\ref{as:regularity_SPD}) and the ergodic theorem it holds for all $T > 0$
	\begin{equation}
		\norm{M^{-1}}_2 \leq \frac1{\bar \lambda},
	\end{equation}
	where $\bar \lambda$ is given in Assumption\ref{as:regularity}\ref{as:regularity_SPD}. We now have that for generic symmetric positive definite matrices $R$ and $S$ it holds
	\begin{equation}
		\norm{(R+S)^{-1}}_2 \leq \norm{S^{-1}}_2.
	\end{equation}
	Applying this inequality to $Q^{-1}$, we obtain
	\begin{equation}
		\norm{Q^{-1}}_2 \leq \frac{2\Sigma}T \norm{(MC_0M)^{-1}}_2 \leq \frac{2\Sigma}T \norm{M^{-1}}_2^2 \norm{C_0^{-1}}_2 = \frac{2\Sigma}{T\bar \lambda^2} \norm{C_0^{-1}}_2,
	\end{equation}
	which implies
	\begin{equation}
		\lim_{T \to \infty}\norm{Q^{-1}}_2 = 0,
	\end{equation}
	and due to the triangle inequality
	\begin{equation}\label{eq:CovarianceShrink}
		\lim_{T \to \infty}\norm{C_{T,\epl}}_2 = 0.
	\end{equation}
	We proved that in the limit for $T \to \infty$ the covariance shrinks to zero independently of $\epl$. We now consider the mean. First, we remark that the triangle inequality yields
	\begin{equation}
		\norm{m_{T,\epl} - \alpha}_2 \leq \norm{m_{T,\epl} - \widehat A(X^\epl, T)}_2 + \norm{\widehat A(X^\epl, T) - \alpha}_2.
	\end{equation}
	For the second term, Theorem \ref{thm:Bias} implies 
	\begin{equation}
		\lim_{\epl \to 0} \lim_{T \to \infty}\norm{\widehat A(X^\epl, T) - \alpha}_2 = 0, \quad \text{a.s.}
	\end{equation}
	Let us now consider the first term.	Replacing the expression of the maximum likelihood estimator \eqref{eq:MLE} and due to the Cauchy--Schwarz and triangle inequalities, we obtain
	\begin{equation}
	\begin{aligned}
		\norm{m_{T,\epl} - \widehat A(X^\epl, T)}_2 &= \frac{2\Sigma}T\norm{M^{-1}C_0^{-1}A_0 - Q^{-1}\left(C_0^{-1}A_0 - \frac{T}{2\Sigma} h \right)}_2\\
		&\leq \frac{2\Sigma}{T\bar \lambda} \norm{C_0^{-1}}_2 \left(\norm{A_0}_2 + \frac1{\bar \lambda}\norm{h}_2 + \frac{2\Sigma}{T\bar \lambda} \norm{C_0^{-1}}_2\norm{A_0}_2\right).
	\end{aligned}
	\end{equation}
	Moreover, the ergodic theorem and the strong law of large numbers for martingales guarantee that $\norm{h}_2$ is bounded a.s. for $T \to \infty$. Therefore
	\begin{equation}
	\lim_{T\to\infty} \norm{m_{T,\epl} - \widehat A(X^\epl, T)}_2 = 0, \quad \text{a.s.},
	\end{equation}
	independently of $\epl$. Finally, 
	\begin{equation}
		\lim_{\epl \to 0} \lim_{T \to \infty}\norm{m_{T,\epl} - \alpha}_2 = 0, \quad \text{a.s.},
	\end{equation}
	which, together with \eqref{eq:CovarianceShrink}, implies the desired result. 
\end{proof}

\subsection{The Filtered Data Approach}\label{sec:BayesianFilter}

In this section, we present how to correct the asymptotic biasedness of the posterior highlighted by Theorem \ref{thm:BayesianBias} employing filtered data. In the Bayesian setting, we consider the modified likelihood function
\begin{equation}
	\widetilde p(X^\epl \mid A) = \exp\left(-\frac{\widetilde I(X^\epl\mid A)}{2\Sigma} \right), 
\end{equation}
where 
\begin{equation}
\begin{aligned}
	\widetilde I(X^\epl \mid A) & = \int_0^T A \cdot V'(Z^\epl_t) \dd X^\epl_t + \frac12 \int_0^T \left(A \cdot V'(X^\epl_t)\right)^2 \dd t \\
	&= \widetilde h \cdot A + \frac12 A \cdot M A.
\end{aligned}
\end{equation}
Since $M$ is symmetric positive definite, the function $\widetilde p(X^\epl \mid A)$ is indeed a valid Gaussian likelihood function. We then obtain the modified posterior $\widetilde \mu_{T,\epl} = \mathcal N(\widetilde m_{T, \epl}, C_{T, \epl})$, whose parameters are given by
\begin{equation}
\begin{aligned}
	C_{T, \epl}^{-1} &= C_0^{-1} + \frac{T}{2\Sigma} M, \\
	C_{T,\epl}^{-1}\widetilde m_{T,\epl} &= C_0^{-1}A_0 - \frac{T}{2\Sigma} \widetilde h. 
\end{aligned}	
\end{equation}
Let us remark that the posterior $\widetilde \mu_{T,\epl}$ has the same covariance as $\mu_{T,\epl}$ given in \eqref{eq:ParamBayes} and that therefore it is indeed a valid Gaussian posterior distribution. Nevertheless, in order to employ the tool of convergence introduced in Theorem \ref{thm:BayesianBias}, we need to study the properties of the MLE based on the likelihood $\widetilde p(X^\epl \mid A)$, i.e., the quantity
\begin{equation}\label{eq:AHatMixedTilde}
	\widetilde A_k(X^\epl, T) = - M^{-1} \widetilde h.
\end{equation}
The following theorem guarantees the unbiasedness of this estimator under a condition on the parameter $\delta$ of the filter.
\begin{theorem}\label{thm:mainTheoremTilde} Let the assumptions of Theorem \ref{thm:mainTheorem_zeta} hold. Then, if $\delta = \epl^\zeta$, with $\zeta \in (0, 2)$, it holds
	\begin{equation}
	\lim_{\epl \to 0} \lim_{T \to \infty} \widetilde A_k(X^\epl, T) = A, \quad \text{in probability},
	\end{equation} 
	for $\widetilde A_k(X^\epl, T)$ defined in \eqref{eq:AHatMixedTilde}.
\end{theorem}
\begin{proof} We first consider the difference between the two estimators $\widetilde A_k(X^\epl, T)$ and $\widehat A_k(X^\epl, T)$. In particular, the ergodic theorem and an algebraic equality imply
	\begin{equation}
	\begin{aligned}
	\lim_{T \to \infty} \left(\widetilde A_k(X^\epl, T) - \widehat A_k(X^\epl, T)\right) &= \left(\mathcal M_\epl^{-1} - \widetilde{\mathcal M}_\epl^{-1}\right) \lim_{T \to \infty}\widetilde h \\
	&= -\mathcal M_\epl^{-1}\left(\mathcal M_\epl - \widetilde{\mathcal M}_\epl\right)  \widetilde{\mathcal M}_\epl^{-1} \lim_{T \to \infty} \widetilde h\\
	&= \mathcal M_\epl^{-1}\left(\mathcal M_\epl - \widetilde{\mathcal M}_\epl\right)\lim_{T \to \infty} \widehat A_k(X^\epl, T),
	\end{aligned}
	\end{equation}
	almost surely, where $\mathcal M_\epl$ and $\widetilde{\mathcal M}_\epl$ are defined in \eqref{eq:DefCalM} and \eqref{eq:DefCalMTilde}, respectively. Therefore, due to Assumption \ref{as:regularity} which allows controlling the norm of $\mathcal M_\epl^{-1}$ and due to Lemma \ref{lem:distanceMandTildeM} we have for a constant $C > 0$
	\begin{equation}\label{eq:DistATildeAHat}
	\lim_{T \to \infty} \norm{\widetilde A_k(X^\epl, T) - \widehat A_k(X^\epl, T)}_2 \leq C \left(\epl + \delta^{1/2} \right),
	\end{equation}
	where we remark that $\widehat A_k(X^\epl, T)$ has a bounded norm for $\epl$ sufficiently small due to Theorem \ref{thm:mainTheorem_zeta}. Now, the triangle inequality yields
	\begin{equation}
	\norm{\widetilde A_k(X^\epl, T) - A}_2 \leq \norm{\widetilde A_k(X^\epl, T) - \widehat A_k(X^\epl, T)}_2 + \norm{\widehat A_k(X^\epl, T) - A}_2.
	\end{equation}
	Therefore, due to Theorem \ref{thm:mainTheorem_zeta}, the inequality \eqref{eq:DistATildeAHat} and since $\delta = \epl^\zeta$, the desired result holds.
\end{proof}

\begin{remark} One could argue that we could have carried on the whole analysis for the estimator $\widetilde A_k(X^\epl, T)$ instead of the estimator $\widehat A_k(X^\epl, T)$. Nevertheless, the latter guarantees the strong result of almost sure convergence in case $\delta$ is independent of $\epl$, which is false for the former. Conversely, analysing the properties of the estimator $\widetilde A_k(X^\epl, T)$ is fundamental for the Bayesian setting, in which the matrix $\widetilde M$ cannot be employed as its symmetric part is not positive definite in general.
\end{remark}

In light of the proof of Theorem \ref{thm:BayesianBias}, Theorem \ref{thm:mainTheoremTilde} guarantees that the mean of the posterior distribution $\widetilde \mu_{T, \epl}$ converges to the drift coefficient of the homogenized equation. Since the covariance matrix is the same for $\mu_{T, \epl}$ and $\widetilde \mu_{T, \epl}$, it is possible to prove a positive convergence result for $\widetilde \mu_{T, \epl}$, which is given by the following Theorem.
\begin{theorem}\label{thm:Bayesian} Let the Assumptions of Theorem \ref{thm:mainTheoremTilde} hold. Then, the modified posterior measure $\widetilde\mu_{T,\epl}(\cdotp\mid X^\epl) = \mathcal N(\widetilde m_{T,\epl}, C_{T,\epl})$ satisfies
	\begin{equation}
		\lim_{\epl \to 0}\lim_{T \to \infty} \E \left[\widetilde \mu_{T,\epl}\left(\{a\colon \norm{a-A}_2\geq c\}\mid X^\epl\right)\right] = 0,
	\end{equation}
	where $\E$ denotes expectation with respect to the Wiener measure and $A$ is the drift coefficient of the homogenized equation \eqref{eq:SDE_HOM}.
\end{theorem}
\begin{proof} The proof follows from the proof of Theorem \ref{thm:BayesianBias} and from Theorem \ref{thm:mainTheoremTilde}.
\end{proof}

\section{Numerical Experiments}\label{sec:NumExp}

In this section we show numerical experiments confirming our theoretical findings and showcasing the potential of the filtered data approach to overcome model misspecification arising when multiscale data is used to fit homogenized models.

\begin{remark} In practice, we consider for numerical experiment the data to be in the form of a high-frequency discrete time series from the solution $X^\epl$ of \eqref{eq:SDE_MS}. Let $\tau > 0$ be the time step at which data is observed, and let $X^\epl \defeq (X^\epl_0, X^\epl_\tau, X^\epl_{2\tau}, \ldots)$. We then compute the estimator $\widehat A_k$ as
	\begin{equation}
		\widehat A_{k,\tau}(X^\epl, T) = - \widetilde M_\tau^{-1}(X^\epl) \widetilde h_\tau(X^\epl),
	\end{equation}
	where
	\begin{equation}
		\widetilde M_\tau(X^\epl) = \frac{\tau}{T} \sum_{j=0}^{n-1} V'(Z^\epl_{j\tau}) \otimes V'(X^\epl_{j\tau}) , \qquad \widetilde h_\tau(X^\epl) = \frac{1}{T} \sum_{j=0}^{n-1} V'(Z^\epl_{j\tau}) (X^\epl_{(j+1)\tau} - X^\epl_{j\tau}).
	\end{equation}
	We take in all experiments $\tau \ll \epl^2$, so that the discretization of the data has negligible effects and does not compromise the validity of our theoretical results.
\end{remark}

\subsection{Parameters of the Filter}\label{sec:Num_Param}

For the first preliminary experiments, we consider $N = 1$ and the quadratic potential $V(x) = x^2/2$. In this case, the solution of the homogenized equation is an Ornstein--Uhlenbeck process. Moreover, we set the the fast potential in the multiscale equation \eqref{eq:SDE_MS} as $p(y) = \cos(y)$. In all experiments, data is generated employing the Euler--Maruyama method with a fine time step.

\subsubsection{Verification of Theoretical Results}\label{sec:Num_Param1}

\begin{figure}[t]
	\centering
	\begin{tabular}{cc}
	\includegraphics[]{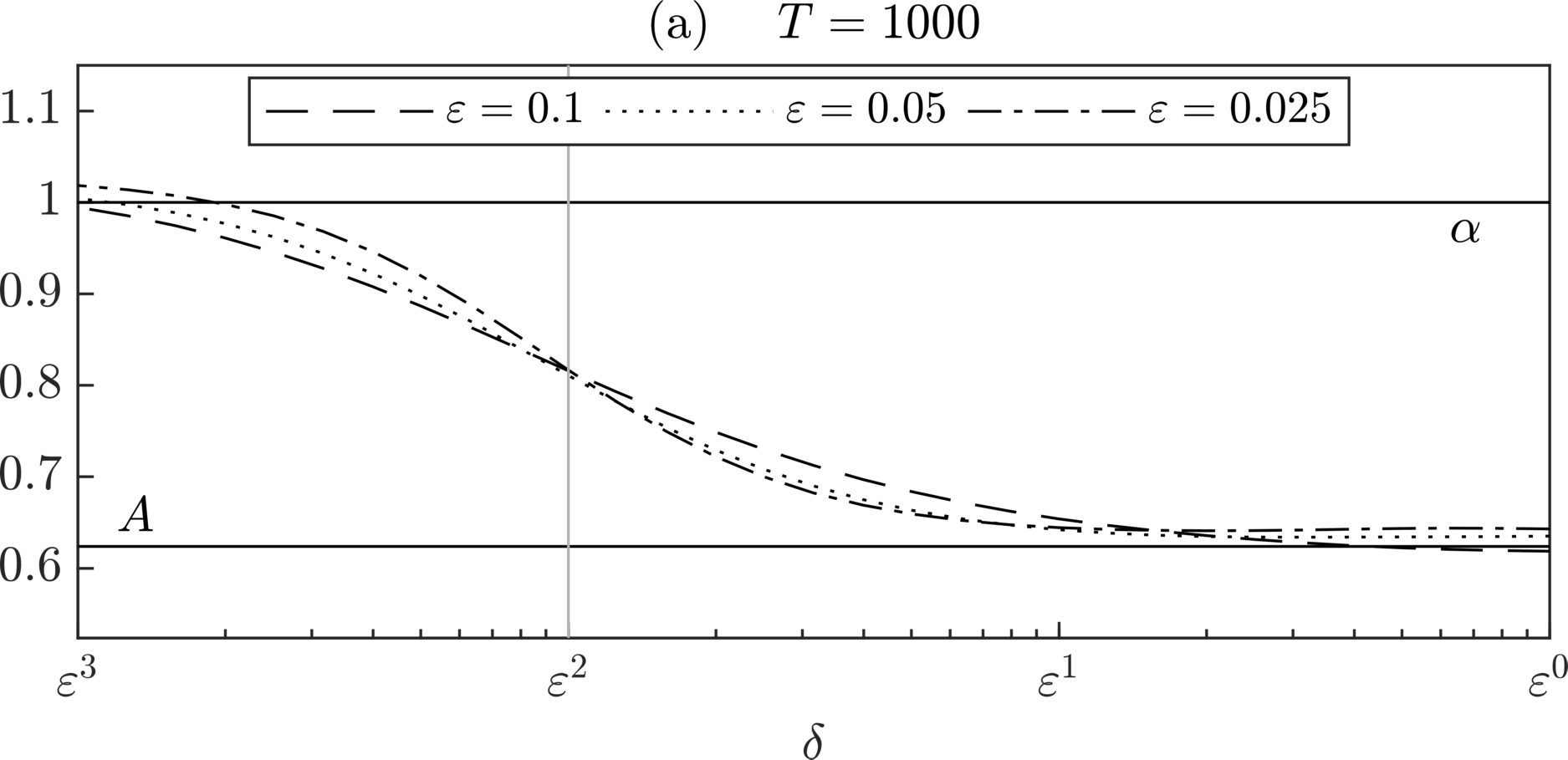} & \includegraphics[]{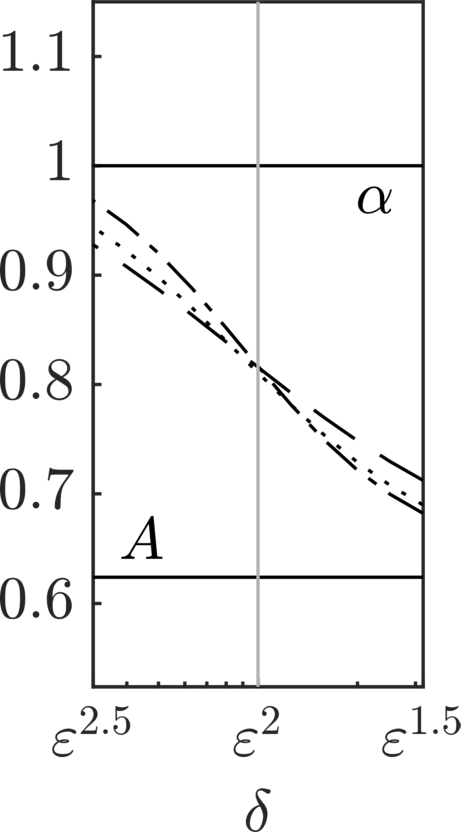} \\
	\addlinespace[0.5em]
	\includegraphics[]{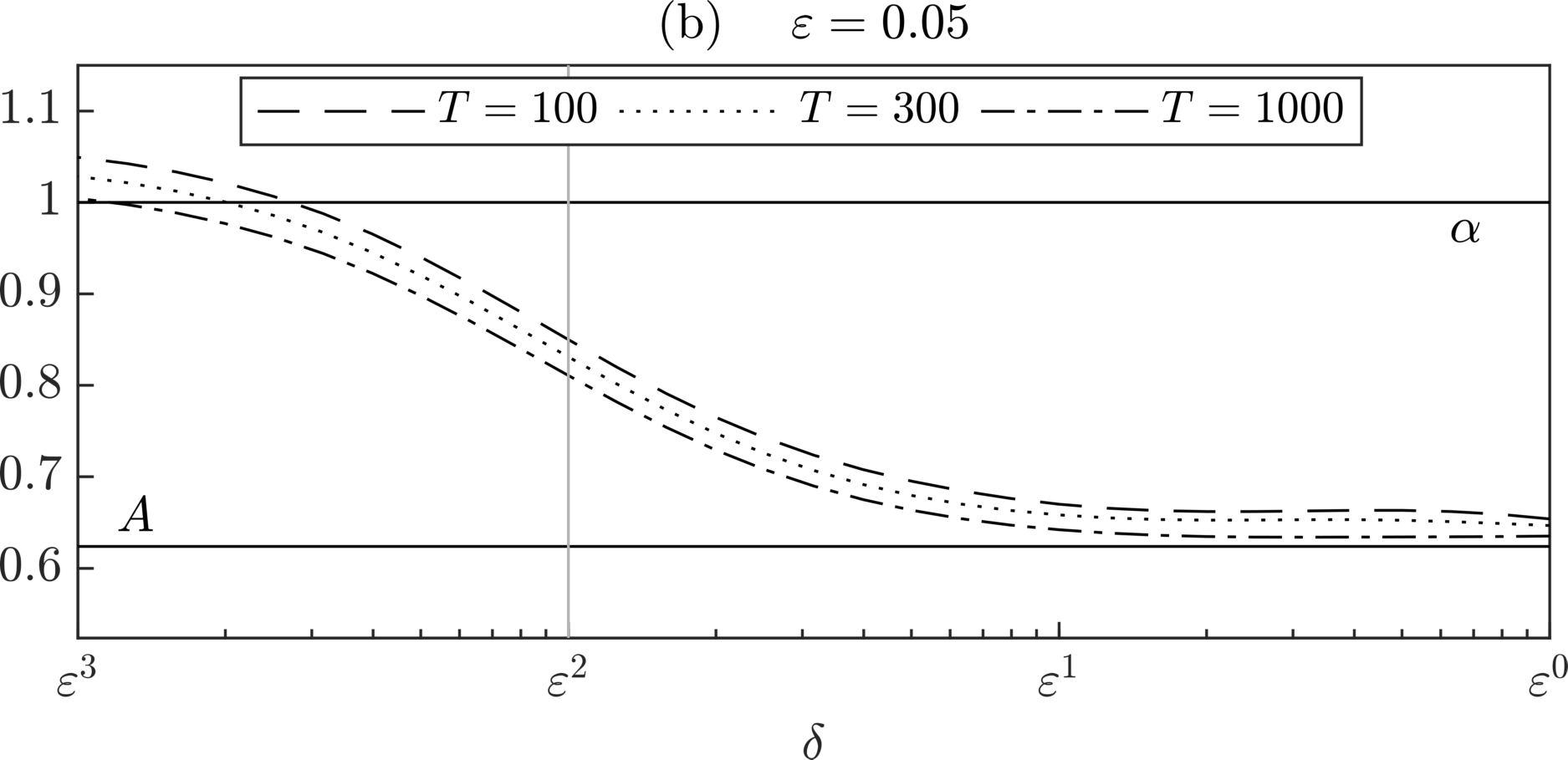} & \includegraphics[]{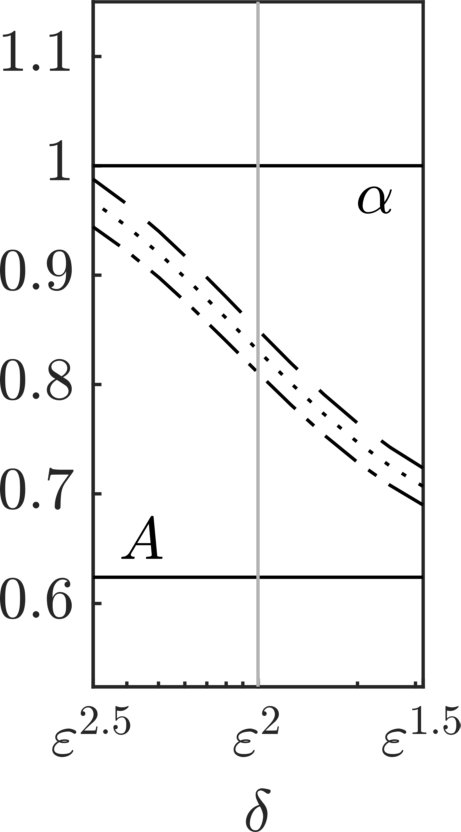}
	\end{tabular}
	\caption{Results for Section \ref{sec:Num_Param1}. On both figures, horizontal lines represent $\alpha$ and $A$, the drift coefficients of the unhomogenized and homogenized equations, and the grey vertical line represents the lower bound for the validity of Theorem \ref{thm:mainTheorem_zeta}. The curved lines (dashed, dotted and dash-dotted) represent on figure (a) the values of $\widehat A_k(X^\epl, T)$ for $\epl = \{0.1, 0.05, 0.025\}$, respectively, computed with $T = 10^3$. On figure (b), they correspond to the values of $\widehat A_k(X^\epl, T)$ at $T = \{100, 300, 1000\}$, respectively, computed with $\epl = 0.05$. We plot next to both figures (a) and (b) a zoom on a neighbourhood of $\epl^2$ to show the transition between the two regimes highlighted by the theoretical results. Note that the $\delta$-axis is in logarithmic scale and is normalized with respect to $\epl$.}
	\label{fig:TheoremVerification}
\end{figure}

We first demonstrate numerically the validity of Theorem \ref{thm:mainTheorem}, Theorem \ref{thm:mainTheorem_zeta} and Theorem \ref{thm:mainTheorem_zetaAlpha}, i.e., the unbiasedness of $\widehat A_k(X^\epl, T)$ for $\delta = \epl^\zeta$ with $\zeta \in [0, 2)$ and biasedness for $\zeta > 2$. Let us recall that for $\zeta = 0$ the analysis and the theoretical result are fundamentally different than for $\zeta \in (0, 2)$. We consider $\epl \in \{0.1, 0.05, 0.025\}$, the diffusion coefficient $\sigma = 1$ and generate data $X^\epl_t$ for $0 \leq t \leq T$ with $T = 10^3$. Then we filter the data by choosing $\delta = \epl^\zeta$, and $\zeta = 0, 0.1, 0.2,\ldots, 3$, and compute $\widehat A_k(X^\epl, T)$. Results are displayed in Figure \ref{fig:TheoremVerification}, and show that for $\zeta > 2$, i.e., $\delta = o(\epl^2)$, the estimator tends to the drift coefficient $\alpha$ of the unhomogenized equation. Conversely, as predicted by the theory, for $\zeta \in [0, 2)$ the estimator tends to $A$, the drift coefficient of the homogenized equation. Therefore, the point $\delta = \epl^2$ acts asymptotically as a switch between two completely different regimes, which is theoretically sharp in the limit for $T \to \infty$ and $\epl \to 0$. Let us remark that the results displayed in Figure \ref{fig:TheoremVerification}.(a) demonstrate that the transition occurs more rapidly for the smallest values of $\epl$. Moreover, in Figure \ref{fig:TheoremVerification}.(b), one can see how with bigger final times $T$ the estimator is closer both to $A$ when $\zeta \in [0, 2]$ and to $\alpha$ when $\zeta > 2$. Still, we observe that in finite computations the switch between $A$ and $\alpha$ is smoother than what we expect from the theory, which suggests to fix, if possible, $\delta = 1$.

\subsubsection{Comparison with Subsampling}\label{sec:Num_Param2}

\begin{figure}[t]
	\centering
	\begin{tabular}{m{1.2cm}m{4cm}m{4cm}m{4cm}}		
		\vspace{-1.2cm}$\sigma=0.5$ & \includegraphics[]{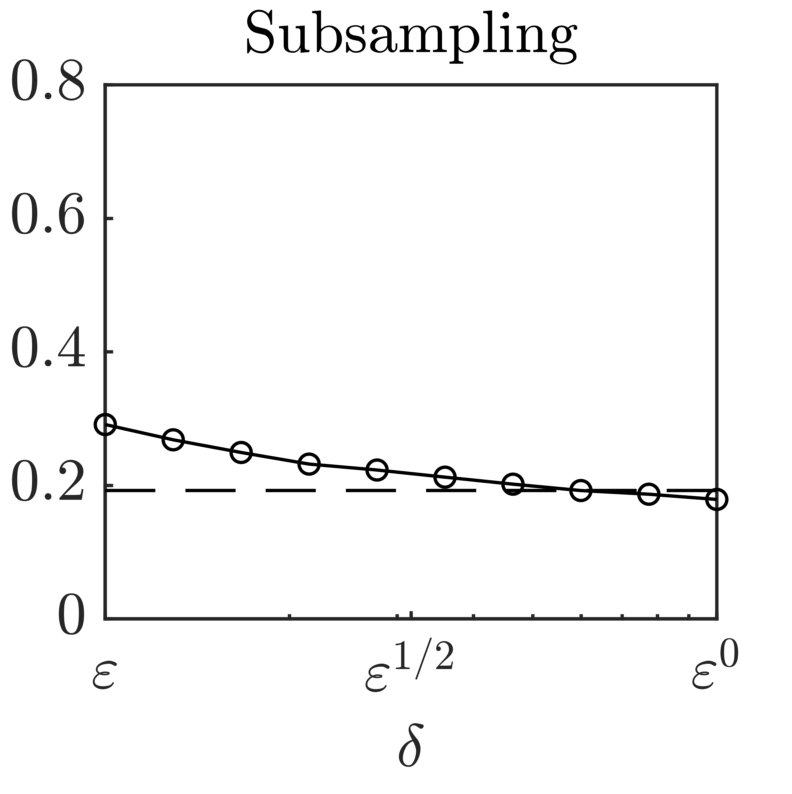} & \includegraphics[]{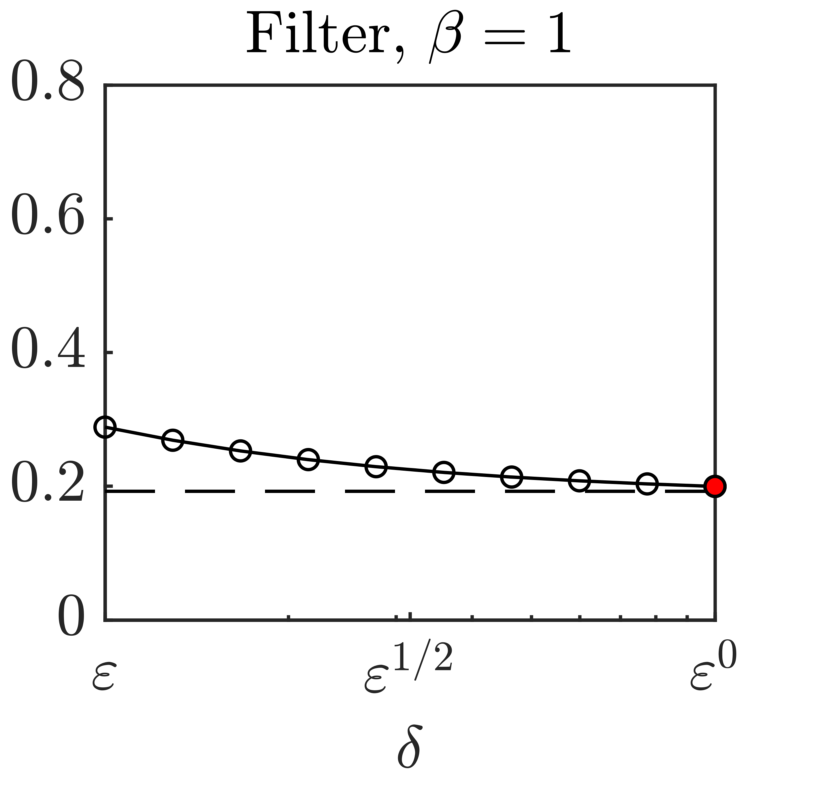}  & \includegraphics[]{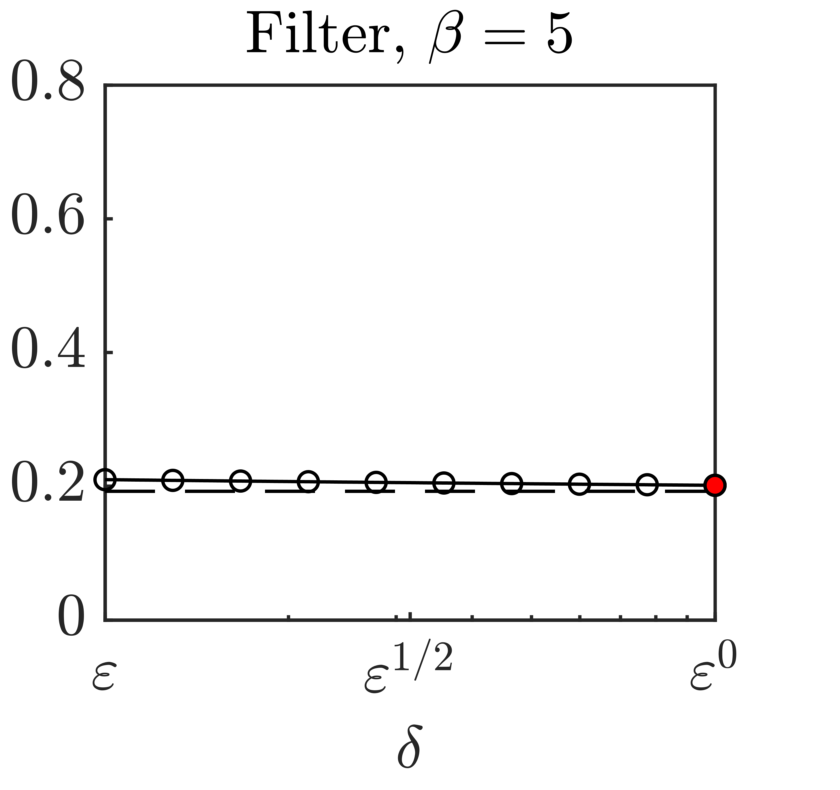} \\ 
		\addlinespace[-1em]
		\vspace{-1.2cm}$\sigma=0.7$ & \includegraphics[]{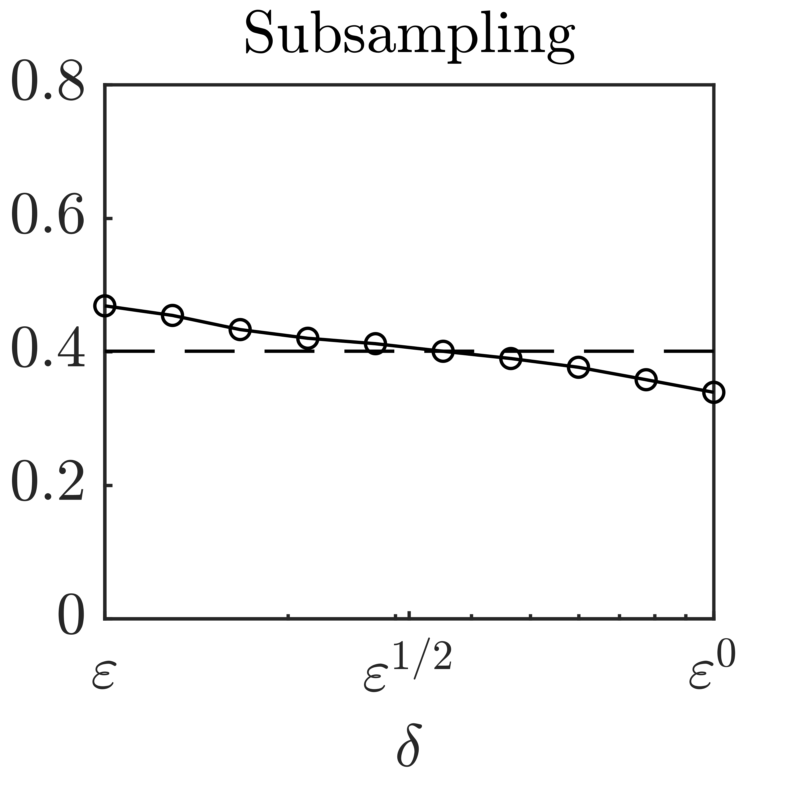} & \includegraphics[]{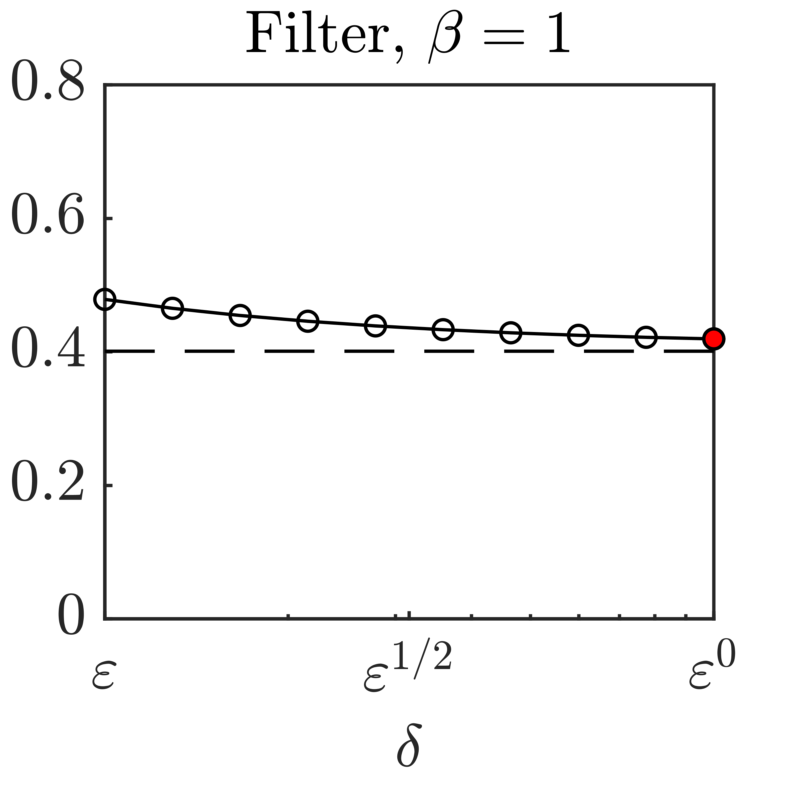}  & \includegraphics[]{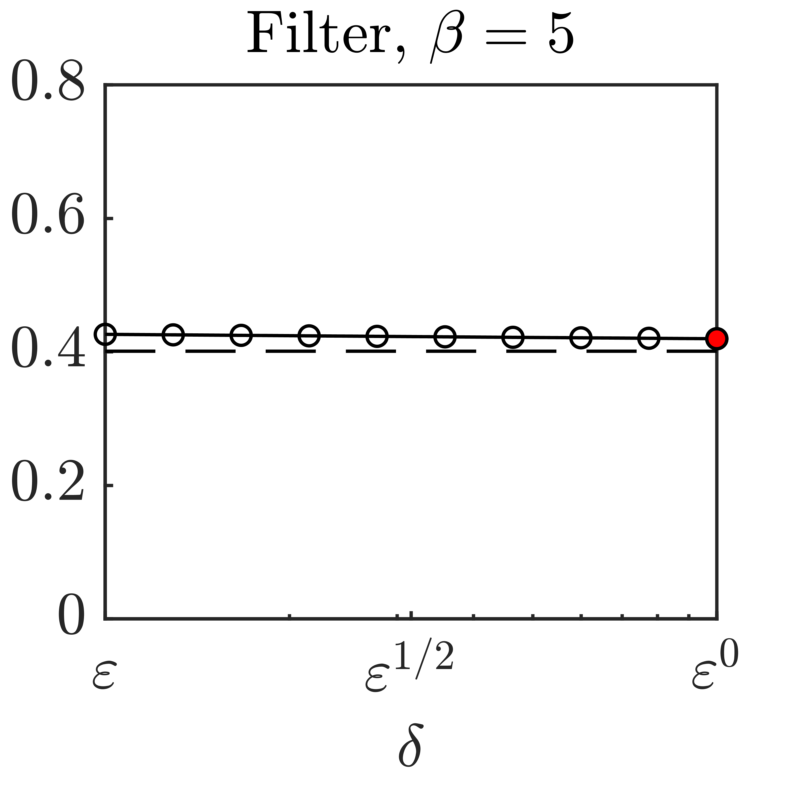} \\ 
		\addlinespace[-1em]
		\vspace{-1.2cm}$\sigma=1.0$ & \includegraphics[]{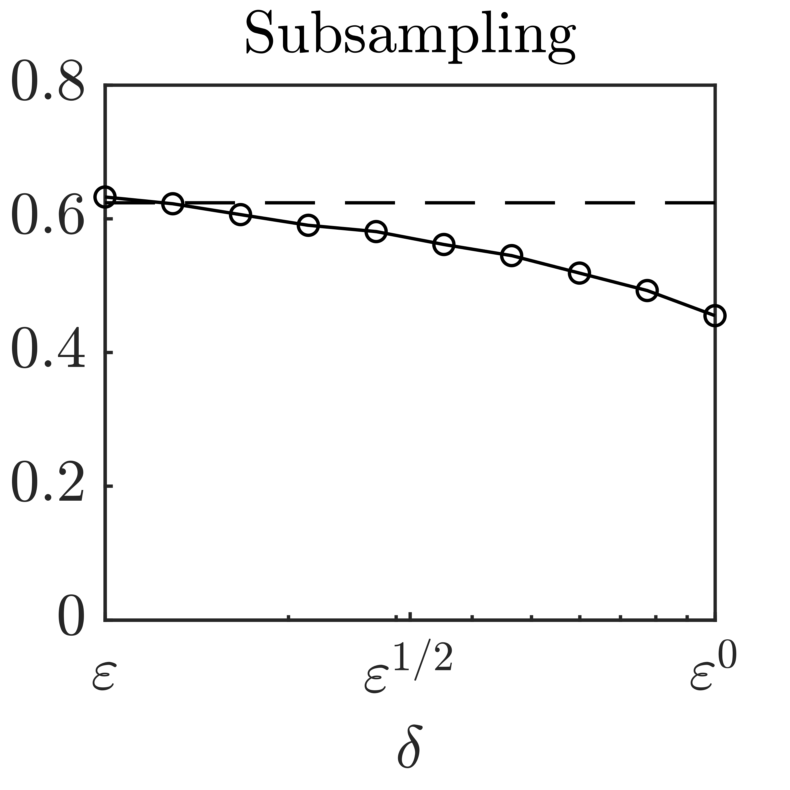} & \includegraphics[]{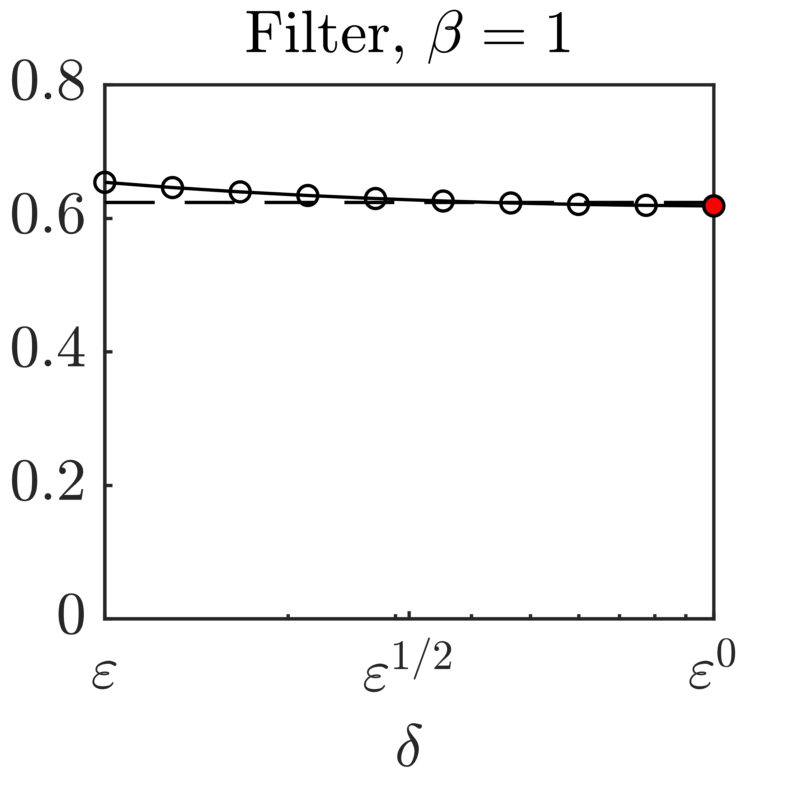}  & \includegraphics[]{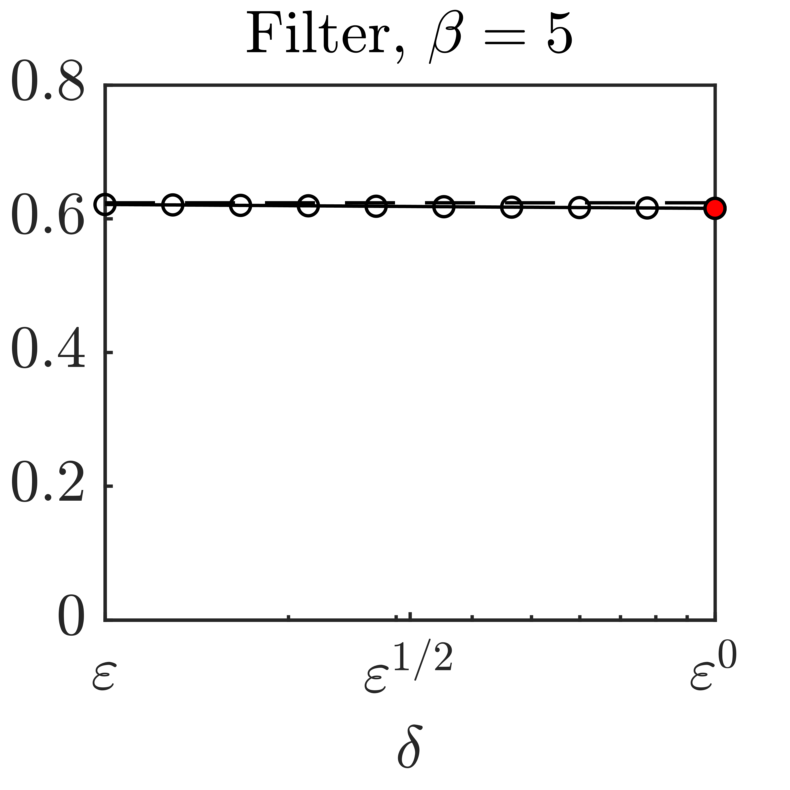}
	\end{tabular}
	\caption{Results for Section \ref{sec:Num_Param2}. The case of $\delta = 1$ is highlighted as a solid dot for the filtered data technique, as the analysis and theoretical result is different in this case. The three rows correspond to $\sigma = 0.5, 0.7, 1.0$ from top to bottom, and the dashed line corresponds to the true value of $A$.}
	\label{fig:OU}
\end{figure}

We now compare the results given by the filtered data technique with the results given by subsampling the data, i.e., the difference between the estimators $\widehat A_k(X^\epl, T)$ and $\widehat A_\delta(X^\epl, T)$. We fix the multiscale parameter $\epl = 0.1$ and generate data for $0 \leq t \leq T$ with $T = 10^3$. We choose $\delta = \epl^{\zeta}$ and vary $\zeta \in [0, 1]$, where $\delta$ is the filtering and the subsampling width, respectively. Moreover, for the filtered data approach we consider both $\beta = 1$ and $\beta = 5$. We report in Figure \ref{fig:OU} the experimental results. Let us remark that:
\begin{enumerate}
	\item for $\sigma = 0.5$ the results given by subsampling and by the filter with $\beta = 1$ are similar, while for higher values of $\sigma$ the filtered data approach seems better than subsampling;
	\item in general, choosing a higher value of $\beta$ seems beneficial for the quality of the estimator;
	\item the dependence on $\delta$ of numerical results given by the filter seems relevant only in case $\beta = 1$ and for small values of $\sigma$. For $\beta = 1$ and higher values of $\sigma$, the estimator is stable with respect to this parameter. This can be observed for a higher value of $\beta$ but we have no theoretical guarantee in this case.
\end{enumerate}

\subsubsection{The Influence of $\beta$}\label{sec:Num_Param3}
We finally test the variability of the estimator with respect to $\beta$ in \eqref{eq:filter}. We consider $\delta = \epl$, which corresponds to $\zeta = 1$ and seems to be the worst-case scenario for the filter, at least for $\beta = 1$. We consider again $\sigma = 0.5, 0.7, 1$ and vary $\beta = 1, 2, \ldots, 10$. Results, given in Figure \ref{fig:OUBeta}, show empirically that the estimator stabilizes fast with respect to $\beta$. Nevertheless, there is no theoretical guarantee supporting this empirical observation.

\begin{figure}[t]
	\centering
	\begin{tabular}{ccc}
		\includegraphics[]{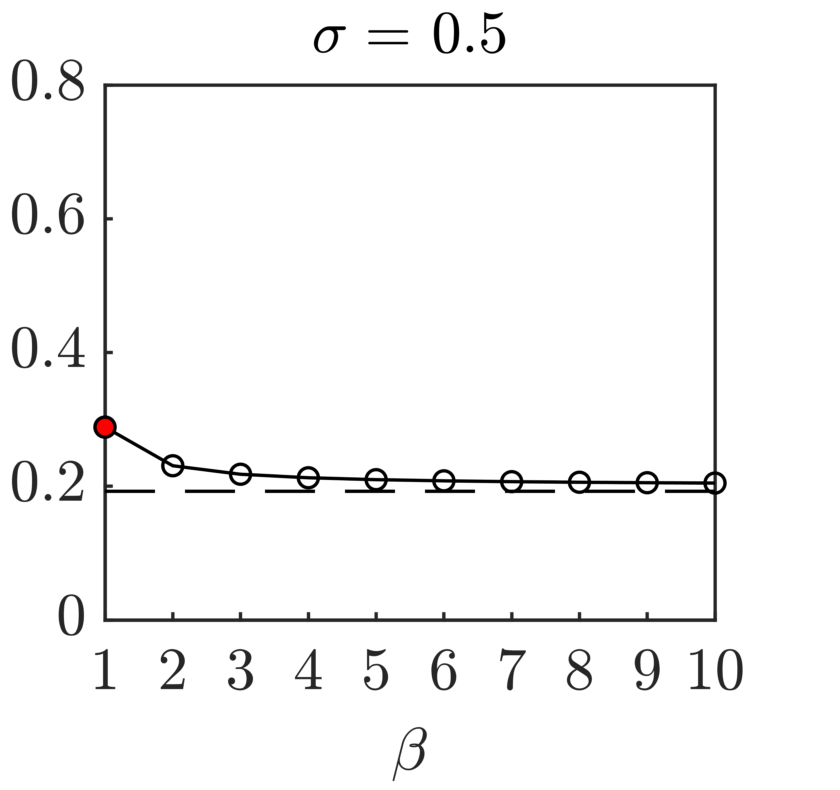} & \includegraphics[]{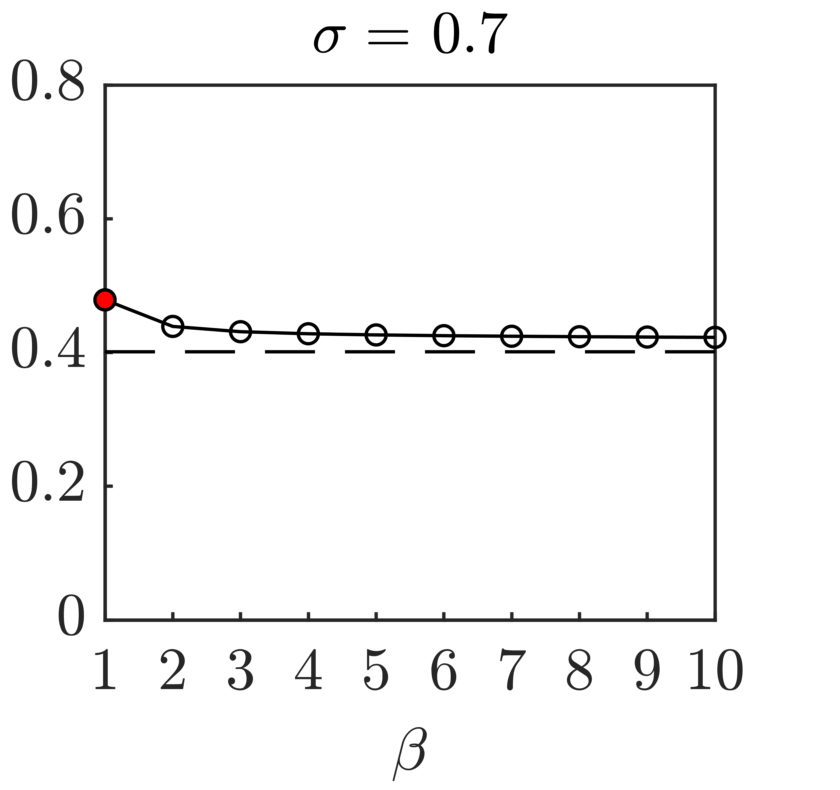}  & \includegraphics[]{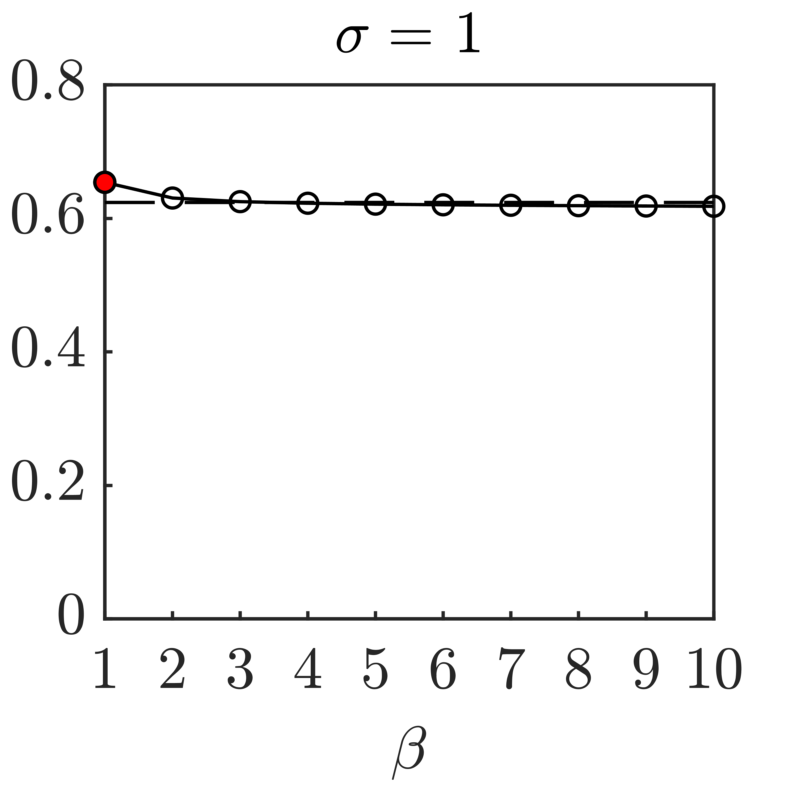} \\
	\end{tabular}
	\caption{Results for the estimator based on filter data with respect to the parameter $\beta$ (Section \ref{sec:Num_Param3}). The result for $\beta=1$, for which there are theoretical guarantees given by Theorem \ref{thm:mainTheorem_zeta}, is highlighted as a solid dot. From left to right we consider different values of $\sigma$, and the dashed line corresponds to the true value of $A$.}
	\label{fig:OUBeta}
\end{figure}

\begin{figure}[t]
	\centering
	\begin{tabular}{cc}
		\includegraphics[]{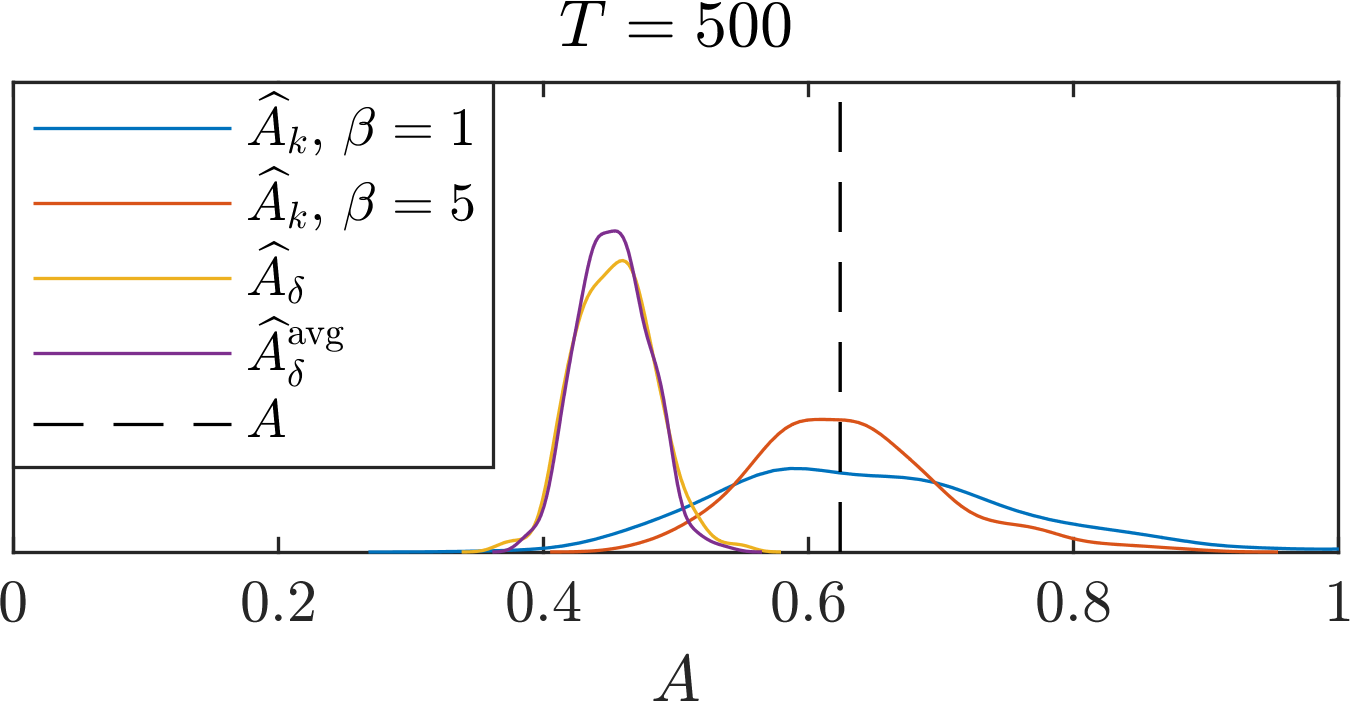} & \includegraphics[]{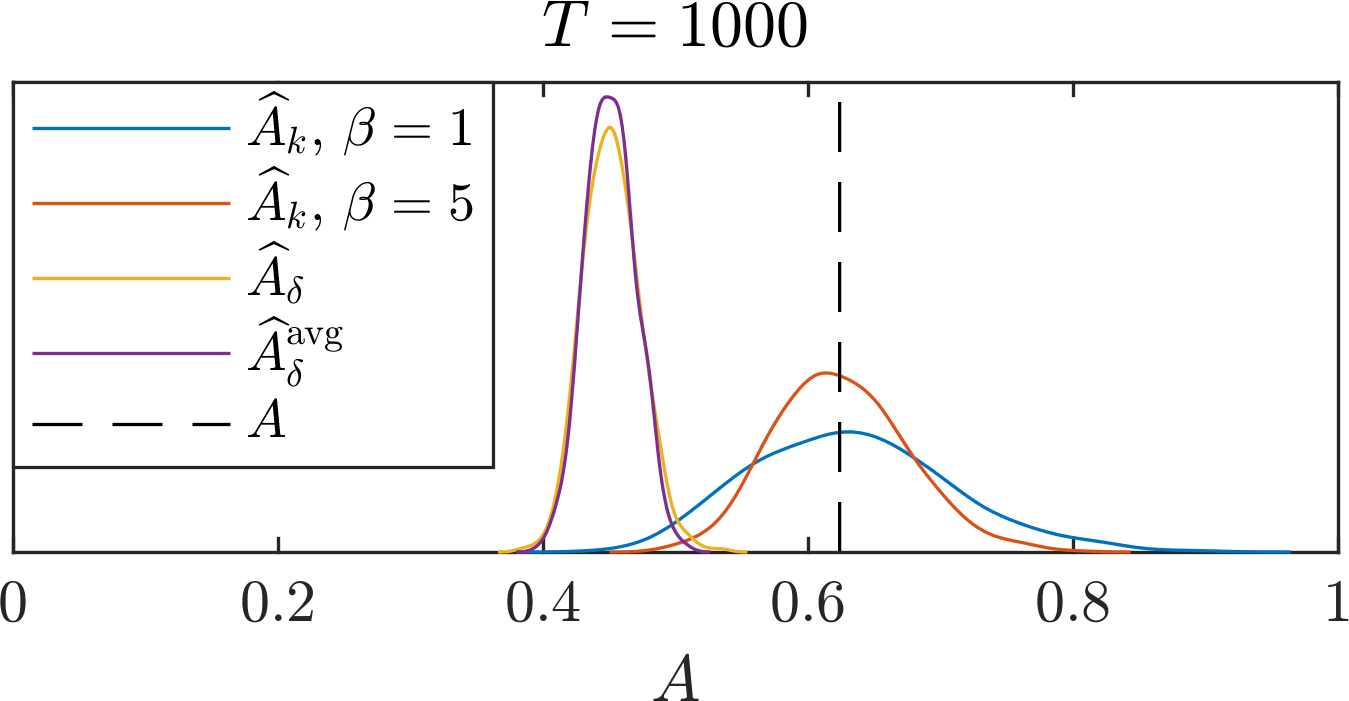}
	\end{tabular}
	\caption{Numerical results for \cref{sec:Num_Var}. Comparison between the density of the estimator of the drift based on filtered data with $\beta = \{1,5\}$, the estimator based on subsampling and the estimator based on shift-subsampling and averaging of \eqref{eq:SubShift}. On the left and on the right, the final time is $T = \{500, 1000\}$, respectively.}
	\label{fig:Variance}
\end{figure}

\subsection{Variance of the Estimators}\label{sec:Num_Var}

We now compare the estimators $\widehat A_k$ based on filtered data and $\widehat A_\delta$ based on subsampling in terms of variance. We consider for this experiment the SDE \eqref{eq:SDE_MS} with $N = 1$, the bistable potential $V(x) = x^4/4 - x^2/2$, the multiscale drift coefficient $\alpha = 1$, the diffusion coefficient $\sigma = 1$ and with $\epl = 0.1$. We then let $X^\epl = (X_t, 0\leq t\leq T)$ be the solution of \eqref{eq:SDE_MS} and generate $N_{\mathrm{s}} = 500$ i.i.d. samples of $X^\epl$. We then compute the estimators $\widehat A_k$ and $\widehat A_\delta$ on each of the realizations of $X^\epl$, thus obtaining $N_{\mathrm{s}}$ replicas $\{\widehat A_k^{(i)}\}_{i=1}^{N_{\mathrm{s}}}$ and $\{\widehat A_\delta^{(i)}\}_{i=1}^{N_{\mathrm{s}}}$. For the estimator $\widehat A_k$, we consider the kernel \eqref{eq:filter} with $\beta = \{1,5\}$ and with $\delta = 1$. For the estimator $\widehat A_\delta$, we employ the subsampling width $\delta = \epl^{2/3}$, which is heuristically optimal following \cite{PaS07}. It could be argued that another estimator based on subsampling and shifting could be employed to reduce the variance. In particular, we let $\tau > 0$ be the time step at which the data is observed. Indeed, in practice we work with high-frequency discrete data, and observe $X^\epl \defeq (X^\epl_0, X^\epl_\tau, \ldots, X^\epl_{n\tau})$, with $n\tau = T$. We assume for simplicity that the subsampling width $\delta$ is a multiple of $\tau$ and compute for all $k = 0, 1, \ldots, \delta/\tau-1$ 
\begin{equation}\label{eq:SubShift_Single}
	\widehat A_{\delta,k}(X^\epl,T) = - \frac{\sum_{j=0}^{n-1} V'(X^\epl_{j\delta + k})
		(X^\epl_{(j+1)\delta + k}-X^\epl_{j\delta + k}) }{\delta \sum_{j=0}^{n-1}
		V'(X^\epl_{j\delta + k})^2},
\end{equation}
i.e. the subsampling estimator obtained by shifting the origin by $k\tau$. We then average over the index $k$ and obtain the new estimator
\begin{equation}\label{eq:SubShift}
	\widehat A_{\delta}^{\mathrm{avg}}(X^\epl,T) = \frac{\tau}{\delta} \sum_{k=0}^{\delta/\tau - 1}\widehat A_{\delta,k}(X^\epl,T).
\end{equation}
We include this estimator in the numerical study for completeness, and compute $N_{\mathrm{s}}$ replicas of $\widehat A_{\delta}^{\mathrm{avg}}$ on all the realizations of $X^\epl$. Results, given in Figure \ref{fig:Variance} for the final times $T = \{500, 1000\}$, show that our novel approach does not outperform subsampling in terms of variance, but clearly does in terms of bias. Moreover, we notice numerically that the shifted-averaged estimator $\widehat A_\delta^{\mathrm{avg}}$ does not reduce sensibly the variance in this case with respect to $\widehat A_\delta$. In fact, this is only partly surprising, since the estimators $\widehat A_{\delta,k}$ of \eqref{eq:SubShift_Single} are highly correlated. Finally, we notice that the filtering estimator $\widehat A_k$ with $\beta = 5$ has a lower variance with respect to the same estimator with $\beta = 1$. This confirms that choosing a higher value of $\beta$ improves the estimation of the effective drift coefficient.

\subsection{Multidimensional Drift Coefficient}\label{sec:Num_Multi}

\begin{figure}[t]
	\centering
	\begin{tabular}{ccc}
		\includegraphics[]{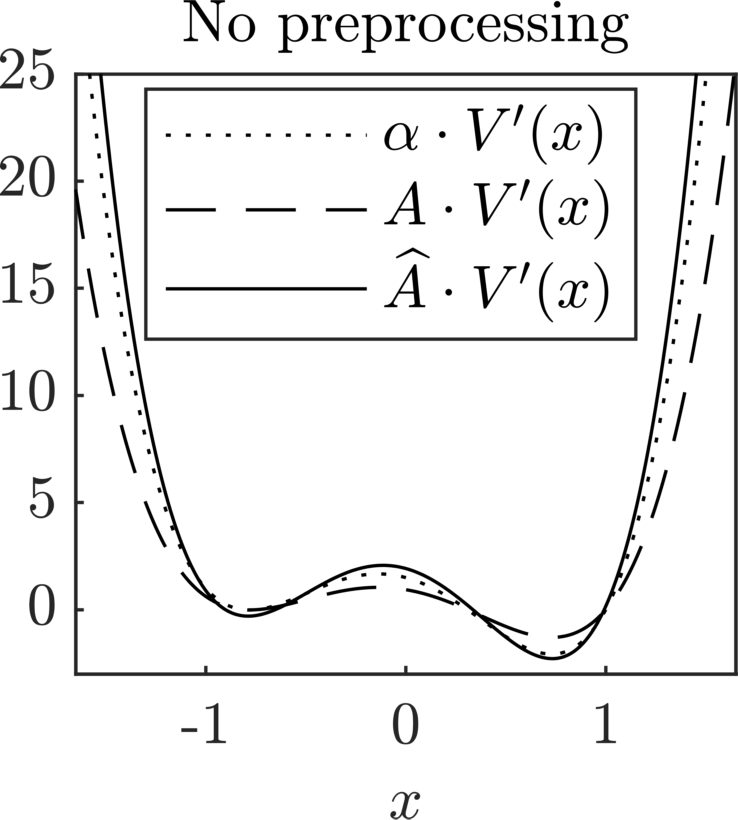} & \includegraphics[]{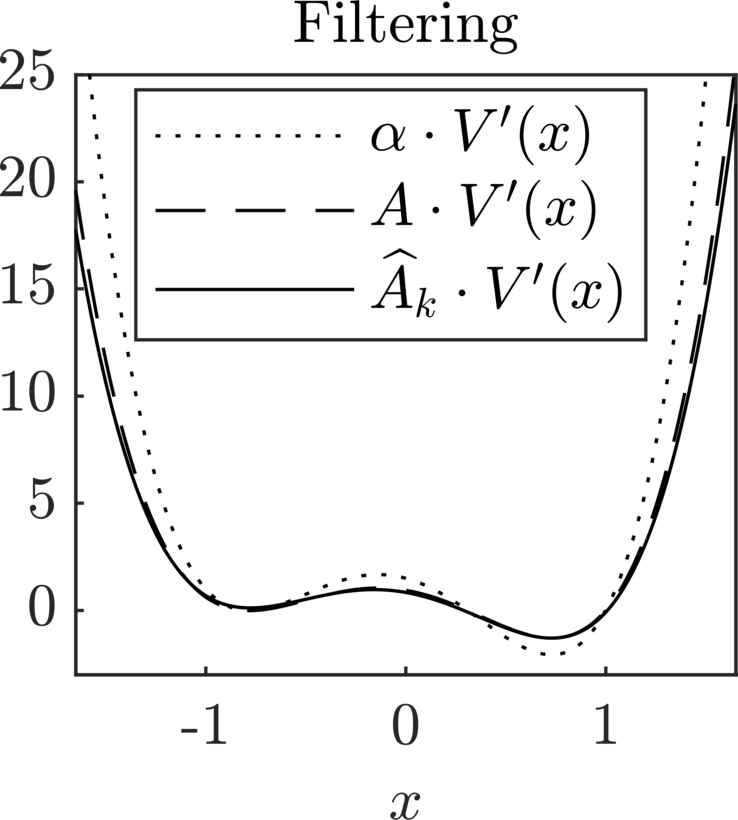}  & \includegraphics[]{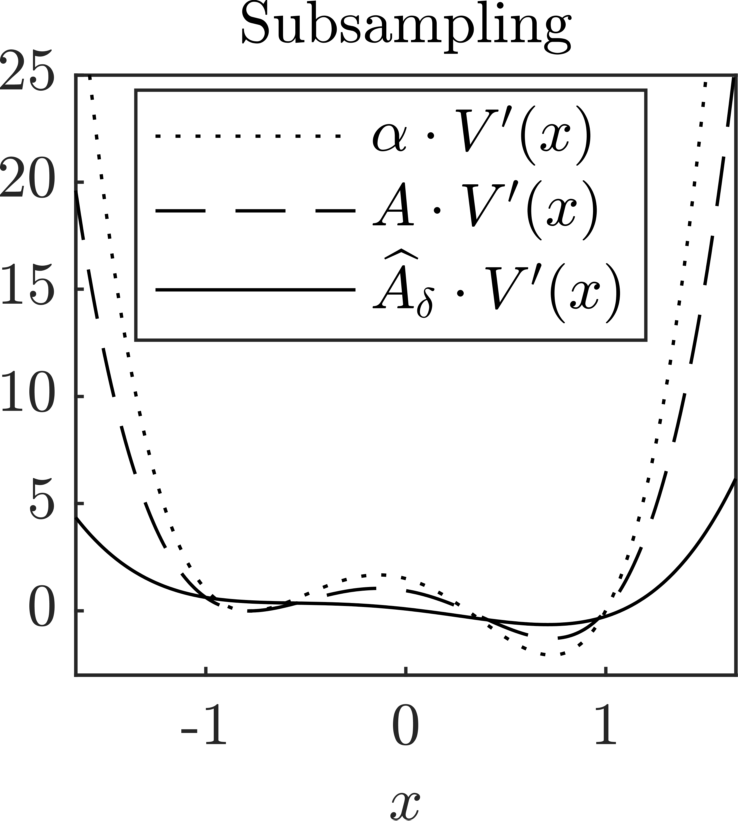}
	\end{tabular}

	\begin{tabular}{cccccc}
		\toprule
		Coefficient & Multiscale & Homogenized & No preprocessing  & Filtering  		  & Subsampling  			 \\ 
		& $\alpha$ & $A$   & $\widehat A$ & $\widehat A_k$ & $\widehat A_\delta$ \\
		\midrule
		$1$ &-1   & -0.62 & -0.92 & -0.70 & -0.59\\
		$2$ &-0.5 & -0.31 & -0.70 & -0.27 & 0.05 \\
		$3$ & 0.5 & 0.31  & 0.55  & 0.31  & 0.14 \\
		$4$ & 1   & 0.62  & 1.22  & 0.57  & 0.13 \\
		\bottomrule
	\end{tabular}
	\caption{Results for Section \ref{sec:Num_Multi}. In the figure, from left to right the potential function estimated with the data itself, the filter, subsampled data. In the table, numerical results for the single components of the true and estimated drift coefficients.}
	\label{fig:KLStyle}
\end{figure}

Let us consider the Chebyshev polynomials of the first kind, i.e., the polynomials $T_i\colon \R \to \R$, $i=0, 1, \ldots$, defined by the recurrence relation
\begin{equation}
T_0(x) = 1, \quad T_1(x) = x, \quad T_{i+1}(x) = 2xT_i(x) - T_{i-1}(x).
\end{equation}
We consider the potential function $V(x)$ as in \eqref{eq:Potential} with
\begin{equation}
V_i(x) = T_i(x), \quad i =1, \ldots, 4,
\end{equation}
thus considering the semi-parametric framework of Remark \ref{rem:SemiParam}. This potential function satisfies Assumption \ref{as:regularity} whenever $N$ is even and if the leading coefficient $\alpha_N$ is positive. We set $N = 4$ and the drift coefficient $\alpha = (-1, -1/2, 1/2, 1)$. With this drift coefficient, the potential function is of the bistable kind. Moreover, we set $\epl = 0.05$, the diffusion coefficient $\sigma = 1$, the fast potential $p(y) = \cos(y)$ and simulate a trajectory of $X^\epl$ for $0 \leq t \leq T$ with $T = 10^3$ employing the Euler--Maruyama method with time step $\Delta t = \epl^3$. We estimate the drift coefficient $A \in \R^4$ with the estimators:
\begin{enumerate}
	\item $\widehat A(X^\epl, T)$ based on the data $X^\epl$ itself;
	\item $\widehat A_\delta(X^\epl, T)$ based on subsampled data with subsampling parameter $\delta = \epl^{2/3}$;
	\item $\widehat A_k(X^\epl, T)$ based on filtered data $Z^\epl$ computed with $\beta = 1$ and $\delta = 1$.
\end{enumerate}
In particular, we pick this specific value of $\delta$ for the subsampling following the optimality criterion given in \cite{PaS07}. Results, given in Figure \ref{fig:KLStyle}, show that the filter-based estimation captures well the homogenized potential as well as the coefficient $A$. Moreover, it is possible to remark the negative result given by Theorem \ref{thm:Bias} holds in practice, i.e., with no pre-processing the estimator $\widehat A(X^\epl, T)$ tends to the drift coefficient $\alpha$ of the unhomogenized equation. Finally, we can observe that the subsampling-based estimator fails to capture the homogenized coefficients. Indeed, the estimator strongly depends on the sampling rate and on the diffusion coefficient, as shown in the numerical experiments of \cite{PaS07}. Even though the authors suggest the choice of $\delta = \epl^{2/3}$, this is just an heuristic and is not guaranteed to be the optimal value in all cases. In the asymptotic limit of $\epl \to 0$ and $T \to \infty$, any valid choice of the subsampling rate is guaranteed theoretically to work, but not in the pre-asymptotic regime. Our estimator, conversely, seems to perform better with no particular tuning of the parameters even in this multi-dimensional case, which demonstrates the robustness of our novel approach.

\subsection{The Bayesian Approach: Bistable Potential} \label{sec:Num_Bayes}
\begin{figure}[t]
	\centering
	\begin{tabular}{ccc}
		\includegraphics[]{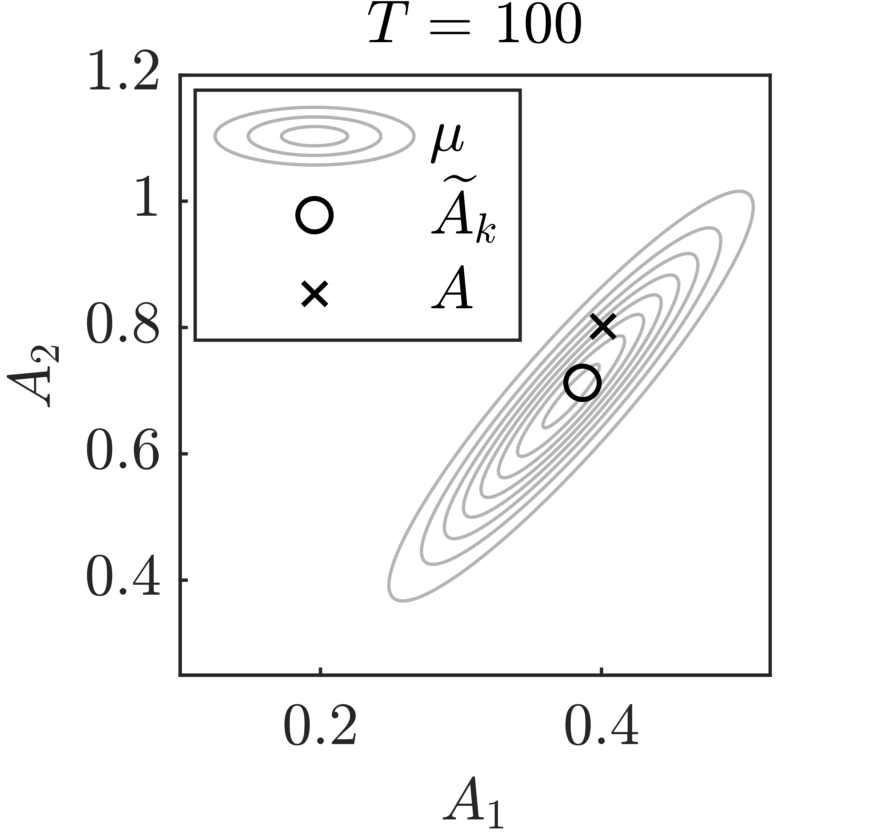} & \includegraphics[]{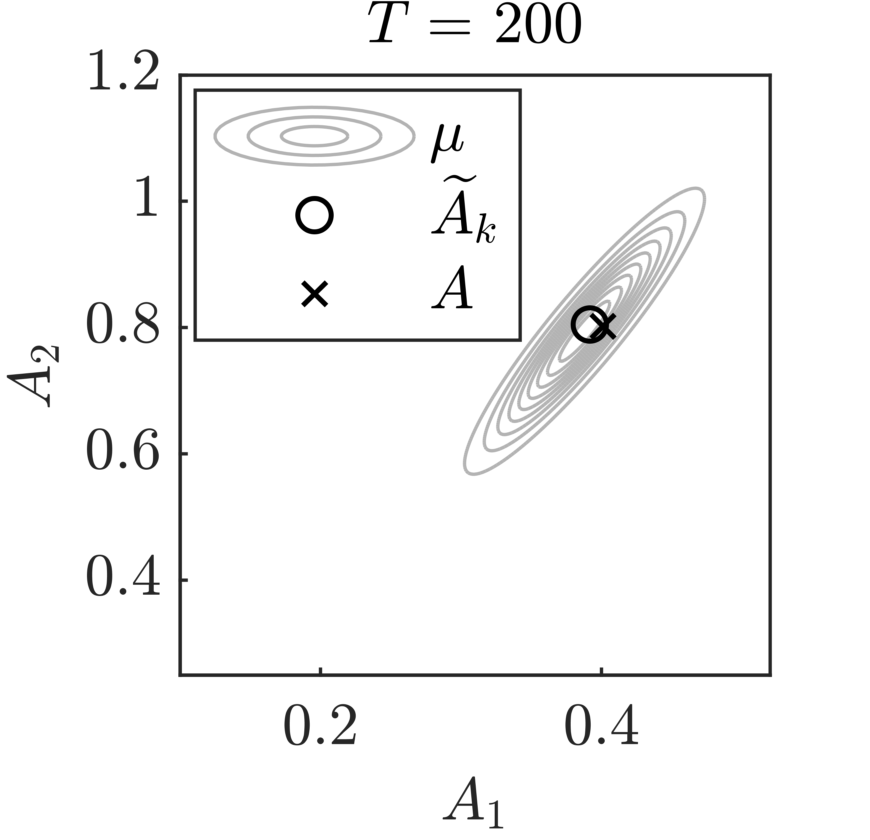}  & \includegraphics[]{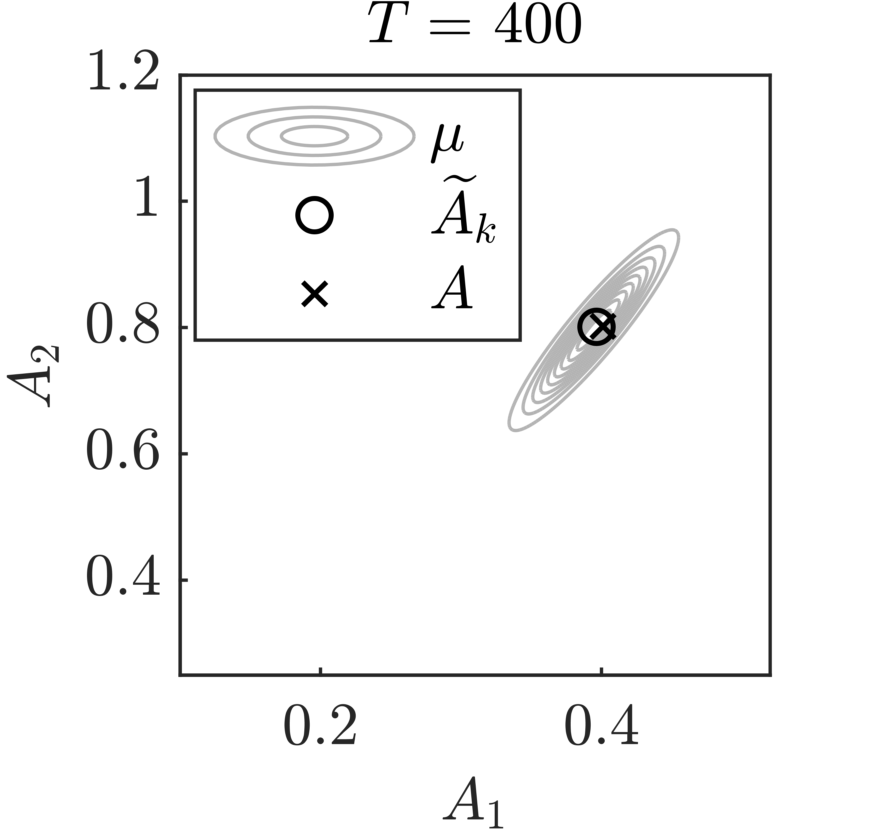} \\
	\end{tabular}
	\caption{Results for Section \ref{sec:Num_Bayes}. Posterior distributions over the parameter $A = (A_1, A_2)^\top$ for the bistable potential obtained with the filtered data approach. The figures refer to final time $T = 100, 200, 400$ from left to right, respectively. The MLE $\widetilde A_k(X^\epl, t)$ is represented with a circle, while the true value $A$ of the drift coefficient of the homogenized equation is represented with a cross.}
	\label{fig:Bayes}
\end{figure}

In this numerical experiment we consider $N = 2$ and the bistable potential, i.e., the function $V$ defined as
\begin{equation}
	V(x) = \begin{pmatrix} \dfrac{x^4}{4} & -\dfrac{x^2}{2} \end{pmatrix}^\top,
\end{equation}
with coefficients $\alpha_1 = 1$ and $\alpha_2 = 2$. We then consider the multiscale equation with $\sigma = 0.7$, the fast potential $p(y) = \cos(y)$ and $\epl = 0.05$, thus simulating a trajectory $X^\epl$. We adopt here a Bayesian approach and compute the posterior distribution $\widetilde \mu_{T, \epl}$ obtained with the filtered data approach introduced in Section \ref{sec:BayesianFilter}. The parameters of the filter are set to $\beta = 1$ and $\delta = \epl$ in \eqref{eq:filter}. Moreover, we choose the non-informative prior $\mu_0 = \mathcal N(0,I)$. Let us remark that in order to compute the posterior covariance the diffusion coefficient $\Sigma$ of the homogenized equation has to be known. In this case, we pre-compute the value of $\Sigma$ via the coefficient $K$ and the theory of homogenization, but notice that $\Sigma$ could be estimated either employing the subsampling technique of \cite{PaS07} or using the estimator $\widehat \Sigma_k$ based on filtered data defined in \eqref{eq:SigmaHat}. In particular, in this case $\Sigma \approx 0.2807$, and we compute numerically
\begin{equation}
	\widehat \Sigma_k(X^\epl, 100) = 0.2901, \quad \widehat \Sigma_k(X^\epl, 200) = 0.2835, \quad \widehat \Sigma_k(X^\epl, 400) = 0.2813,
\end{equation}
so that employing the estimator $\widehat \Sigma_k$ instead of the true value would have negligible effects on the computation of the posterior over the effective drift coefficient. We stop computations at times $T = \{100, 200, 400\}$ in order to observe the shrinkage of the Gaussian posterior towards the MLE $\widetilde A_k(X^\epl, T)$ with respect to time. In Figure \ref{fig:Bayes}, we observe that the posterior does indeed shrink towards the MLE, which in turn gets progressively closer to the true value of the drift coefficient $A$ of the homogenized equation.

\section{Conclusion}

In this work we considered a novel methodology to confront the
problem of model misspecification when homogenized models are fit to multiscale data. Our approach is based on using filtered data for the estimation of the drift of the homogenized diffusion process. We proved asymptotic unbiasedness of estimators drawn from our methodology. Moreover, we found a modified Bayesian approach which guarantees robust uncertainty quantification and posterior contraction, based on the same filtered data approach. Numerical experiments demonstrate how the estimator based on filtered data requires less knowledge of the characteristic time-scales of the multiscale equation with respect to subsampling, and how it can be employed as a black-box tool for parameter estimation on a range of academic examples. We note that in many applications one can only obtain discrete measurements of the diffusion process. Recently, using the filtering approach developed in this paper and martingale estimating functions a new estimator for learning homogenised SDEs from noisy discrete data has been introduced \cite{APZ21}. We believe this work gives way to several further developments. In particular, we believe it would be relevant to 
\begin{enumerate}
	\item analyse the filtered data approach for $\beta > 1$ in \eqref{eq:filter}, which seems to give more robust results in practice,
	\item extend the analysis to the non-parametric framework most likely by means of Bayesian regularization techniques, thus allowing to recover effective drift functions for which a parametric representation does not exist,
	\item consider multiscale models for which the homogenized equation presents multiplicative noise,
	\item test the filtered data methodology against real-world data,
	\item apply similar methodologies to correct faulty behaviour of other methods.
\end{enumerate} 

\subsection*{Acknowledgements} 

AA, AZ and GG are partially supported by the Swiss National Science Foundation, under grant No. 200020\_172710. The work of GAP was partially funded by the EPSRC, grant number EP/P031587/1, and by JPMorgan Chase \& Co. Any views or opinions expressed herein are solely those of the authors listed, and may differ from the views and opinions expressed by JPMorgan Chase \& Co. or its affiliates. This material is not a product of the Research Department of J.P. Morgan Securities LLC. This material does not constitute a solicitation or offer in any jurisdiction. AMS is grateful to NSF (grant DMS 18189770) for financial support.

\begin{appendices}
	
\section{Proofs of Sections \ref{sec:ergodic}}\label{ap:ProofsErgodic}

\begin{proof}[Proof of Lemma \ref{lem:density}] We have to show that the joint process solution to \eqref{eq:systemSDE} is hypo-elliptic. Denoting as $f\colon \R \to \R$ the function 
	\begin{equation}
	f(x) = -\alpha \cdot V'(x) - \frac1\epl p'\left(\frac{x}\epl\right),
	\end{equation}
	the generator of the process $(X^\epl, Z^\epl)^\top$ is given by
	\begin{equation}
	\mathcal L = f \partial_x + \sigma \partial_{xx}^2 + \frac1\delta (x - z)\partial_z \eqdef \mathcal X_0 + \sigma \mathcal X_1^2, 
	\end{equation}
	where 
	\begin{equation}
	\mathcal X_0 = f \partial_x + \frac1\delta (x - z)\partial_z, \quad \mathcal X_1 = \partial_x.
	\end{equation}
	The commutator $[\mathcal X_0, \mathcal X_1]$ applied to a test function $v$ then gives
	\begin{equation}
	\begin{aligned}
	[\mathcal X_0, \mathcal X_1]v &= f \partial_x^2 v + \frac1\delta (x - z) \partial_x \partial_z v  - \partial_x\left(f\partial_x v + \frac1\delta (x - z)\partial_z v\right)\\
	&= -\partial_x f \partial_x v - \frac1\delta \partial_z v.
	\end{aligned}
	\end{equation}
	Consequently, 
	\begin{equation}
	\mathrm{Lie}\left(\mathcal X_1, [\mathcal X_0, \mathcal X_1]\right) = \mathrm{Lie} \left(\partial_x, -\partial_x f \partial_x - \frac1\delta\partial_z\right),
	\end{equation}
	which spans the tangent space of $\R^2$ at $(x, z)$, denoted $T_{x, z}\R^2$. The desired result then follows from Hörmander's theorem (see e.g. \cite[Chapter 6]{Pav14}).
\end{proof}

\begin{proof}[Proof of Lemma \ref{lem:ergodicity}] Lemma \ref{lem:density} guarantees that the Fokker--Planck equation can be written directly from the system \eqref{eq:systemSDE}. For geometric ergodicity, let 
	\begin{equation}
	\mathcal S(x,z) \defeq \begin{pmatrix} -\alpha \cdot V'(x) - \frac1\epl p'(\frac{x}{\epl}) \\ \frac1\delta (x - z) \end{pmatrix} \cdot \begin{pmatrix} x \\ z \end{pmatrix} =  -\left(\alpha \cdot V'(x) + \frac1\epl p'\left(\frac{x}{\epl}\right)\right)x + \frac1\delta(xz-z^2).
	\end{equation}
	Due to Assumption \ref{as:regularity}\ref{as:regularity_diss}, Remark \ref{rem:regularity_diss} and Young's inequality, we then have for all $\gamma > 0$
	\begin{equation}
	\mathcal S(x,z) \leq a + \left(\frac{1}{2\gamma\delta} - b\right)x^2 + \frac1\delta\left(\frac{\gamma}{2} - 1\right)z^2.
	\end{equation}
	We choose $\gamma = \gamma^* \defeq 1 - b\delta + \sqrt{1 + (1 - b\delta)^2} > 0$ so that
	\begin{equation}
	C(\gamma^*) \defeq -\frac{1}{2\gamma^*\delta} + b = -\frac1\delta\left(\frac{\gamma^*}{2} - 1\right),
	\end{equation}
	and we notice that $C(\gamma^*) > 0$ if $\delta > 1/(4b)$. In this case, we have
	\begin{equation}
	\mathcal S(x, z) \leq a - C(\gamma^*) \norm{\begin{pmatrix} x & z \end{pmatrix}^\top}^2,
	\end{equation}
	and problem \eqref{eq:systemSDE} is dissipative. It remains to prove the irreducibility condition \cite[Condition 4.3]{MSH02}. We remark that the system \eqref{eq:systemSDE} fits the framework of the example the end of \cite[Page 199]{MSH02}, and therefore \cite[Condition 4.3]{MSH02} is satisfied. The result then follows from \cite[Theorem 4.4]{MSH02}.
\end{proof}

\begin{proof}[Proof of Lemma \ref{lem:FPMarginal}] Integrating equation \eqref{eq:FPsystem} with respect to $z$ we obtain the stationary Fokker--Planck equation for the process $X^\epl$, i.e.
	\begin{equation}
	\label{FPx}
	\sigma (\phi^\epl)''(x) + \frac{\d}{\d x} \left( \left ( \alpha \cdot V'(x) + \frac{1}{\epl} p' \left ( \frac{x}\epl \right ) \right ) \phi^\epl(x)\right) = 0,
	\end{equation}
	whose solution is given by
	\begin{equation}
	\phi^\epl(x) = \frac{1}{C_{\phi^\epl}} \exp\left(- \frac{1}{\sigma} \alpha \cdot V(x) - \frac{1}{\sigma} p \left ( \frac{x}\epl \right )\right),
	\end{equation}
	and which proves \eqref{eq:marginalX}. In view of \eqref{eq:densityDecomposition} and \eqref{FPx}, equation \eqref{eq:FPsystem} can be rewritten as
	\begin{equation}\label{eq:psiEqualityPDE}
	\partial_x \left( \sigma \phi^\epl \psi^\epl \partial_x R^\epl \right) + \partial_z \left( \frac1\delta(z-x) \phi^\epl \psi^\epl R^\epl \right) = 0.
	\end{equation}
	We now multiply the equation above by a continuous differentiable function $f \colon \R^2 \to \R^N$, $f = f(x,z)$, and integrate with respect to $x$ and $z$. Then an integration by parts yields
	\begin{equation}
	\sigma \int_\R \int_\R \partial_x f(x,z) \phi^\epl(x) \psi^\epl(z) \partial_x R^\epl(x,z) \dd x \dd z = \frac1\delta \int_\R \int_\R \partial_z f(x,z) (x-z) \phi^\epl(x) \psi^\epl(z) R^\epl(x,z) \dd x \dd z,
	\end{equation}
	which implies the following identity in $\R^N$
	\begin{equation}
	\sigma \delta \int_\R \int_\R \partial_x f(x,z) \phi^\epl(x) \psi^\epl(z) \partial_x R^\epl(x,z) \dd x \dd z = \E^{\rho^\epl} \left[ \partial_z f(X^\epl,Z^\epl) (X^\epl-Z^\epl) \right].
	\end{equation}
	Finally, choosing 
	\begin{equation}
	f(x,z) = (x - z) V'(z) + V(z),
	\end{equation}
	we obtain the desired result.
\end{proof}

\section{Proof of Proposition \ref{prop:zeta}}\label{ap:ProofDistance}

\subsection{Preliminary estimates}
In order to prove the characterization provided by Proposition \ref{prop:zeta}, we need to prove two additional results on the filter. First, we prove a Jensen-like inequality for the kernel of the filter.

\begin{lemma}\label{lem:FilterJensen} Let $\delta > 0$ and $k(r)$ be defined as
	\begin{equation}
		k(r) = \frac1\delta e^{-r/\delta}.
	\end{equation}
	Then, for any $t > 0$, $p \geq 1$ and any function $g\in \mathcal C^0([0, t])$ it holds 
		\begin{equation}
			\abs{\int_0^t k(t-s) g(s) \dd s}^p \leq \int_0^t k(t-s) \abs{g(s)}^p \dd s.
		\end{equation}
\end{lemma}
\begin{proof} Let us first note that
	\begin{equation}
		\int_0^t k(t-s) \dd s = 1 - e^{-t/\delta}.
	\end{equation}
	Therefore, the measure $\kappa_t(\d s)$ on $[0, t]$ defined as
	\begin{equation}
		\kappa_t(\d s) \defeq \frac{k(t-s)}{1 - e^{-t/\delta}} \dd s,
	\end{equation}
	is a probability measure. An application of Jensen's inequality therefore yields
	\begin{equation}
	\begin{aligned}
		\abs{\int_0^t k(t-s) g(s) \dd s}^p &\leq (1 - e^{-t/\delta})^p \int_0^t \abs{g(s)}^p \kappa_t(\d s) \\
		&= (1 - e^{-t/\delta})^{p-1} \int_0^t k(t-s) \abs{g(s)}^p \dd s.
	\end{aligned}
	\end{equation}
	Finally since $0 < (1 - e^{-t/\delta})< 1$ and $p \geq 1$, this yields the desired result.
\end{proof}

The following lemma characterizes the action of the filter when it is applied to polynomials in $(t-s)$.
\begin{lemma}\label{lem:FilterPoly} With the notation of Lemma \ref{lem:FilterJensen}, it holds for all $p \geq 0$
	\begin{equation}
		\int_0^t k(t-s)(t-s)^p \dd s \leq C \delta^p,
	\end{equation}
	where $C > 0$ is a positive constant independent of $\delta$.
\end{lemma}
\begin{proof} The change of variable $u = (t-s)/\delta$ yields
	\begin{equation}
		\int_0^t k(t-s)(t-s)^p \dd s = \delta^p \int_0^{t/\delta} u^p e^{-u} \dd u = \delta^p \gamma\left(p+1, \frac{t}{\delta}\right),
	\end{equation}
	where $\gamma$ is the lower incomplete Gamma function, which is bounded by the complete Gamma function $\Gamma(p+1)$ independently of the second argument.
\end{proof}

\subsection{Proof of Proposition \ref{prop:zeta}}
Denoting $Y_t^\epl \defeq X_t^\epl/\epl$, we will make use of the decomposition
\cite[Formula 5.8]{PaS07}
\begin{equation}\label{eq:App_DiffXtXs}
\begin{aligned}
X_t^\epl - X_s^\epl &= -\int_s^t (\alpha \cdot V'(X_r^\epl))(1 + \Phi'(Y^\epl_r)) \dd r\\
&\quad +\sqrt{2\sigma} \int_s^t (1 + \Phi'(Y_r^\epl)) \dd W_r - \epl(\Phi(Y^\epl_t)-\Phi(Y^\epl_s)),
\end{aligned}
\end{equation}
which is obtained applying the Itô formula to $\Phi$, the solution of the cell problem \eqref{eq:CellProblem}. Recall that by definition of $Z_t^\epl$ we have
	\begin{equation}
	X_t^\epl - Z_t^\epl = \int_0^t k(t-s) (X_t^\epl - X_s^\epl) \dd s + e^{-t/\delta}X_t^\epl.
	\end{equation}
	Plugging the decomposition \eqref{eq:App_DiffXtXs} into the equation above, we obtain
	\begin{equation}
	X_t^\epl - Z_t^\epl = I_1^\epl(t) + I_2^\epl(t) + I_3^\epl(t) + I_4^\epl(t),
	\end{equation}
	where
	\begin{equation}
	\begin{aligned}
	I_1^\epl(t) &\defeq -\int_0^t k(t-s)\int_s^t (\alpha \cdot V'(X_r^\epl))(1 + \Phi'(Y^\epl_r)) \dd r \dd s,\\
	I_2^\epl(t) &\defeq \sqrt{2\sigma}\int_0^t k(t-s) \int_s^t (1 + \Phi'(Y_r^\epl)) \dd W_r \dd s, \\
	I_3^\epl(t) &\defeq - \epl \int_0^t k(t-s)(\Phi(Y^\epl_t)-\Phi(Y^\epl_s)) \dd s,\\
	I_4^\epl(t) &= e^{-t / \delta}X_t^\epl.
	\end{aligned}
	\end{equation}
	Let us analyze the terms above singularly. For $I_1^\epl(t)$, one can show \cite[Proposition 5.8]{PaS07}
	\begin{equation}
	\int_s^t (\alpha \cdot V'(X_r^\epl)) (1 + \Phi'(Y^\epl_r)) \dd r = (t - s)(A \cdot V'(X_t^\epl)) + R_1^\epl(t-s),
	\end{equation}
	where the remainder $R_1^\epl$ satisfies 
	\begin{equation}\label{eq:FirstTermPavStu}
	\left(\E^{\phi^\epl}\abs{R_1^\epl(t-s)}^p\right)^{1/p} \leq C(\epl^2 + \epl(t-s)^{1/2}+(t-s)^{3/2}).
	\end{equation}
	Therefore, it holds 
	\begin{equation}\label{eq:Diff_I1}
	\begin{aligned}
	I_1^\epl(t) &= -(A \cdot V'(X_t^\epl)) \int_0^t k(t-s) (t-s)\dd s + \int_0^t k(t-s)R_1^\epl(t-s) \dd s\\
	&= -\delta (A \cdot  V'(X_t^\epl)) + e^{-t/\delta}(t + \delta)(A \cdot V'(X_t^\epl)) + \widetilde R_1^\epl(t),
	\end{aligned}
	\end{equation}
	where we exploited the equality
	\begin{equation}
	\int_0^t k(t-s) (t-s) \dd s = \delta - e^{-t/\delta}(t + \delta),
	\end{equation}
	and where
	\begin{equation}
	\widetilde R_1^\epl(t) \defeq \int_0^t k(t-s)R_1^\epl(t-s) \dd s.
	\end{equation}
	Now, Lemma \ref{lem:FilterJensen}, the inequality \eqref{eq:FirstTermPavStu} and Lemma \ref{lem:FilterPoly} yield for all $p \geq 1$
	\begin{equation}
	\begin{aligned}
		\E^{\phi^\epl}\abs{\widetilde R_1^\epl(t)}^p &\leq C \int_0^t k(t-s)\E^{\phi^\epl} \abs{R_1^\epl(t-s)}^p \dd s\\
		&\leq C \int_0^t k(t-s)(\epl^{2p} + \epl^p(t-s)^{p/2}+(t-s)^{3p/2}) \dd s\\
		&\leq C \left( \epl^{2p} + \epl^p \delta^{p/2} + \delta^{3p/2} \right),
	\end{aligned}
	\end{equation}
	where $C$ is a positive constant independent of $\epl$ and $\delta$. Therefore, for $\delta$ sufficiently small, we get
	\begin{equation}
		\left(\E^{\phi^\epl}\abs{I_1^\epl(t)}^{p}\right)^{1/p} \leq C \left( \delta + \epl^2 + \epl \delta^{1/2} + t e^{-t/\delta} \right).
	\end{equation}
	We now consider the second term. Let us introduce the notation  
	\begin{equation}
		Q_t^\epl \defeq \int_0^t \left(1 + \Phi'(Y_r^\epl)\right) \dd W_r,
	\end{equation}
	and therefore rewrite
	\begin{equation}
		I_2^\epl(t) = \sqrt{2\sigma} \int_0^t k(t-s)(Q_t^\epl - Q_s^\epl) \dd s.
	\end{equation}
	An application of the Itô formula to $u(s, Q_s^\epl)$ where $u(s, x) = k(t-s)x$ yields
	\begin{equation}\label{eq:rewriteI2}
	\begin{aligned}
		I_2^\epl(t) &= \sqrt{2\sigma} \left(Q_t^\epl\int_0^t k(t-s) \dd s - Q_t^\epl + \delta \int_0^t k(t-s)\left(1 + \Phi'(Y_s^\epl)\right) \dd W_s\right)\\
		  			&= \delta B_t^\epl - \sqrt{2\sigma} e^{-t/\delta} Q_t^\epl \eqdef \delta B_t^\epl - R_2^\epl(t).
	\end{aligned}
	\end{equation}
	where $B_t^\epl$ is defined in \eqref{eq:defBeta}. For the remainder $R_2^\epl(t)$, let us remark that for all $p \geq 1$ it holds
	\begin{equation}\label{eq:BoundQ}
	\begin{aligned}
		\left(\E\abs{Q_t^\epl}^p\right)^2 \leq \E\abs{Q_t^\epl}^{2p} \leq Ct^{p-1}\int_0^t \E \abs{1 + \Phi'(Y_r^\epl)}^{2p} \dd r \leq C t^p
	\end{aligned}
	\end{equation}
	where we applied Jensen's inequality, an estimate for the moments of stochastic integrals \cite[Formula (3.25), p. 163]{KaS91} and the boundedness of $\Phi$. Therefore we have
	\begin{equation}\label{eq:boundR2}
		\left( \E^{\phi^\epl}\abs{R_2^\epl(t)}^p \right)^{1/p} \leq C \sqrt{t} e^{-t/\delta}.
	\end{equation}
	In order to obtain the bound \eqref{eq:estimate_beta} on $B_t^\epl$, let us remark that from \eqref{eq:rewriteI2} it holds for a constant $C > 0$ depending only on $p$ 
	\begin{equation}
		\left(\E\abs{B_t^\epl}^p\right)^{1/p}\leq C \delta^{-1}\left(\E\abs{I_2^\epl(t)}^p\right)^{1/p} + C \delta^{-1} \left(\E\abs{R_2^\epl(t)}^p\right)^{1/p}.
	\end{equation}
	The second term is bounded exponentially fast with respect to $t$ and $\delta$ due to \eqref{eq:boundR2}. For the first term, applying Lemma \ref{lem:FilterJensen}, the inequality \cite[Formula (3.25), p. 163]{KaS91} and Lemma \ref{lem:FilterPoly} we obtain for a constant $C > 0$ independent of $\delta$ and $t$
	\begin{equation}
	\begin{aligned}
		\E\abs{I_2^\epl(t)}^p &\leq C \int_0^t k(t-s) \E\abs{Q_t - Q_s}^p \dd s \\
		&\leq C\int_0^t k(t-s)(t-s)^{p/2} \dd s \leq C\delta^{p/2}.
	\end{aligned}
	\end{equation} 
	Therefore, it holds for $\delta$ sufficiently small
	\begin{equation}
		\left(\E\abs{B_t^\epl}^p\right)^{1/p}\leq C \delta^{-1/2},
	\end{equation}
	which proves the bound \eqref{eq:estimate_beta}. Let us now consider $I_3^\epl(t)$. Since $\Phi$ is bounded, we simply have
	\begin{equation}
	\abs{I_3^\epl(t)} \leq C \epl, 
	\end{equation}
	almost surely. Finally, due to \cite[Corollary 5.4]{PaS07}, we know that $X_t^\epl$ has bounded moments of all orders and therefore
	\begin{equation}
	\left(\E^{\phi^\epl} \abs{I_4^\epl(t)}^p\right)^{1/p} \leq C e^{-t/\delta},
	\end{equation}
	which concludes the proof. \qed

\section{Proofs of Section \ref{sec:Fast}} \label{ap:ProofsDeltaZeta}

\subsection{Preliminary estimates}

The following lemma shows that $Z^\epl$ has bounded moments of all orders.
\begin{lemma} \label{lem:bounded_momentZ}
	Under Assumption \ref{as:regularity}, let $Z^\epl$ be distributed as the invariant measure $\mu^\epl$ of the couple $(X^\epl, Z^\epl)^\top$. Then for any $p \geq 1$ there exists a constant $C > 0$ uniform in $\epl$ such that 
	\begin{equation}
	\E^{\rho^\epl}|Z^\epl|^p \le C.
	\end{equation}
\end{lemma}
\begin{proof} Let $X_t^\epl$ be at stationarity with respect to its invariant measure, which we recall having density denoted as $\phi^\epl$. Let $Z_t^\epl$ be the corresponding filtered process. By definition of $Z^\epl_t$ and applying Lemma \ref{lem:FilterJensen} we have
	\begin{equation}
	\begin{aligned}
	\E^{\phi^\epl} |Z_t^\epl|^p &= \E^{\phi^\epl} \left| \int_0^t k(t-s) X_s^\epl \dd s \right|^p \\
	&\leq \int_0^t k(t - s) \E^{\phi^\epl} |X_s^\epl|^p \dd s,
	\end{aligned}
	\end{equation}
	which, together with the definition of $k$ and the fact that $X_s^\epl$ has bounded moments of all orders \cite[Corollary 5.4]{PaS07}, implies for a constant $C>0$
	\begin{equation}
	\E^{\phi^\epl}|Z_t^\epl|^p \le C.
	\end{equation}
	In order to conclude, we remark that due to Lemma \ref{lem:ergodicity} we have for all $t \geq 0$
	\begin{equation}
	\E^{\rho^\epl} \abs{Z^\epl}^p \leq \E^{\phi^\epl}\abs{Z_t^\epl}^p + C e^{-\lambda t},
	\end{equation}
	which, for $t$ sufficiently big, yields the desired result.
\end{proof}

Corollary \ref{cor:distanceZandX} is a direct consequence of Proposition \ref{prop:zeta} and provides a rough estimate of the difference between the trajectories $X_t^\epl$ and $Z_t^\epl$ when they are at stationarity.

\begin{corollary} \label{cor:distanceZandX} Under Assumption \ref{as:regularity}, let the couple $(X^\epl, Z^\epl)^\top$ be distributed as its invariant measure $\mu^\epl$. Then, if $\delta \leq 1$, it holds for any $p \geq 1$
	\begin{equation}
	\left(\E^{\rho^\epl} \abs{X^\epl - Z^\epl}^p\right)^{1/p} \leq C\left(\epl + \delta^{1/2}\right),
	\end{equation}
	for a constant $C > 0$ independent of $\epl$ and $\delta$.
\end{corollary}
\begin{proof}
	Let $p \geq 1$, then due to Proposition \ref{prop:zeta} there exists a constant $C > 0$ depending only on $p$ such that
	\begin{equation}
	\E^{\phi^\epl}\abs{X_t^\epl - Z_t^\epl}^p \leq C\left(\epl^p + \delta^{p/2}\right).
	\end{equation}
	Let us now remark that this result holds for $X_t^\epl$ being at stationarity and for $Z_t^\epl$ being its filtered process, and not for a couple $(X^\epl, Z^\epl)^\top \sim \mu^\epl$. In order to conclude, we remark that due to Lemma \ref{lem:ergodicity} we have for all $t \geq 0$
	\begin{equation}
	\E^{\rho^\epl} \abs{X^\epl - Z^\epl}^p \leq \E^{\phi^\epl}\abs{X_t^\epl - Z_t^\epl}^p + Ce^{-\lambda t},
	\end{equation}
	which, for $t$ sufficiently big, yields the desired result.
\end{proof}

The result above can be in some sense rather counter-intuitive. Indeed, for a fixed $\epl > 0$ and for $\delta \to 0$ independently of $\epl$, one expects the filtered trajectory $Z^\epl$ to approach $X^\epl$. This is provided by the following Lemma.
\begin{lemma}\label{lem:distanceZandX2} Under Assumption \ref{as:regularity}, let the couple $(X^\epl, Z^\epl)^\top$ be distributed as its invariant measure $\mu^\epl$. Then, if $\delta \leq 1$, it holds for any $p \geq 1$
	\begin{equation}
	\left(\E^{\rho^\epl} \abs{X^\epl - Z^\epl}^p\right)^{1/p} \leq C\left(\delta\epl^{-1} + \delta^{1/2}\right),
	\end{equation}
	for a constant $C > 0$ independent of $\epl$ and $\delta$.	
\end{lemma}
\begin{proof}
	By equation \eqref{eq:SDE_MS} we have for all $0 \le s < t$
	\begin{equation}
	X_t^\epl - X_s^\epl = -\alpha \int_s^t V'(X_r^\epl) \dd r - \frac1\epl \int_s^t p' \left( \frac{X_r^\epl}{\epl} \right) \dd r + \sqrt{2\sigma}(W_t - W_s).
	\end{equation}
	Therefore, by Assumption \ref{as:regularity} and since $X_t^\epl$ has bounded moments of all orders at stationarity \cite[Corollary 5.4]{PaS07}, it holds for any $p \ge 1$ and a constant $C > 0$
	\begin{equation} \label{eq:Xt_Xs}
	\E^{\phi^\epl} \abs{X_t^\epl - X_s^\epl}^p \le C \left( (t-s)^p + (t-s)^p\epl^{-p} + (t-s)^{p/2} \right),
	\end{equation}
	where $\phi^\epl$ is the invariant measure of $X^\epl$. By definition of $Z_t^\epl$ we have
	\begin{equation}
	X_t^\epl - Z_t^\epl = \int_0^t k(t-s) (X_t^\epl - X_s^\epl) \dd s + e^{-t/\delta}X_t^\epl,
	\end{equation}
	which, applying Lemma \ref{lem:FilterJensen}, the inequality \eqref{eq:Xt_Xs} and Lemma \ref{lem:FilterPoly}, implies
	\begin{equation}
	\begin{aligned}
	\E^{\phi^\epl} \abs{X_t^\epl - Z_t^\epl}^p &\le C \left( \int_0^t k(t-s) \E^{\phi^\epl} \abs{X_t^\epl - X_s^\epl}^p \dd s + e^{-pt/\delta}  \E^{\phi^\epl} \abs{X_t^\epl}^p \right) \\
	&\le C \left( \delta^p + \delta^p\epl^{-p} + \delta^{p/2} + e^{-pt/\delta} \right).
	\end{aligned}
	\end{equation}
	Geometric ergodicity (Lemma \ref{lem:ergodicity}) then implies for $\rho^\epl$ the measure of the couple $(X^\epl, Z^\epl)^\top$
	\begin{equation}
	\E^{\rho^\epl} \abs{X^\epl - Z^\epl}^p \leq \E^{\phi^\epl}\abs{X_t^\epl - Z_t^\epl}^p + Ce^{-\lambda t},
	\end{equation}
	which, for $t$ sufficiently big and since $\delta \leq 1$ yields the desired result.
\end{proof}

Let us conclude with a last preliminary estimate concerning the matrices $\widetilde{\mathcal M}_\epl$ and $\mathcal M_\epl$ defined in \eqref{eq:DefCalMTilde} and \eqref{eq:DefCalM}, respectively.

\begin{lemma}\label{lem:distanceMandTildeM} Let the assumptions of Corollary \ref{cor:distanceZandX} hold. Then the matrices $\mathcal M_\epl$ and $\widetilde{\mathcal M}_\epl$ satisfy
	\begin{equation}
	\norm{\mathcal M_\epl - \widetilde{\mathcal M}_\epl}_2 \leq C\left(\epl + \delta^{1/2}\right),
	\end{equation}
	for a constant $C > 0$ independent of $\epl$ and $\delta$.
\end{lemma}
\begin{proof} Applying Jensen's and Cauchy--Schwarz inequalities we have
	\begin{equation}
	\begin{aligned}
	\norm{\mathcal M_\epl - \widetilde{\mathcal M}_\epl}_2 &\leq \E^{\rho^\epl} \norm{\left(V'(Z^\epl) - V'(X^\epl)\right) \otimes V'(X^\epl)}_2 \\
	&\leq \left(\E^{\rho^\epl}\norm{V'(Z^\epl) - V'(X^\epl)}_2^2\right)^{1/2} \left(\E^{\rho^\epl} \norm{V'(X^\epl)}_2^{2}\right)^{1/2}.
	\end{aligned}
	\end{equation}
	The Lipschitz condition on $V'$ together with the boundedness of the moments of $X^\epl$ and Corollary \ref{cor:distanceZandX} yield for a constant $C > 0$
	\begin{equation}
	\begin{aligned}
	\norm{\mathcal M_\epl - \widetilde{\mathcal M}_\epl}_2 &\leq C\left(\E^{\rho^\epl}\abs{Z^\epl - X^\epl}^{2}\right)^{1/2} \leq C \left(\epl + \delta^{1/2} \right),
	\end{aligned}
	\end{equation}
	which is the desired result.
\end{proof}

\subsection{Proof of Lemma \ref{lem:quadruple}}

	Let us consider the following system of stochastic differential equations for the processes $X_t^\epl, Z_t^\epl, B_t^\epl, Y_t^\epl$
	\begin{equation}
	\label{eq:systemSDEquadruple}
	\begin{aligned}
	\d X_t^\epl &= -\alpha \cdot V'(X_t^\epl) \dd t - \frac1\epl p'(Y_t^\epl) \dd t+ \sqrt{2\sigma} \dd W_t, \\
	\d Z_t^\epl &= \frac1\delta \left ( X^\epl_t - Z^\epl_t \right ) \dd t, \\
	\d B_t^\epl &= - \frac1\delta B_t^\epl \dd t + \frac{\sqrt{2\sigma}}{\delta}(1+\Phi'(Y_t^\epl)) \dd W_t, \\
	\d Y_t^\epl &= -\frac1\epl \alpha \cdot V'(X_t^\epl) \dd t - \frac{1}{\epl^2} p'(Y_t^\epl) \dd t+ \frac{\sqrt{2\sigma}}{\epl} \dd W_t,
	\end{aligned}
	\end{equation}
	whose generator $\widetilde{\mathcal L}_\epl$ is given by
	\begin{equation}
	\begin{aligned}
	\widetilde{\mathcal L}_\epl =& -\left( \alpha \cdot V'(x) + \frac1\epl p'(y) \right) \partial_x + \frac1\delta (x-z) \partial_z - \frac1\delta b \partial_b - \left( \frac1\epl \alpha \cdot V'(x) + \frac{1}{\epl^2} p'(y) \right) \partial_y \\
	& + \sigma \left( \partial_{xx}^2 + \frac2\epl \partial_{xy}^2 + \frac{1}{\epl^2} \partial_{yy}^2 + \frac{2(1+\Phi'(y))}{\delta} \partial_{xb}^2 + \frac{2(1+\Phi'(y))}{\epl \delta} \partial_{yb}^2 + \frac{(1+\Phi'(y))^2}{\delta^2} \partial_{bb}^2 \right).
	\end{aligned}
	\end{equation}
	Let us denote by $\eta^\epl \colon \R^3 \times [0,L] \to \R$, $\eta^\epl = \eta^\epl(x,z,b,y)$, the invariant measure of the quadruple $(X_t^\epl, Z_t^\epl, B_t^\epl, Y_t^\epl)$. Then $\eta^\epl$ solves the stationary Fokker-Planck equation $\widetilde{\mathcal L}_\epl^* \eta^\epl = 0$, i.e., explicitly
	\begin{equation} \label{eq:FPSDEquadruple}
	\begin{aligned}
	&\partial_x \left( \left( \alpha \cdot V'(x) + \frac1\epl p'(y) \right) \eta^\epl \right) + \frac1\delta\partial_z \left( (z-x) \eta^\epl \right) \\
	&\quad + \frac1\delta \partial_b (b\eta^\epl)  + \partial_y \left( \left( \frac1\epl \alpha \cdot V'(x) + \frac{1}{\epl^2} p'(y) \right) \eta^\epl \right) + \sigma \left( \partial_{xx}^2 \eta^\epl + \frac2\epl \partial_{xy}^2 \eta^\epl + \frac1{\epl^2} \partial_{yy}^2 \eta^\epl \right) \\
	&\quad + \sigma \left( \frac2\delta \partial_{xb}^2 \left( (1+\Phi'(y)) \eta^\epl \right) + \frac2{\epl\delta}\partial_{yb}^2 \left( (1+\Phi'(y))\eta^\epl \right) + \frac1{\delta^2}\partial_{bb}^2 \left( (1+\Phi'(y))^2 \eta^\epl \right) \right) = 0.
	\end{aligned}
	\end{equation}
	We now multiply the equation above by a continuous differentiable function $f \colon \R^2 \to \R^N$, $f = f(z,b)$, and integrate with respect to $x$, $z$, $b$ and $y$. Then an integration by parts yields
	\begin{equation}
	\frac1\delta \int_{\R^3 \times [0,L]} \partial_z f(z,b) (x-z) \eta^\epl - \frac1\delta \int_{\R^3 \times [0,L]} \partial_b f(z,b) b \eta^\epl + \frac{\sigma}{\delta^2} \int_{\R^3 \times [0,L]} \partial_{bb}^2 f(z,b) (1+\Phi'(y))^2 \eta^\epl,
	\end{equation}
	which implies the following identity in $\R^N$
	\begin{equation} \label{eq:super_magic_formula}
	\delta \E^{\eta^\epl}\left[ \partial_b f(Z^\epl, B^\epl) B^\epl \right] = \sigma \E^{\eta^\epl} \left[ \partial_{bb}^2 f(Z^\epl, B^\epl) (1+\Phi'(Y^\epl)) \right] + \delta \E^{\eta^\epl} \left[ \partial_z f(Z^\epl, B^\epl) (X^\epl-Z^\epl) \right].
	\end{equation}
	Choosing
	\begin{equation}
	f(z,b) = \frac12 b^2 V''(z),
	\end{equation}
	we obtain
	\begin{equation}
	\begin{aligned}
	\delta \E^{\eta^\epl}\left[ (B^\epl)^2 V''(Z^\epl) \right] &= \sigma \E^{\eta^\epl} \left[ V''(Z^\epl) (1+\Phi'(Y^\epl)) \right] + \frac\delta2 \E^{\eta^\epl} \left[ (B^\epl)^2 V'''(Z^\epl) (X^\epl-Z^\epl) \right] \\
	&\eqdef \sigma \E^{\eta^\epl} \left[ V''(Z^\epl) (1+\Phi'(Y^\epl)) \right] + \widetilde R(\epl,\delta).
	\end{aligned}
	\end{equation}
	We now consider the remainder and, applying Hölder's inequality, Corollary \ref{cor:distanceZandX}, Lemma \ref{lem:bounded_momentZ}, Assumption \ref{as:regularityZeta} and \eqref{eq:estimate_beta}, we get for $p,q,r$ such that $1/p+1/q+1/r=1$
	\begin{equation}
	\left| \widetilde R(\epl,\delta) \right| \le C \delta \left( \E^{\eta^\epl} |B^\epl|^{2p} \right)^{1/p} \left( \E^{\eta^\epl} |V'''(Z^\epl)|^q \right)^{1/q} \left( \E^{\eta^\epl} |X^\epl - Z^\epl|^r \right)^{1/r} \le C (\delta^{1/2} + \epl),
	\end{equation}
	which completes the proof. \qed 

\subsection{Proof of Lemma \ref{lem:quadruple_convergence}}

	Let us introduce the notation 
	\begin{equation}
	\Delta(\epl) = \left| \sigma \E^{\eta^\epl} [V''(Z^\epl) (1 + \Phi'(Y^\epl))^2] - \Sigma \E^{\phi^0} [V''(X)] \right|,
	\end{equation}
	and note that the aim is to show that $\lim_{\epl \to 0} \Delta(\epl) = 0$. By the triangle inequality we get
	\begin{equation}
	\begin{aligned}
	\Delta(\epl) \le& \left| \sigma \E^{\eta^\epl} [V''(Z^\epl) (1 + \Phi'(Y^\epl))^2] - \sigma \E^{\eta^\epl} [V''(X^\epl) (1 + \Phi'(Y^\epl))^2] \right| \\
	&+ \left| \sigma \E^{\eta^\epl} [V''(X^\epl) (1 + \Phi'(Y^\epl))^2] - \Sigma \E^{\phi^0} [V''(X)] \right| \\
	\eqdef& \Delta_1(\epl) + \Delta_2(\epl).
	\end{aligned}
	\end{equation}
	We first study $\Delta_1(\epl)$ and due to the boundedness of $\Phi'$, Assumption \ref{as:regularityZeta} and Lemma \ref{cor:distanceZandX} we have
	\begin{equation}
	\Delta_1(\epl) \le C \E^{\eta^\epl} \abs{X^\epl - Z^\epl} \le C (\delta^{1/2} + \epl) = C (\epl^{\zeta/2} + \epl),
	\end{equation}
	which implies
	\begin{equation}
	\lim_{\epl \to 0} \Delta_1(\epl) = 0.
	\end{equation}
	We now consider $\Delta_2(\epl)$. Integrating equation \eqref{eq:FPSDEquadruple} with respect to $z$ and $b$ we obtain the Fokker-Planck equation for the stationary marginal distribution $\lambda \colon \R \times [0,L]$, $\lambda = \lambda(x,y)$, of the couple $(X^\epl,Y^\epl)$
	\begin{equation}
	\begin{aligned}
	\partial_x \left( \left( \alpha \cdot V'(x) + \frac1\epl p'(y) \right) \lambda \right) + \partial_y \left( \left( \frac1\epl \alpha \cdot V'(x) + \frac{1}{\epl^2} p'(y) \right) \lambda \right) & \\
	+ \sigma \left( \partial_{xx}^2 \lambda + \partial_{xy}^2 \left( \frac2\epl \lambda \right) + \partial_{yy}^2 \left( \frac{1}{\epl^2} \lambda \right) \right) &= 0,
	\end{aligned}
	\end{equation}
	whose solution is given by
	\begin{equation}
	\lambda(x,y) = \frac{1}{C_\lambda} \exp \left( - \frac\alpha\sigma V(x) - \frac1\sigma p(y) \right),
	\end{equation}
	where
	\begin{equation}
	\begin{aligned}
	C_{\lambda} &= \int_\R \int_0^L \exp \left( - \frac\alpha\sigma V(x) - \frac1\sigma p(y) \right) \dd x \dd y \\
	&= \left( \int_\R \exp \left( - \frac\alpha\sigma V(x) \right) \dd x \right) \left( \int_0^L \exp \left( - \frac1\sigma p(y) \right) \dd y \right) \\
	&\eqdef C_{\lambda_x} C_{\lambda_y}.
	\end{aligned}
	\end{equation}
	Therefore, since $\Sigma = K\sigma$ and by equations \eqref{eq:K_HOM} and \eqref{eq:phi0} we have
	\begin{equation}
	\begin{aligned}
	\sigma \E^{\eta^\epl} [V''(X^\epl) (1 + \Phi'(Y^\epl))^2] &= \sigma \int_\R \int_0^L V''(x) (1 + \Phi'(y))^2 \frac{1}{C_\lambda} \exp \left( - \frac\alpha\sigma V(x) - \frac1\sigma p(y) \right) \dd x \dd y \\
	&= \sigma \left( \int_\R V''(x) \frac{1}{C_{\lambda_x}} \exp \left( - \frac\alpha\sigma V(x) \right) \dd x \right) \\
	&\qquad \qquad \times \left( \int_0^L (1+\Phi'(y))^2 \frac{1}{C_{\lambda_y}} \exp \left( - \frac1\sigma p(y) \right) \dd y \right) \\
	&= \sigma K \E^{\phi^0}[V''(X)] = \Sigma \E^{\phi^0}[V''(X)],
	\end{aligned}
	\end{equation}
	which shows that $\Delta_2(\epl) = 0$ and completes the proof. \qed 

\subsection{Proof of Theorem \ref{thm:mainTheorem_zetaAlpha}}
	Let us consider the decomposition \eqref{eq:alphaDecomposition}, i.e.,
	\begin{equation}
	\widehat A_k(X^\epl,T) = \alpha + I_1^\epl(T) - I_2^\epl(T),
	\end{equation}
	where $I_1^\epl(T)$ is defined in \eqref{eq:alphaDecomposition} and satisfies
	\begin{equation}
		\lim_{T \to \infty} I_1^\epl(T) = \widetilde {\mathcal M}_\epl^{-1} \E^{\rho^\epl} \left [ \frac{1}{\epl} p' \left ( \frac{X^\epl}{\epl} \right ) V'(Z^\epl) \right ], \quad \text{a.s.}
	\end{equation}
	and, by the proof of Theorem \ref{thm:mainTheorem} we have independently of $\epl$
	\begin{equation}
		\lim_{T \to \infty} I_2^\epl(T) = 0, \quad \text{a.s.}
	\end{equation}
	A Taylor expansion of the first order of $V'$ yields
	\begin{equation}
	V'(Z^\epl) = V'(X^\epl) + V''(\widetilde X^\epl) (Z^\epl - X^\epl),
	\end{equation}
	where $\widetilde X^\epl$ is a random variable which assumes values between $X^\epl$ and $Z^\epl$. We can therefore write
	
	\begin{equation}
	\begin{aligned}
	\lim_{T \to \infty} I_1^\epl(T) &= \widetilde {\mathcal M}_\epl^{-1} \left(\E^{\rho^\epl} \left [ \frac{1}{\epl} p' \left ( \frac{X^\epl}{\epl} \right ) V'(X^\epl) \right ] + \E^{\rho^\epl} \left [ \frac{1}{\epl} p' \left ( \frac{X^\epl}{\epl} \right ) V''(\widetilde X^\epl) (Z^\epl - X^\epl) \right ] \right)\\
	&\eqdef \widetilde {\mathcal M}_\epl^{-1}\left(J_1^\epl + J_2^\epl\right).
	\end{aligned}
	\end{equation}
	We now consider the two terms separately and show they vanish for $\epl \to 0$. Integrating by parts in $J_1^\epl$ we obtain
	\begin{equation}
	\begin{aligned}
	J_1^\epl &= \int_\R \frac{1}{\epl} p' \left( \frac{x}{\epl} \right) V(x) \frac{1}{C_{\rho^\epl}} \exp \left( -\frac\alpha\sigma V(x) - \frac1\sigma p\left(\frac{x}{\epl}\right) \right) \dd x \\
	&= \int_\R \left ( \sigma V''(x) -  \left(V'(x) \otimes V'(x)\right) \alpha \right) \frac{1}{C_{\rho^\epl}} \exp \left( -\frac\alpha\sigma V(x) - \frac1\sigma p\left(\frac{x}{\epl}\right) \right) \dd x \\
	&= \sigma \E^{\rho^\epl} \left[ V''(X^\epl) \right] - \E^{\rho^\epl} \left[ V'(X^\epl) \otimes V'(X^\epl) \right] \alpha.
	\end{aligned}
	\end{equation}
	We then pass to the limit as $\epl \to 0$ and integrate by parts again to obtain
	\begin{equation} \label{eq:limJ1_0}
	\lim_{\epl\to 0} J_1^\epl = \sigma \E^{\rho^0} \left[ V''(X) \right] - \E^{\rho^0} \left[ V'(X) \otimes V'(X) \right] \alpha = 0.
	\end{equation}
	We now turn to $J_2^\epl$. The Hölder's inequality with conjugate exponents $p$ and $q$ and the assumptions on $p$ and $V$ yield
	\begin{equation}
	\abs{J_2^\epl} \le C \epl^{-1} \left( \E^{\rho^\epl} \abs{\widetilde X^\epl}^q \right)^{1/q} \left( \E^{\rho^\epl} \abs{Z^\epl - X^\epl}^p \right)^{1/p}.
	\end{equation}
	Since $\widetilde X^\epl$ assumes values between $X^\epl$ and $Z^\epl$, it has bounded moments by \cite[Corollary 5.4]{PaS07} and Lemma \ref{lem:bounded_momentZ}. Hence, applying Lemma \ref{lem:distanceZandX2} we have
	\begin{equation} 
	\abs{J_2^\epl} \le C \left( \delta \epl^{-2} + \delta^{1/2} \epl^{-1} \right),
	\end{equation}
	which, since $\delta = \epl^\zeta$ with $\zeta > 2$, implies
	\begin{equation}\label{eq:limJ2_0}
		\lim_{\epl \to 0} \abs{J_2^\epl} = 0.
	\end{equation}
	Finally, Lemma \ref{lem:distanceMandTildeM} and the weak convergence of the invariant measure $\phi^\epl$ to $\phi^0$ imply
	\begin{equation}
	\lim_{\epl \to 0} \widetilde{\mathcal M}_\epl = \mathcal M_0,
	\end{equation}
	which, together with \eqref{eq:limJ1_0}, \eqref{eq:limJ2_0} implies that $I_1^\epl(T) \to 0$ for $T \to \infty$ and $\epl \to 0$, which implies the desired result. \qed
	
\section{Proof of Theorem \ref{thm:diffusion_unbiasedness}} \label{ap:diffusion}

First, the ergodic theorem yields
	\begin{equation}
		\lim_{T \to \infty} \widehat \Sigma_k = \frac1\delta \E^{\rho^\epl} \left[ (X^\epl - Z^\epl)^2 \right],
	\end{equation}	
	then applying Proposition \ref{prop:zeta} at stationarity we obtain
	\begin{equation} 
	\begin{aligned}
		\lim_{T \to \infty} \widehat \Sigma_k &= \delta \E^{\rho^\epl} \left[ (B^\epl)^2 \right] + 2\E^{\rho^\epl} \left[ B^\epl R(\epl,\delta) \right] + \frac1\delta \E^{\rho^\epl} \left[ R(\epl,\delta)^2 \right] \\
		&\eqdef I_1^\epl + I_2^\epl + I_3^\epl,
	\end{aligned}
	\end{equation}
	and due to the Cauchy-Schwarz inequality and estimates \eqref{eq:estimate_beta} and \eqref{eq:remainder_zeta} we have
	\begin{equation} \label{eq:bound_diffusion_unbiased}
		\abs{I_2^\epl} \leq C \left( \delta^{1/2} + \epl \delta^{-1/2} \right) \qquad \text{and} \qquad \abs{I_3^\epl} \leq C \left( \delta + \epl^2 \delta^{-1} \right),
	\end{equation}
	for a constant $C>0$ independent of $\epl$ and $\delta$. Let us now consider $I_1^\epl$. Employing equation \eqref{eq:super_magic_formula} with the function $f(z,b) = 1/2 b^2$ gives
	\begin{equation}
		\E^{\eta^\epl} \left[ (B^\epl)^2 \right] = \frac\sigma\delta \E^{\eta^\epl} \left[ 1 + \Phi'(Y^\epl) \right] = \frac{\sigma K}{\delta} = \frac\Sigma\delta,
	\end{equation}
	which together with bounds \eqref{eq:bound_diffusion_unbiased} and the hypothesis on $\delta$ implies
	\begin{equation}
		\lim_{\epl \to 0} \lim_{T \to \infty} \widehat \Sigma = \Sigma, \quad \text{in probability},
	\end{equation}
	which is the desired result.\qed
\end{appendices}

\def\cprime{$'$}

\end{document}